\documentclass[twoside,11pt]{article}
\usepackage{jmlr2e}
\usepackage[OT1]{fontenc}
\usepackage{graphicx}
\usepackage{epstopdf}
\usepackage{todonotes}
\usepackage{breakcites}
\usepackage{caption}

\usepackage{amsmath,amsfonts}
\usepackage{mathtools}
\usepackage{enumerate}
\usepackage{enumitem}
\usepackage{bbm}
\usepackage{subcaption}
\usepackage{upgreek}
\usepackage{xcolor}

\usepackage{algorithm}
\usepackage{algpseudocode}

\usepackage[export]{adjustbox}

\DeclareGraphicsExtensions{.eps,.pdf,.png,.jpg}

\newcommand{\bs}{\boldsymbol}

\newcommand{\bX}{\bs X}
\newcommand{\bY}{\bs Y}

\newcommand{\bU}{\bs U}

\newcommand{\bx}{\bs x}
\newcommand{\by}{\bs y}

\newcommand{\bu}{\bs u}

\newcommand{\mI}{\mathsf I}
\newcommand{\mC}{\mathsf C}
\newcommand{\mH}{\mathsf H}
\newcommand{\mF}{\mathsf F}
\newcommand{\mG}{\mathsf G}
\newcommand{\mL}{\mathsf L}
\newcommand{\mM}{\mathsf M}
\newcommand{\mJ}{\mathsf J}
\newcommand{\cF}{\mathcal F}
\newcommand{\cT}{\mathcal T}

\newcommand{\cS}{\mathcal S}

\newcommand{\bth}{\bs \theta}
\newcommand{\bTh}{\bs \Theta}
\newcommand{\bxi}{\bs \xi}
\newcommand{\bXi}{\bs \Xi}

\newcommand{\R}{\mathbb{R}}

\newcommand{\tr}[1]{f\left(\bx_{#1} | \bx_{#1\text{-}1},\bth\right)}
\newcommand{\lk}[1]{g\left(\by{}_{#1} | \bx_{#1},\bth\right)}

\newcommand{\dd}{\mathrm{d}}

\newcommand{\lerr}[1]{\| #1  \|_{L^2}} 
\newcommand{\dhd}[2]{D_{\rm H} (#1 , #2)} 
\newcommand{\df}[2]{D_{f} (#1 \| #2)} 

\newcommand{\tminusone}{t\text{-}1}
\newcommand{\jminusone}{j\text{-}1}

\newcommand{\refover}[2]{\overset{\hspace{-6pt}\eqref{#1}\hspace{-6pt}}{#2}}

\makeatletter
\def\mathcenterto#1#2{\mathclap{\phantom{#1}\mathclap{#2}}\phantom{#1}}
\let\old@widetilde\widetilde
\def\widetildeto#1#2{\mathcenterto{#2}{\old@widetilde{\mathcenterto{#1}{#2\,}}}}
\let\old@widehat\widehat
\def\widehatto#1#2{\mathcenterto{#2}{\old@widehat{\mathcenterto{#1}{#2\,}}}}
\makeatother
\renewcommand{\widetilde}[1]{\widetildeto{o}{#1}}
\renewcommand{\widehat}[1]{\widehatto{o}{#1}}

\def\app#1#2{%
  \mathrel{%
    \setbox0=\hbox{$#1\sim$}%
    \setbox2=\hbox{%
      \rlap{\hbox{$#1\propto$}}%
      \lower1.1\ht0\box0%
    }%
    \raise0.25\ht2\box2%
  }%
}
\def\approxprop{\mathpalette\app\relax}

\newcounter{algcounter}
\newenvironment{myalg}[1]
{\refstepcounter{algcounter}  
\medskip\hrule\medskip
{\noindent Algorithm~\thealgcounter: #1}\vspace{4pt}
}
{
\medskip\hrule\medskip
}

\DeclareRobustCommand{\svdots}{
  \vbox{%
    \baselineskip=0.2\normalbaselineskip
    \lineskiplimit=0pt
    \hbox{.}\hbox{.}\hbox{.}%
    \kern-0.2\baselineskip
  }%
}

\newtheorem{assumption}{Assumption}

\ShortHeadings{Tensor trains for sequential learning}{Zhao, Cui}
\firstpageno{1}{}

\begin{document}

\title{Tensor-train methods for sequential state and parameter learning in state-space models}

\author{\name Yiran Zhao \email{yiran.zhao@monash.edu} \\
\addr School of Mathematics, Monash University\\
Victoria 3800, Australia\\
\name Tiangang Cui \email{tiangang.cui@sydney.edu.au} \\
\addr School of Mathematics and Statistics, The University of Sydney\\
New South Wales 2006, Australia
}

\editor{}

\maketitle

\begin{abstract}

We consider sequential state and parameter learning in state-space models with intractable state transition and observation processes. By exploiting low-rank tensor train (TT) decompositions, we propose new sequential learning methods for joint parameter and state estimation under the Bayesian framework. 
Our key innovation is the introduction of scalable function approximation tools such as TT for recursively learning the sequentially updated posterior distributions. 
The function approximation perspective of our methods offers tractable error analysis and potentially alleviates the particle degeneracy faced by many particle-based methods. In addition to the new insights into the algorithmic design, our methods complement conventional particle-based methods. 
Our TT-based approximations naturally define conditional Knothe--Rosenblatt (KR) rearrangements that lead to parameter estimation, filtering, smoothing and path estimation accompanying our sequential learning algorithms, which open the door to removing potential approximation bias. 
We also explore several preconditioning techniques based on either linear or nonlinear KR rearrangements to enhance the approximation power of TT for practical problems. 
We demonstrate the efficacy and efficiency of our proposed methods on several state-space models, in which our methods achieve state-of-the-art estimation accuracy and computational performance. 
\end{abstract}


\allowdisplaybreaks


\section{Introduction}

State-space models have been widely used in mathematical and statistical modelling to analyze time-varying complex phenomena \citep{cappe2009inference,kantas2009overview}. 
Examples include time series analysis in finance, temporal pattern recognition in bioinformatics, forecast in meteorology, and more. A typical state-space model is also referred to as a hidden Markov model, which consists of two stochastic processes $\{\bX_{t}\}_{t\geq 0}$ and $\{\bY_{t}\}_{t\geq 1}$. The state transition process $\{\bX_{t}\}_{t\geq 0}$ is an $\mathcal{X}$-valued latent Markov process specified by a conditional transition density 
\begin{equation}\label{eq:state_transition}
(\bX_t | \bX_{0:\tminusone} = \bx_{0:\tminusone})  \equiv (\bX_t | \bX_{\tminusone} = \bx_{\tminusone}) \sim f(\bx_t |\bx_{\tminusone}, \bth),
\end{equation}
and an initial density $p(\bx_0 |\bth)$. Here $\bX_t$ represents the hidden state of the underlying system at an integer-valued time index $t\geq 0$, and the notation $\bx_{i:j}$ denotes vectors $(\bx_i, \bx_{i+1}, \ldots, \bx_j)$ from a sequence $\{\bx_t\}_{t \geq 0}$. The $\mathcal{Y}$-valued observation process $\{ \bY_{t}\}_{t\geq 1}$ is characterized by the likelihood function $ \lk{t} $ as
\begin{equation}\label{eq:observation}
	(\bY_t | \bX_{0:t} =  \bx_{0:t}, \bY_{1:\tminusone} = \by{}_{1:\tminusone}) \equiv ( \bY_t | \bX_{t} = \bx_{t}) \sim g(\by{}_t |\bx_t, \bth), 
\end{equation}
where $\bY_t$ represents the observable at time index $t > 0$. In \eqref{eq:state_transition} and \eqref{eq:observation}, $\bth \in \Uptheta$ is a set of parameters that govern the state transition process and the observation process. 

\begin{example}{(Volatility of financial instruments)} \label{eg:sv}
The stochastic volatility model consists of  the return of an asset $\{\bY_t\}_{t \geq 1}$ and the logarithm of its squared volatility $\{\bX_t\}_{t \geq 0}$. 
In the simplest setup, the logarithm of the squared volatility $\{\bX_t\}_{t \geq 0}$ is a scalar-valued autoregressive process with order one, i.e., an AR$(1)$ process, while the scalar-valued return $\{\bY_t\}_{t \geq 1}$ is the observable determined by the volatility. 
We arrange the unknown parameters to appear in both the state transition process and the observation process. 
The system is described by
\begin{equation*}
	\left\{\begin{array}{rl}
		\bX_{t} \; = \! & \gamma \bX_{\tminusone}+ \sigma\varepsilon^{(x)}_{t} \vspace{3pt}\\
		\bY_{t} \; = \! & \varepsilon^{(y)}_{t}\beta\,\exp(\frac12\bX_t)
	\end{array}\right. ,
\end{equation*}
where $\bTh=(\gamma, \sigma, \beta)$ are model parameters, $\varepsilon^{(x)}_{t}$ and $\varepsilon^{(y)}_{t}$ are independent and identically distributed (i.i.d.) standard Gaussian random variables, and the initial state is given as
$
\bX_{0}\sim \mathcal{N}(0,\frac{\sigma^2}{1-\gamma^{2}}).
$
The goal is to estimate the parameter $\bTh$ and the hidden states $\{\bX_t\}_{t \geq 0}$ from the observed returns $\{\bY_t\}_{t \geq 1}$.
\end{example}

\subsection{Sequential learning problems}

Following the state-space model defined in \eqref{eq:state_transition} and \eqref{eq:observation}, random variables $(\bTh, \bX_{0:t}, \bY_{1:t})$ have the joint density
\begin{equation} 
	p(\bth, \bx_{0:t}, \by{}_{1:t}) = p(\bth)\, p(\bx_0 |\bth)\prod_{j=1}^{t} \Big( f(\bx_j |\bx_{\jminusone}, \bth) \, g(\by{}_j |\bx_j, \bth) \Big)
	\label{eq:joint_pdf},
\end{equation}
where $ p(\bth) $ and $ p(\bx_0 |\bth) $ are prescribed prior densities.
Conditioned on all available data $\bY_{1:t} = \by_{1:t}$ at time $t$, the trajectory of the states $\bX_{0:t}$ and the parameter $\bTh$ jointly follow the posterior density
\begin{equation} 
	p(\bth, \bx_{0:t} | \by{}_{1:t}) = \frac{p(\bth,\bx_{0:t}, \by{}_{1:t})}{p(\by{}_{1:t})} ,\label{eq:joint_post_pdf}
\end{equation}
where $p(\by{}_{1:t})$ is an unknown constant commonly referred to as the evidence. We aim to design sequential algorithms that  simultaneously solve the following inference problems at each time $t$: 
\begin{itemize}[leftmargin=18pt]
\item \textbf{Filtering} that infers the current state, which is the marginal conditional random variable
\begin{equation} \label{eq:filtering_prob}
	(\bX_t | \bY_{1:t} = \by{}_{1:t}) \sim p(\bx_{t} | \by{}_{1:t}) := \int p(\bth,\bx_{0:t} | \by{}_{1:t}) \dd \bx_{0:\tminusone} \dd \bth.
\end{equation} 
\item \textbf{Parameter estimation} that infers the unknown parameter, which is the marginal conditional random variable
\begin{equation} \label{eq:param_prob}
	(\bTh | \bY_{1:t} = \by{}_{1:t}) \sim p(\bth | \by{}_{1:t}) := \int p(\bth,\bx_{0:t} | \by{}_{1:t}) \dd \bx_{0:t}.
\end{equation} 
\item \textbf{Path estimation} that infers the trajectory of the states, which are the marginal conditional random variables 
\begin{equation} \label{eq:path_prob}
	(\bX_{0:t} | \bY_{1:t} = \by{}_{1:t}) \sim p(\bx_{0:t} | \by{}_{1:t}) := \int p(\bth,\bx_{0:t} | \by{}_{1:t}) \dd \bth.
\end{equation} 
\item \textbf{Smoothing.} As a result of increasing state dimensions and the observation size, path estimation becomes increasingly challenging over time. Smoothing is a related learning problem focusing on the lower-dimensional marginal conditional random variable
\begin{equation} \label{eq:smoothing_prob}
	(\bX_k | \bY_{1:t} = \by{}_{1:t}) \sim p(\bx_k | \by{}_{1:t}) := \int p(\bth,\bx_{0:t} | \by{}_{1:t}) \dd \bx_{0:k\text{-}1} \dd \bx_{k+1:t} \dd \bth
\end{equation} 
for some previous time index $k < t$.
\end{itemize}
In the presence of unknown parameters and stochastic noises in the state transition density \eqref{eq:state_transition}, all the above-mentioned learning problems are essentially characterizations of marginal random variables. 
Furthermore, the random variables of interest are intractable, in the sense that the characteristic functions and moments often do not admit analytical forms, and it is infeasible to directly simulate these random variables. 
Thus, we need to design numerical methods to solve these problems. In principle, algorithms for solving these problems need to be \emph{sequential} by design---the computation of new solutions to the problems \eqref{eq:filtering_prob}--\eqref{eq:smoothing_prob} at time $t$ only relies on either previous solutions at $t-1$ for forward algorithms, or solutions at $t+1$ for backward algorithms. We refer readers to \cite{evensen2022data,kantas2009overview,reich2015probabilistic} and \cite{sarkka2013bayesian} for further details and references of these sequential learning problems. 

\subsection{Related work and our contributions}

Assuming the parameter $\bth$ is known, the classical filtering problem aims to estimate the distribution of the current state conditioned on the data available up to $t$, i.e., $(\bX_t | \bTh = \bth, \bY_{1:t} = \by{}_{1:t})$. For linear models with Gaussian noises, the filtering density yields a closed-form solution given by the classical Kalman filter \citep{kalman1960new}. The extended Kalman filter \citep{anderson2012optimal, einicke1999robust} and the ensemble Kalman filter \citep{evensen2003ensemble} generalize the Kalman filter to nonlinear problems. These generalizations utilize linearization and Monte Carlo sampling to create Gaussian approximations and sample-based approximations of posterior densities. As a result, these methods are flexible to implement and often robust with respect to dimensionality. However, the Gaussian assumption used by various Kalman filters introduces unavoidable approximation errors to the resulting filtering density, which makes them only suitable for tracking the states for nonlinear problems.

Beyond the Gaussian form, sequential Monte Carlo (SMC) methods recursively maintain a set of weighted particles to fully characterize the filtering density \eqref{eq:filtering_prob} for nonlinear problems. In the simplest form, the bootstrap filter \citep{gordon1993novel} updates weighted particles using the state transition process and reweighs them using the likelihood function with newly collected data. It is well-known that the weights of the bootstrap filter tend to degenerate over time, a phenomenon commonly referred to as the particle degeneracy \citep{doucet2009tutorial, snyder2008obstacles,liu2001theoretical, pitt1999filtering}. To mitigate the particle degeneracy, the auxiliary particle filter \citep{pitt1999filtering, pitt2001auxiliary} and various advanced resampling techniques such as residual resampling \citep{carpenter1999improved}, systematic resampling \citep{kitagawa1996monte}, and resample-move \citep{gilks2001following} have been developed. See \cite{del2012adaptive,doucet2009tutorial,reich2015probabilistic} and \cite{maskell2002tutorial} for comprehensive reviews of these techniques.

An emerging trend is the development of transport map methods that overcome the particle degeneracy by transforming the forecast particles into equally weighted particles following the filtering density. The ensemble transform filter \citep{reich2013nonparametric} builds such transformations via a discrete optimal transport problem. The coupling technique of \cite{spantini2022coupling} relaxes the Gaussian assumption of the ensemble Kalman filter by building conditional Knothe--Rosenblatt (KR) rearrangements \citep{knothe1957contributions,rosenblatt1952remarks}. The implicit sampling method \citep{chorin2009implicit,morzfeld2012random} couples a reference density, e.g., a standard Gaussian, to a particular approximation of the filtering density via implicit maps. Methods based on neural networks are also developed. For example, \cite{gottwald2021combining} build random feature maps on delayed coordinates to accelerate filtering, and \cite{hoang2021machine} approximate conditional mean filters using neural networks.

One of our innovations is treating the filtering problem and other aforementioned sequential learning problems as recursive function approximation problems. This perspective opens the door to applying scalable function approximation tools in sequential learning. In particular, we employ tensor-train (TT) decompositions \citep{hackbusch2012tensor, oseledets2010tt, oseledets2011tensor} to design non-parametric solution ansatzes for recursively approximating probability densities in sequential learning. As a result, our new algorithms significantly generalize the linear Gaussian assumptions used by various Kalman filters and bypass sample-based empirical approximations used in various particle filters. This not only makes our algorithms computationally efficient but also capable of fully quantifying uncertainties in sequential learning problems. The computed TT decompositions in our algorithms naturally lead to a sequence of KR rearrangements, enabling us to design an accompanying particle filter to correct for biases due to function approximation errors. Moreover, we utilize these KR rearrangements to develop a backward sampling algorithm, in line with many particle smoothing methods---e.g., \cite{bresler1986two, briers2010smoothing, fearnhead2010sequential, godsill2004monte}---to solve the path estimation and smoothing problems.

Another significant challenge in sequential learning problems is to estimate unknown parameters, which intrinsically requires marginalizing over the state path to evaluate the posterior parameter density \eqref{eq:param_prob}. The particle Markov chain Monte Carlo (MCMC) method \citep{andrieu2010particle} and the method developed in \cite{sarkka2015posterior} use inner-loop particle filters or Gaussian filters to estimate the computationally intractable marginalized posterior parameter density $p(\bth |\by{}_{1:t})$, and then build outer-loop Markov chain transition kernels to sample posterior parameters. These methods are designed for batched data and thus may have difficulties updating the solution in an online manner with new observations. To estimate posterior parameters online, the SMC$^2$ method \citep{chopin2013smc2} and its variants---e.g., the nested particle filter \citep{crisan2018nested}---run an outer-loop particle filter to sequentially update the parameter distribution. Since these methods are based on particle filters, they often need a restart step to handle the particle degeneracy. As a result, they need to sacrifice either the computational complexity (SMC$^2$ has a computational complexity quadratic in time) or the rate of convergence (the nested particle filter has a convergence rate of $N^{-1/4}$, where $N$ is the sample size) in online estimation. 

Our new algorithms are directly applicable to the online parameter estimation problem. In fact, our new algorithms solve all four aforementioned sequential learning problems in an integrated manner by recursively approximating the joint posterior density of parameters and states over time. The separability of TT decompositions naturally leads to marginal posterior densities for solving the parameter estimation and filtering problems. This also allows us to build accompanying KR rearrangements to correct approximation biases and solve the smoothing and path estimation problems. More importantly, we show that the accumulation of these density approximation errors has a linear rate under mild assumptions. Within the same framework, we also present preconditioning techniques to improve the approximation power of TT decompositions. On a range of numerical examples, we show that our method outperforms the SMC$^2$ method and can produce meaningful results in real-world applications.

\subsection{Outline} 

We propose a basic version of the TT-based algorithm that recursively approximates the time-varying joint posterior densities in Section \ref{sec:background}. The separable form of TT provides approximations to the marginal densities required in learning problems \eqref{eq:filtering_prob}--\eqref{eq:smoothing_prob}. The basic TT-based algorithm is useful to outline the design principles and key steps of our proposed methods. However, the rank truncation used by TT may not preserve the non-negativity of density functions. This leads to difficulties in removing the approximation bias using techniques such as importance sampling. We overcome this limitation in Section~\ref{sec:estimation} by integrating the square-root approximation technique of \cite{cui2021deep} into the basic algorithm of Section~\ref{sec:background}. This preserves the recursive learning pattern of the basic algorithm for problems \eqref{eq:filtering_prob}--\eqref{eq:smoothing_prob} and leads to non-negative approximations by construction. More importantly, the separable form of the non-negative TT naturally leads to a sequence of KR rearrangements that couple reference random variables and the marginal random variables of interest. We utilize the resulting (conditional) KR rearrangements to design particle filters and particle smoothers accompanying TT approximations to remove estimation biases. 

The smoothing procedure accompanying our TT-based learning algorithm shares many similarities with the joint KR rearrangements developed in \cite{spantini2018inference}. 
In particular, our joint KR rearrangement for the simultaneous path and parameter estimation (see Section~\ref{sec:smoothing}) follows the same sparsity pattern as that of \cite{spantini2018inference}, as a result of the hidden Markov structure of the learning problem. 
Our key innovation lies in algorithmic design. 
The work of \cite{spantini2018inference} adopts a \emph{variational} approach that builds an approximate map by minimizing the statistical divergence of the pushforward of some reference density under the candidate map from the target density. 
The training procedure of the variational approach is often quite involved---the objective function presents many local minima and each optimization iteration requires repeated evaluations of the target density at transformed reference variables under the candidate map. In comparison, our algorithmic design utilizes multilinear function approximations to directly approximate marginal and conditional densities of interest. In addition to the benefit of bypassing nonlinear optimization, we can exploit the TT-based approximations to compute the \emph{exact} KR rearrangement of the \emph{approximate} density in both lower-triangular and upper-triangular forms. This enables us to solve the filtering problem \eqref{eq:filtering_prob} and the smoothing problem \eqref{eq:smoothing_prob} using different conditional KR rearrangements constructed from the same approximate density. 

In Section \ref{sec:error}, we discuss the error accumulation rate of our proposed algorithms. In Section \ref{sec:precondition}, we present a set of preconditioning techniques for our sequential learning problems to further enhance the approximation power of TT. 
These preconditioning techniques are also given in the form of KR rearrangements. In Section \ref{sec:numerics}, we present a set of numerical experiments to demonstrate the efficacy and efficiency of our proposed methods and compare them to particle-based methods. Our code is accessible online via \url{https://github.com/DeepTransport/tensor-ssm-paper-demo}.

\subsection{Notation}

The dimensionalities of the state space $\mathcal{X}$, the data space $\mathcal{Y}$, and the parameter space $\Uptheta$ are $m$, $n$, and $d$, respectively. We assume that the spaces $\mathcal{X}$, $\mathcal{Y}$, and $\Uptheta$ can be expressed as Cartesian products. We denote the $j$-th element of a vector $\bx$ by $x_j$. The index $t \in \mathbb{N}$ is specifically reserved for denoting the time throughout the paper. For a state random vector $\bX_t$ and its realization $\bx_t$ at time $t$, their $j$-th elements are denoted by $X_{t,j}$ and $x_{t,j}$, respectively. The same convention applies to the data vector $\bY_t$ and its realization $\by{}_t$.

For a vector $\bx_t \in \R^m$ and an index $j$, it is convenient to group a subset of elements as follows. The vector $\bx_{t, < j} = (x_{t,1}, \ldots, x_{t,\jminusone})$ collects the first $j-1$ elements, and the vector $\bx_{t, >j} = (x_{t,j+1}, \ldots, x_{t,m})$ collects the last $m-j$ elements. Similarly, we have $\bx_{t, \leq j} = (x_{t,1}, \ldots, x_{t,j})$, $\bx_{t, \geq j} = (x_{t,j}, \ldots, x_{t,m})$, and $\bx_{t,\leq m} \equiv \bx_{t,\geq 0} \equiv \bx_t$.

We consider probability measures absolutely continuous with respect to the Lebesgue measure. We denote normalized posterior probability densities in the sequential learning problems by $p$ and its unnormalized version by $\pi$. Approximations to these densities are denoted by $\hat p$ and $\hat \pi$, respectively. The Hellinger distance between random variables with densities $p$ and $\hat p$ is defined by
\[
 D_{\rm H}(p, \hat p ) = \Big(\frac{1}{2} \int \big(\sqrt{p(\bx)}-\sqrt{\hat p(\bx)}\big)^2\dd \bx \Big)^\frac12.
\]
We denote the $m$-dimensional uniform random variable on a unit hypercube by $\bXi \sim \mathrm{uniform}(\bxi; [0,1]^m)$, in which we drop $[0,1]^m$ in the definition when no confusion arises. 

Consider a diffeomorphism $\mathcal{S}{\,:\,} \mathcal{X} {\,\rightarrow\,} \mathcal{U}$, where $\mathcal{X}, \mathcal{U} \subseteq \R^m$.
The density of the transformed variable $\mathcal{S}(\bX)$ where $ \bX\sim p,$ is the \emph{pushforward} of $p$ under  $\mathcal{S}$, which takes the form 
\[
\mathcal{S}_\sharp\, p(\bu) = p \big(\mathcal{S}^{-1}(\bu)\big) \, \big| \nabla \mathcal{S}^{-1}(\bu)\big|.
\]
Similarly, the density of the transformed variable $\mathcal{S}^{-1}(\bU)$, where $ \bU \sim \eta $, is the \emph{pullback} of the density of $\bU$ under $\mathcal{S}$, which takes the form  
\[
\mathcal{S}^\sharp\, \eta(\bx) = \eta \big(\mathcal{S} (\bx)\big) \, \big| \nabla \mathcal{S}(\bx)\big|.
\]
Here $|\cdot|$ denotes the absolute value of the determinant of a matrix.


\section{TT-based recursive posterior approximation}\label{sec:background}

We first discuss the recursive formula outlining the design principle of sequential learning problems. We then introduce the TT decomposition and present a basic implementation of the TT-based sequential learning algorithm. 

\subsection{Recursive state and parameter learning}\label{sec:rec}

Following the definition of the joint density of $(\bTh, \bX_{0:t}, \bY_{1:t})$ in \eqref{eq:joint_pdf},  the density of the posterior random variables $(\bTh, \bX_{0:t} | \bY_{1:t}= \by{}_{1:t})$ also has a recursive form
\begin{equation}
	p(\bth, \bx_{0:t} | \by{}_{1:t}) = \frac{p(\bth, \bx_{0:\tminusone} | \by{}_{1:\tminusone}) f(\bx_t |\bx_{\tminusone}, \bth) g(\by{}_t |\bx_t, \bth)}{{p(\by{}_t |\by{}_{1:\tminusone})}}
	\label{eq:joint_post},
\end{equation}
where $p(\by{}_t |\by{}_{1:\tminusone})$ is a computationally intractable conditional evidence. Marginalizing both sides over $\bX_{0:t\text{-}2}$, the parameter $\bTh$ and the adjacent states $(\bX_{t}, \bX_{\tminusone})$ jointly follow the posterior density
\begin{equation} 
	p(\bth, \bx_{t} , \bx_{\tminusone} | \by{}_{1:t}) = \frac{p(\bth, \bx_{\tminusone} | \by{}_{1:\tminusone}) f(\bx_t |\bx_{\tminusone}, \bth) g(\by{}_t |\bx_t, \bth)}{p(\by{}_t |\by{}_{1:\tminusone})}.
	\label{eq:rec_post}
\end{equation} 
The above formula outlines basic steps needed by a sequential estimation algorithm to solve the filtering problem \eqref{eq:filtering_prob} and the parameter estimation problem \eqref{eq:param_prob}:
\begin{enumerate}[leftmargin=18pt]
	\item At time $t{-}1$, the posterior random variables $(\bTh, \bX_{\tminusone}|\bY_{1:\tminusone} = \by{}_{1:\tminusone})$ has the density $p(\bth, \bx_{\tminusone} | \by{}_{1:\tminusone})$.
	\item At time $t$, using the state transition density $f(\bx_t |\bx_{\tminusone}, \bth)$ and the likelihood function $g(\by{}_t |\bx_t, \bth)$, we can compute the joint posterior density $p(\bth, \bx_t, \bx_{\tminusone} | \by{}_{1:t})$, up to the unknown conditional evidence $p(\by{}_t |\by{}_{1:\tminusone})$.
	\item We obtain the density of the new posterior random variables $(\bTh, \bX_{t}|\bY_{1:t} = \by{}_{1:t})$ by solving a marginalization problem
	\begin{equation}\label{eq:marginalization}
		p(\bth, \bx_{t} | \by{}_{1:t}) = \int p(\bth, \bx_{t} , \bx_{\tminusone} | \by{}_{1:t}) \mathrm{d} \bx_{\tminusone}.
	\end{equation}
\end{enumerate} 

In these steps, the marginalization in \eqref{eq:marginalization} plays a key role in sequentially updating the joint posterior random variables $(\bTh, \bX_{t}|\bY_{1:t} = \by{}_{1:t})$. For example, the bootstrap filter \citep{gordon1993novel} solves the marginalization using a weighted update of conditional samples---it first draws conditional random variables  $\bX_{t}^{(i)} |\bX_{\tminusone}^{(i)}$ that follow the state transition density $ f(\bx_t | \bX_{\tminusone}^{(i)})$ conditioned on each of the previous particles, and then updates the weights using the likelihood function $g(\by{}_t |\bX_t^{(i)})$. 
The path estimation \eqref{eq:path_prob} and the smoothing \eqref{eq:smoothing_prob} can be solved by additional backward propagation steps given the solutions to the filtering and parameter estimation problems. 
In the rest of this section, we will introduce the TT decomposition that solves the marginalization problem \eqref{eq:marginalization} by function approximation, and then present a basic TT-based algorithm to outline the procedure for solving the filtering and parameter estimation problems. 
Algorithms for solving the smoothing problem will be discussed in later sections.

\subsection{TT decomposition}

The central computational tool used in this work is the functional TT decomposition \citep{bigoni2016spectral,gorodetsky2019continuous}. Consider we have a general multivariate function $h: \mathcal{X} \rightarrow \R$ where $\mathcal{X}\in \R^m$ is a Cartesian product. Then, one can approximately decompose $h(\bx)$ in the following form
\[
h(\bx) \approx \hat h(\bx) = \sum_{\alpha_0 = 1}^{r_0} \sum_{\alpha_1 = 1}^{r_1} \cdots \sum_{\alpha_m = 1}^{r_m} \mH_{1}^{(\alpha_0, \alpha_1)}(x_1) \cdots \mH_{k}^{(\alpha_{k\text{-}1}, \alpha_{k})}(x_k) \cdots \mH_{m}^{(\alpha_{m\text{-}1}, \alpha_m)}(x_m),
\]
where $r_0 = r_m = 1$ and the summation ranges $r_0, r_1, \ldots, r_m$ are called TT ranks. 
Each scalar-valued univariate function $\mH_{k}^{(\alpha_{k\text{-}1}, \alpha_{k})}(x_k)$ is represented as a linear combination of a set of $\ell_k$ basis functions $\{\phi_{k}^{(1)}(x_{k}), \ldots, \phi_{k}^{(\ell_k)}(x_{k})\}$, which yields
\[
	\mH_{k}^{(\alpha_{k\text{-}1},\alpha_{k})}(x_{k})=\sum_{j=1}^{\ell_{k}}\phi_{k}^{(j)}(x_{k}) {\mathbf A}_k[\alpha_{k\text{-}1},j,\alpha_{k}],
\]
where ${\mathbf A}_k \in \R^{r_{k\text{-}1} \times \ell_k \times r_k}$ is an order-$3$ coefficient tensor. Examples of the basis functions include piecewise polynomials, orthogonal functions, radial basis functions, etc. 
Grouping all scalar-valued univariate functions $\mH_{k}^{(\alpha_{k\text{-}1}, \alpha_{k})}(x_k)$ for each coordinate $x_k$ yields a matrix-valued function $\mH_{k}(x_k): \mathcal{X}_{k} \rightarrow \R^{r_{k\text{-}1} \times r_k}$, which is referred to as the $k$-th tensor core. This way, the decomposed function can also be expressed as a sequence of multiplications of matrix-valued univariate functions, which is given by
\begin{equation}\label{eq:tt}
	\hat h(\bx) = \mH_{1}(x_1) \cdots \mH_{k}(x_k) \cdots \mH_{m}(x_m).
\end{equation}

The TT decomposition can be efficiently computed via alternating linear schemes together with cross interpolation \citep{bigoni2016spectral,gorodetsky2019continuous,oseledets2010tt}.
We employ the functional extension of the alternating minimal energy method with a residual-based rank adaptation of \cite{dolgov2014alternating}. 
It requires only $\mathcal O(m \ell r^2)$ evaluations of the function $f$ and $\mathcal O(m \ell r^3)$ floating point operations, where $\ell = \max_k \ell_k$ and  $r = \max_k r_k$. 
In general, the maximal rank $r$ depends on the dimension $m$ and can be large when the function $h$ concentrates in some part of its domain. 
Some theoretical results exist that provide rank bounds. For example, the work of \cite{rohrbach2022rank} establishes specific bounds for certain multivariate Gaussian densities that depend poly-logarithmically on $m$, while the work of \cite{griebel2021analysis} proves dimension-independent bounds for general functions in weighted spaces with dominating mixed smoothness.

The integral of the factorized function $\hat h$ can be simplified to the integral of certain univariate tensor cores. 
For example, the integral over $x_{k}$ can be calculated as
\begin{equation}
	\label{tt_int}
	\begin{aligned}
	\int \hat h (x_{1},x_{2},\ldots,x_{m}) \dd x_{k} & = \int \mH_{1}(x_{1}) \cdots \mH_{k}(x_{k}) \cdots \mH_{m}(x_{m}) \dd x_{k}\\
	& = \quad\, \mH_{1}(x_{1}) \cdots \overline\mH_{k} \cdots \mH_{m}(x_{m}), 
	\end{aligned}
\end{equation}
where the overlined matrix $\overline\mH_{k} = \int \mH_{k}(x_{k}) \dd x_k \in \R^{r_{k\text{-}1}\times r_k}$ is obtained by integrating the $k$-th tensor core elementwisely. 
This way, the cost of the integration problem scales linearly in the number of variables and quadratically in tensor ranks. This opens the door to solving the marginalization step in sequential learning problems. 

\subsection{Basic algorithm}\label{sec:basic_alg}

We integrate the TT decomposition and the recursive formula in Section \ref{sec:rec} to design a basic TT-based algorithm for solving the sequential learning problems. Although this basic algorithm does not require a particular variable ordering, we order the variables as $(\bx_t, \bth, \bx_{\tminusone})$ to be compatible with algorithms that will be introduced in Section \ref{sec:stt_recursion}.

At time $t{-}1$, we suppose the density of the posterior variables $(\bTh, \bX_{\tminusone} | \bY_{1:\tminusone} = \by{}_{1:\tminusone})$ is approximated by a TT decomposition, i.e., 
\[
	p(\bth, \bx_{\tminusone} | \by{}_{1:\tminusone}) \approxprop \hat \pi(\bth, \bx_{\tminusone} | \by{}_{1:\tminusone}).
\]
Here $ h_1(\bx) \approxprop h_2(\bx)$ denotes that $h_1(\bx)$ is approximately proportional to $ h_2(\bx)$, i.e., they are close to each other after normalization:
\[
	\frac{h_1(\bx)}{\int h_1(\bx) \dd \bx} \approx	\frac{h_2(\bx)}{\int h_2(\bx) \dd \bx}.
\]
Then, for the new observed data $\by{}_t$, we can recursively approximate the new density of the joint posterior random variables $(\bTh, \bX_{t} | \bY_{1:t} = \by{}_{1:t})$, the filtering density of  $(\bX_{t} | \bY_{1:t} = \by{}_{1:t})$ and the posterior parameter density  $(\bTh | \bY_{1:t} = \by{}_{1:t})$ as follows.

\begin{myalg}{TT-based sequential estimation.}\label{alg:basic} 
\begin{enumerate}[nosep,leftmargin=18pt]
\item[(a)] {\bf Non-separable approximation.} Following \eqref{eq:rec_post}, the density of the joint posterior random variables $(\bX_t, \bTh, \bX_{\tminusone} | \bY_{1:t} = \by{}_{1:t})$ yields a non-separable, unnormalized approximation $q_t$ in the form of 
\begin{align} \label{eq:tt_0}
	q_t(\bx_t, \bth, \bx_{\tminusone}) := \hat \pi(\bx_{\tminusone}, \bth | \by{}_{1:\tminusone}) f(\bx_t |\bx_{\tminusone}, \bth) g(\by{}_t |\bx_t, \bth),
\end{align}
which can be evaluated pointwisely. 

\item[(b)] {\bf Separable approximation.} We re-approximate the non-separable unnormalized density $q_t$ by a TT decomposition $\hat \pi$, i.e.,
\begin{align*}
	q_t(\bx_t, \bth, \bx_{\tminusone}) & \approx \hat \pi (\bx_t, \bth, \bx_{\tminusone} | \by{}_{1:t}) \nonumber \\
	& = \mF_{1}(x_{t,1}) \cdots \mF_{m}(x_{t,m}) \mG_{1}(\theta_1) \cdots \mG_{d}(\theta_d) \mH_{1}(x_{\tminusone,1}) \cdots \mH_{m}(x_{\tminusone,m}),
\end{align*}
where $\mF$, $\mG$ and $\mH$ denote tensor cores for the state at time $t$, the parameters, and the state at time  $t{-}1$, respectively. 

\item[(c)] {\bf Integration.} By integrating the TT-approximation $\hat \pi$, we are able to approximate the density of the posterior random variables $(\bTh, \bX_t | \bY_{1:t} = \by{}_{1:t})$ by 
\begin{align*}
	\hat \pi(\bx_t , \bth| \by{}_{1:t}) & = \int \hat \pi(\bx_t, \bth, \bx_{\tminusone} | \by{}_{1:t}) \dd \bx_{\tminusone} \nonumber \\
	& = \quad\, \mF_{1}(x_{t,1}) \cdots \mF_{m}(x_{t,m}) \mG_{1}(\theta_1) \cdots \mG_{d}(\theta_d) ( \overline \mH_1 \cdots \overline \mH_m ),
\end{align*}
and the normalizing constant 
\[
	c_t = \int 	\hat \pi(\bx_t, \bth, \bx_{\tminusone} | \by{}_{1:t})  \dd \bx_t\dd \bth \dd \bx_{\tminusone} = \overline\mF_{1} \cdots \overline\mF_{m}\overline\mG_{1} \cdots \overline\mG_{d} \overline \mH_1 \cdots \overline \mH_m.
\]
The densities of the posterior parameters $(\bTh|\bY_{1:t} = \by{}_{1:t})$ and the filtering state $(\bX_t | \bY_{1:t} = \by{}_{1:t})$ can be approximated in a similar way.\vspace{4pt}
\end{enumerate}
In the next step $t+1$, we can apply the same procedure using the newly computed approximation $\hat \pi(\bx_t, \bth | \by{}_{1:t})$. 
\end{myalg}

The above basic TT-based algorithm reveals some key principles of the algorithms introduced in this paper. Given the existing approximation $\hat \pi(\bx_{\tminusone}, \bth | \by{}_{1:\tminusone})$, the joint posterior random variables $(\bX_t, \bTh, \bX_{t{-}1} | \bY_{1:t} = \by{}_{1:t})$ yield a non-separable, unnormalized approximate density $q_t$ in step (a) of Alg.~\ref{alg:basic}. Since $q_t$ cannot be directly marginalized, we re-approximate $q_t$ using a new TT in step (b) of Alg.~\ref{alg:basic}, which enables further marginalizations in step (c) of Alg.~\ref{alg:basic}. 

In Section \ref{sec:estimation}, we design new algorithms to remove estimation biases due to approximation errors.
In Alg.~\ref{alg:basic}, the rank truncation used by TT may not preserve the non-negativity of density functions. An analogue is that the truncated singular value decomposition of a non-negative matrix may contain negative entries. The non-negativity is essential for designing transport maps for debiasing. We will overcome this barrier by using an alternative form of the TT decomposition in Section~\ref{sec:estimation}.
Then in Section~\ref{sec:error}, we analyze the accumulation of approximation errors over time. 
Since the complexity and the approximation power of TT-based algorithms crucially rely on the tensor ranks, we develop preconditioning methods in Section~\ref{sec:precondition} to improve the efficiency of our proposed algorithms. 


\section{Squared-TT algorithms, debiasing and smoothing}\label{sec:estimation}

We first review the squared-TT method for building KR rearrangements, which is originally introduced in \cite{cui2021deep}. 
We then integrate the resulting transport maps into the recursive procedure defined in Section \ref{sec:basic_alg} to sequentially solve the filtering problem \eqref{eq:filtering_prob} and the parameter estimation problem \eqref{eq:param_prob} with sample-based debiasing. 
Furthermore, we also design an algorithm to solve the path estimation problem \eqref{eq:path_prob} and the smoothing problem \eqref{eq:smoothing_prob} with additional backward propagation steps. 

\subsection{Squared TT and KR rearrangement}\label{sec:stt}
Consider the normalized target probability density $p(\bx) = \frac1z \pi(\bx)$ with $\bx \in \R^m$, in which $z$ is an unknown constant, and we can only evaluate the unnormalized density $\pi(\bx)$. To preserve the non-negativity in function approximation, one can decompose the square root of $\pi(\bx)$, i.e., 
\[
\surd\pi(\bx) \approx \phi(\bx) = \mH_{1}(x_{1}) \cdots \mH_{m}(x_{m}).
\]
Then, the approximate density function 
\(
\phi(\bx)^2
\)
is non-negative for all $\bx$ by construction. 
Given some reference tensor-product probability density $\lambda(\bx) := \prod_{i=1}^{m} \lambda_i(x_i)$ such that $\sup_{\bx} \pi(\bx) / \lambda(\bx) < \infty$ and a sufficiently small constant $\tau > 0$, we can further construct a defensive version of the approximate density function
\begin{equation} \label{eq:sirt_approx}
	\hat p(\bx) = \frac{1}{\hat z} \hat \pi(\bx), \quad \hat \pi(\bx) = \phi(\bx)^2 + \tau \lambda(\bx), \quad \hat z = \int \hat \pi(\bx) \dd \bx.
\end{equation}
The approximate density $\hat p$ with the defensive term $ \lambda(\bx) $ satisfies 
\[
	\sup_{\bx} \frac{p(\bx)}{\hat p(\bx)} = \sup_{\bx} \frac{\hat z p(\bx)}{\phi(\bx)^2 + \tau \lambda(\bx)} < \frac{\hat z}{\tau z} \sup_{\bx} \frac{\pi(\bx)}{\lambda(\bx)} < \infty.
\]

We introduce the defensive term $ \tau \lambda(\bx) $ to ensure that the target density $ p(\bx) $ is absolutely continuous with respect to the approximate density $ \hat p(\bx) $. This way, when the approximation $ \hat p(\bx) $ is used as the proposal density in importance sampling, e.g., in the case of correcting the bias of our recursive posterior approximation scheme, the resulting estimators can fulfil the requirement of the central limit theorem. Under mild assumptions, the following lemma establishes the $L^2$ error of $\surd \hat \pi$ and the Hellinger distance between the normalized density $ p$ and its approximation $ \hat p $.

\begin{lemma}\label{lemma:sirt_error}
Suppose the TT approximation $\phi$ satisfies	
\(
\lerr{\phi - \surd \pi} \leq \epsilon 
\) 
and the constant $ \tau $ satisfies 
\(
\tau \leq \lerr{\phi - \surd \pi}^2.
\)
Then, the $ L^2 $ error of $\surd \hat \pi$ defined in \eqref{eq:sirt_approx} satisfies 
\(
\lerr{\surd \hat \pi - \surd \pi} \leq \surd 2 \epsilon.
\) 
The Hellinger distance between $ p $ and its normalized approximation $\hat p$ defined in \eqref{eq:sirt_approx} satisfies 
\[
	\dhd{\hat p}{p} \leq \frac{\surd 2} {\surd{z}} \lerr{\surd \hat \pi - \surd \pi} \leq \frac{2\epsilon} {\surd{z}}.
\]
\end{lemma}
\begin{proof}
The bounds on the $L^2$ error and the Hellinger distance are shown in  Proposition 4 and Theorem 1 of~\cite{cui2021deep}, respectively.
\end{proof}

As outlined in Section \ref{sec:basic_alg}, marginalization of the approximation is a key operation in deriving the recursive algorithm. For the squared TT approximation defined in \eqref{eq:sirt_approx}, the following \cite[Proposition 2]{cui2021deep} gives its marginal density.

\begin{proposition}\label{prop:marginal}
We consider the normalized approximation  $\hat p$ defined in \eqref{eq:sirt_approx}. For the TT decomposition $\phi(\bx) = \mH_{1}(x_{1}) \cdots \mH_{m}(x_{m})$ and a given index $1 \leq k < m$, we define the left accumulated tensor core as 
\(
	\mH_{\leq k} (\bx_{\leq k}) = \mH_{1}(x_1) \cdots \mH_{k}(x_k) 	\in \R^{1 \times r_k}
\)
and the right accumulated tensor core as
\(
	\mH_{> k} (\bx_{> k}) = \mH_{k+1}(x_{k+1}) \cdots \mH_{m}(x_m) 	\in \R^{r_k\times 1}.
\)
Then the marginal density of $\hat p$ takes the form
\begin{equation} \label{eq:sirt_marginal}
	\hat p(\bx_{\leq k}) = \int \hat p(\bx) \dd \bx_{> k } = \frac1{\hat z}\Big( \sum_{\gamma_{k} = 1}^{r_{k}}  \Big(\sum_{\alpha_{k} = 1}^{r_{k}} \mH_{\leq k}^{(\alpha_{k})} (\bx_{\leq k}) \mL_{> k}^{(\alpha_k, \gamma_{k})} \Big)^2 + \tau \lambda(\bx_{\leq k}) \Big),
\end{equation}
where $\lambda(\bx_{\leq k}) = \prod_{i = 1}^k \lambda_i(x_i)$ and $\mL_{> k} \in \R^{r_k \times r_k}$ is the (lower-triangular) Cholesky decomposition of the accumulated mass matrix
\[
\mM_{> k}^{(\alpha_k, \beta_{k})} = \int \mH_{> k}^{(\alpha_{k})} (\bx_{> k}) \mH_{> k}^{(\beta_{k})} (\bx_{> k}) \dd \bx_{> k }, \quad \alpha_k = 1, \ldots, r_k,  \quad \beta_{k} = 1, \ldots, r_k.
\]
\end{proposition}

As shown in \cite{cui2021deep}, the Cholesky decomposition in Proposition \ref{prop:marginal} can be recursively computed from $k = m-1$ to $k=1$. 
The normalizing constant $\hat z$ can also be computed in a similar way as the last iteration of the recursion. 
Computing all the marginal densities $\hat p(x_1), \ldots, \hat p(\bx_{\leq m\text{-}1})$ requires $\mathcal{O}(m \ell r^3) $ floating point operations (flops).

The normalized approximate density $\hat p$ defines a new random variable $\widehat{\bX}$. 
Then using the marginal densities constructed in Proposition \ref{prop:marginal}, the densities of random variables $\widehat{X}_1, \widehat{X}_2|\widehat{X}_1$, $\ldots, \widehat{X}_k |\widehat{\bX}_{<k}, \ldots,\widehat{X}_m |\widehat{\bX}_{<m}$ are given by $\hat p(x_1)$ and
\[
\hat p(x_k |\bx_{< k}) = 	\frac{\hat p(\bx_{\leq k})}{\hat p(\bx_{<k})} = \frac{\phi(\bx_{\leq k})^2 + \tau \lambda(\bx_{\leq k})}{\phi(\bx_{< k})^2 + \tau \lambda(\bx_{< k})}, \quad k = 2, 3, \ldots, m,
\]
respectively. The corresponding distribution functions
\begin{equation}\label{eq:marginal_cdf}
	F_1(x_1 ) = \int_{-\infty}^{x_1} \hat p(x_1^\prime) \dd x_1^\prime \quad \text{and} \quad F_k(x_k |\bx_{<k}) = \int_{-\infty}^{x_k} 	\hat p(x_k^\prime |\bx_{<k})  \dd x_k^\prime
\end{equation}
define an order-preserving map $\mathcal{F}: \R^m \mapsto [0,1]^m$ in the form of 
\begin{equation}\label{eq:KR}
	 \mathcal{F}(\bx) = \big[F_1(x_1), \ldots, F_k(x_k |\bx_{<k}), \ldots, F_m(x_m |\bx_{<m})\big]^\top ,
\end{equation} 
which is referred to as the KR rearrangement \citep{rosenblatt1952remarks,knothe1957contributions}. The map $\mathcal{F}$ transforms a random variable $\widehat{\bX} \sim \hat p(\bx)$ into a uniform random variable $\bXi \sim \mathrm{uniform}(\bxi)$, i.e., $\bXi = \mathcal{F}(\widehat{\bX})$.
Since the $k$-th component of $\mathcal{F} $ depends on only the previous $k-1$ coordinates, both $\mathcal{F}$ and its inverse $\mathcal{F}^{-1}$ have triangular structures, and thus can be evaluated dimension-by-dimension. 

The construction of the marginal densities in \eqref{eq:sirt_marginal}, and hence the conditional densities $\{\hat p(x_k |\bx_{<k})\}_{k = 1}^m$, requires $\mathcal{O}(m \ell r^3) $ flops. These conditional densities can be constructed before evaluating the transformations $\mathcal{F}$ and $\mathcal{F}^{-1}$. 
Evaluating the KR rearrangement $\mathcal{F}$ consists of the dimension-by-dimension evaluation of the conditional densities $\hat p(x_k |\bx_{<k}) $ and the computation of the corresponding distribution functions $\{F_k\}_{k=1}^m$. For each sample, the former is essentially a sequence of vector-matrix products with a total cost of $ \mathcal O(m\ell r^2) $ flops, while the latter can be computed algebraically exactly with a total cost of $ \mathcal O(m \ell (\log\ell + r) + m\ell) $ flops, where the $ \mathcal O(m\ell(\log\ell + r)) $ term is for applying pseudo-spectral methods to construct the distribution function and the $ \mathcal O(m\ell) $ term is for evaluating the resulting distribution function. This way, evaluating $\mathcal{F}$ for $N$ samples costs $ \mathcal O(m \ell r^3 + Nm\ell r^2 + N m \ell (\log\ell+r) + N m \ell) $. For evaluating the inverse transformation $\mathcal{F}^{-1}$, we need to numerically solve a sequence of root finding problems for inverting each distribution function $F_k$. Since $F_k$ is strictly increasing and bounded by construction, applying either the regula falsi method or Newton's method for a fixed number of iterations (denoted by $c$) can accurately invert $F_k$ up to machine precision, where $c$ is often less than 10. This way, evaluating the inverse transform $\mathcal{F}^{-1}$ for $N$ samples costs $ \mathcal O(m \ell r^3 + Nm\ell r^2 + N m \ell (\log\ell+r) + N m \ell c)$ flops. Therefore, for cases where the sample size $N$ is large and the TT rank $r$ is sufficiently high, the evaluations of $\mathcal{F}$ and $\mathcal{F}^{-1}$ cost $\mathcal O( N m\ell r^2)$ flops.

\begin{remark}\label{rmk:uptri_map}
The marginal densities $\hat p(\bx_{\geq k}) $ where $ k = 2, \ldots,m$, can be computed in a similar way from the first coordinate to the last coordinate. 
This way, the resulting KR rearrangement is upper-triangular, i.e.,
\begin{equation}
	\mathcal{F}^{u}(\bx) = \big[F_1(x_1 | \bx_{>1}), \ldots, F_k(x_k |\bx_{>k}), \ldots, F_m(x_m)\big]^\top , 
\end{equation}
as the variable dependency follows a reverse order. 
\end{remark}

\subsection{Non-negativity-preserving algorithm and conditional maps}\label{sec:stt_recursion}

Following the steps of the basic algorithm (Alg.~\ref{alg:basic}), we first define a new sequential learning algorithm using the non-negativity-preserving approximation presented in \eqref{eq:sirt_approx}, and then construct the associated conditional KR rearrangements. 
The conditional KR rearrangements will be used to sequentially generate weighted samples to correct for the approximation bias.  

\begin{myalg}{Sequential estimation using squared-TT approximations.}\label{alg:stt}
\begin{enumerate}[leftmargin=18pt]
\item[(a)] {\bf Non-separable approximation.} 
At time $t{-}1$, suppose the posterior random variables $(\bTh, \bX_{\tminusone} | \bY_{1:\tminusone})$ has a TT-based (unnormalized) approximate density 
\(
	\hat \pi(\bth, \bx_{\tminusone} | \by{}_{1:\tminusone}).
\)
The density of the new joint posterior random variables $(\bX_t,\bTh, \bX_{\tminusone} |\bY_{1:t} = \by{}_{1:t})$ yields a non-separable, unnormalized approximation
\begin{align} 
	q_t(\bx_t, \bth, \bx_{\tminusone}) := \hat \pi( \bx_{\tminusone}, \bth | \by{}_{1:\tminusone}) f(\bx_t |\bx_{\tminusone}, \bth) g(\by{}_t |\bx_t, \bth).  \label{eq:sqrt_0}
\end{align}

\item[(b)] {\bf Separable approximation.} We re-approximate the square root of $q_t$ by a TT decomposition $\phi_t$, i.e.,
\begin{align}
\sqrt{ q_t(\bx_t, \bth, \bx_{\tminusone})}  & \approx \phi_t (\bx_t, \bth, \bx_{\tminusone} ) \nonumber \\
& =  \mG_{1}(x_{t,1}) \!\cdots\! \mG_{m}(x_{t,m}) \mF_{1}(\theta_1) \cdots \mF_{d}(\theta_d) \mH_{1}(x_{\tminusone,1}) \cdots \mH_{m}(x_{\tminusone,m}). \label{eq:sqrt_1}
\end{align}
and choose a constant $\tau_t$ such that $\tau_t \leq \lerr{\phi_t - \surd q_t}^2$. Following the construction in \eqref{eq:sirt_approx}, this gives a non-negative approximation 
\begin{equation}
	\hat \pi(\bx_t, \bth, \bx_{\tminusone} | \by{}_{1:t}) = \phi_t ( \bx_t, \bth, \bx_{\tminusone} )^2 + \tau_t \lambda(\bx_t) \lambda(\bth) \lambda(\bx_{\tminusone}).\label{eq:sqrt_2} 
\end{equation}

\item[(c)] {\bf Integration.} Applying Proposition \ref{prop:marginal} from the right variable $x_{\tminusone,m}$ to the left variable $x_{\tminusone,1}$, we obtain 
\(
	\hat \pi(\bx_t, \bth | \by{}_{1:t}) =  \int \hat \pi(\bx_t, \bth, \bx_{\tminusone} | \by{}_{1:t}) \, \dd \bx_{\tminusone}
\).
\vspace{4pt}

\end{enumerate}
In the next step $t+1$, we can apply the same procedure using the newly computed marginal approximation $\hat \pi(\bx_t, \bth | \by{}_{1:t})$. 
\end{myalg}

Using Proposition \ref{prop:marginal}, we can integrate the approximate density $\hat \pi(\bx_t, \bth, \bx_{\tminusone}|\by{}_{1:t})$ from the right variable $x_{\tminusone,m}$ to the left variable $x_{t,1}$ to define the lower-triangular KR rearrangement $\mathcal{F}_t^{l} : \mathcal{X}\times \Uptheta \times \mathcal{X} \rightarrow [0,1]^{2m+d}$ in the form of
\begin{equation}\label{eq:lower_KR_full}	
\mathcal{F}^l_t(\bx_t, \bth, \bx_{\tminusone}) =  
\begin{bmatrix*}[l]  \mathcal{F}^l_{t,t}\hspace{6pt}(\bx_{t}) \vspace{4pt}\\ \mathcal{F}^l_{t,\theta}\hspace{5pt}(\bth & \hspace{-11pt} | \bx_{t})  \vspace{4pt}\\ \mathcal{F}^l_{t,\tminusone}(\bx_{\tminusone} & \hspace{-11pt} | \bx_{t},\bth)  \end{bmatrix*} 
= \begin{bmatrix*}[l] 
F_{t,t,1} & \hspace{-9pt} (x_{t,1}) &  \\ & \svdots & \\ 
F_{t,t,m} & \hspace{-9pt} (x_{t,m} & \hspace{-10pt} | \bx_{t,<m}) \\ & \svdots & \\ 
F_{t,\theta,k} & \hspace{-9pt} (\theta_k & \hspace{-10pt} | \bx_{t},\bth_{<k}) \\ & \svdots & \\ 
F_{t,\tminusone,1} & \hspace{-9pt} (x_{\tminusone,1} & \hspace{-10pt} | \bx_{t},\bth) \\ & \svdots & \\ 
F_{t,\tminusone,m} & \hspace{-9pt} (x_{\tminusone,m} & \hspace{-10pt} | \bx_{t},\bth,\bx_{\tminusone,<m}) \end{bmatrix*},
\end{equation}
in which each scalar-valued function $F_{t, \cdot, \cdot}$ is a (conditional) distribution function. For example, we have
\begin{align}\label{eq:conditional_dist}
	F_{t,\tminusone,k} (x_{\tminusone,k} | \bx_t, \bth, \bx_{\tminusone,<k}) = \int_{-\infty}^{x_{\tminusone,k}} \frac{\hat \pi(\bx_t, \bth, \bx_{\tminusone,< k}, x_{\tminusone,k}^\prime |\by{}_{1:t})}{\hat \pi(\bx_t, \bth, \bx_{\tminusone, < k}|\by{}_{1:t})} \dd x_{\tminusone,k}^\prime.
\end{align}
The inverse map $(\mathcal{F}^l_t)^{-1}$ transforms uniform random variables to the approximate posterior random variables $(\widehat{\bX}_t, \widehat{\bTh}, \widehat{\bX}_{\tminusone} | \bY_{1:t} = \by{}_{1:t}) \sim \hat p$, where $\hat p(\bx_t, \bth, \bx_{\tminusone}|\by{}_{1:t}) \propto \hat \pi(\bx_t, \bth, \bx_{\tminusone}|\by{}_{1:t})$. 

Note that the construction of conditional distribution function, e.g., \eqref{eq:conditional_dist}, does not require the normalizing constant of the approximate density $\hat \pi$, 
because the marginal densities share the same normalizing constant. 
Thus, by only integrating the state $\bx_{\tminusone}$  dimension-by-dimension from the right variable $x_{\tminusone,m}$ to the left variable $x_{\tminusone,1}$, we obtain the lower conditional KR rearrangement
\begin{equation}\label{eq:lower_KR_conditional}	
	\mathcal{F}^l_{t,\tminusone}(\bx_{\tminusone} | \bx_{t},\bth)
	= \begin{bmatrix*}[l] F_{t,\tminusone,1} & \hspace{-9pt} (x_{\tminusone,1} & \hspace{-10pt} | \bx_{t},\bth) \\ & \svdots & \\ F_{t,\tminusone,m} & \hspace{-9pt} (x_{\tminusone,m} & \hspace{-10pt} | \bx_{t},\bth,\bx_{\tminusone,<m}) \end{bmatrix*}.
\end{equation}
The following proposition shows that the lower conditional map $\mathcal{F}^l_{t,\tminusone}$ defines a \emph{backward} sampler, which generates samples backward in time from the conditional density $\hat p(\bx_{\tminusone} | \bx_{t}, \bth, \by{}_{1:t})$. This is particularly useful for defining  smoothing algorithms. 

\begin{proposition}\label{prop:backward}
Conditioned on $\widehat{\bX}_{t} = \bx_t $ and $ \widehat{\bTh} = \bth$, inverting the lower conditional KR rearrangement $\mathcal{F}^l_{t,\tminusone}$ defined in \eqref{eq:lower_KR_conditional}, we transform a uniform random variable to the conditional random variable 
\[
	(\widehat{\bX}_{\tminusone} | \widehat{\bX}_{t} = \bx_t,\widehat{\bTh} = \bth,  \bY{}_{1:t} = \by{}_{1:t}) \sim \hat p(\bx_{\tminusone} | \bx_{t}, \bth, \by{}_{1:t}).
\]
\end{proposition}

\begin{proof}
Conditioned on $\widehat{\bX}_{t} = \bx_t $ and $\widehat{\bTh} = \bth$, the Jacobian of $\mathcal F_{t,\tminusone}^l(\cdot|\bx_t, \bth)$,  denoted by $ \mJ \in \R^{m\times m}$, is a lower triangular matrix with diagonal entries 
\[
	\mJ_{kk} = \frac{\partial F_{t,\tminusone,k} (x_{\tminusone,k} | \bx_t, \bth, \bx_{\tminusone,<k})}{\partial x_{\tminusone,k}} = \hat p(x_{\tminusone,k} | \bx_t, \bth, \bx_{\tminusone,<k}, \by{}_{1:t}), \quad k = 1,\ldots,m,
\]
by \eqref{eq:conditional_dist}. Thus, the pullback of the uniform density under $\mathcal F_{t,\tminusone}^l(\cdot|\bx_t, \bth)$ is
\begin{align*}
	(\mathcal F_{t,\tminusone}^l)^\sharp \mathrm{uniform}(\bx_t) =  |\mJ| = \prod_{k=1}^{m}  \hat p(x_{\tminusone,k} | \bx_t, \bth, \bx_{\tminusone,<k}, \by{}_{1:t}) =  \hat p(\bx_{\tminusone} | \bx_t, \bth, \by{}_{1:t}).
\end{align*}
This concludes the result. 
\end{proof}

Similarly, one can integrate the state $\bx_t$ in the approximation $\hat \pi(\bx_t, \bth, \bx_{\tminusone} | \by{}_{1:t})$ from the left variable $x_{t,1}$ to the right variable $x_{t,m}$ to define an upper conditional KR rearrangement 
\begin{equation}\label{eq:upper_KR_conditional}
	\mathcal{F}^u_{t,t}(\bx_t | \bth, \bx_{\tminusone})
	= \begin{bmatrix*}[l] 
	F_{t,t,1} & \hspace{-9pt} (x_{t,1} & \hspace{-10pt} | \bx_{t,>1}, \bth, \bx_{\tminusone}) \\ & \svdots & \\ F_{t,t,m} & \hspace{-9pt} (x_{t,m} & \hspace{-10pt} | \bth, \bx_{\tminusone}) \end{bmatrix*},
\end{equation}
Inverting the conditional map $\mathcal{F}^u_{t,t}$, we are able to transform a uniform random variable to the conditional random variable 
\[
	(\widehat{\bX}_{t} | \widehat{\bTh} = \bth, \widehat{\bX}_{\tminusone} = \bx_{\tminusone},  \bY{}_{1:t} = \by{}_{1:t}) \sim \hat p(\bx_t | \bth, \bx_{\tminusone}, \by{}_{1:t}).
\]
This way, the upper conditional map $\mathcal{F}^u_{t,t}$ defines a \emph{forward} sampler that generates samples forward in time. This will be used in particle filtering.

\subsection{Particle filter}

The upper conditional map $\mathcal{F}^u_{t,t}$ in \eqref{eq:upper_KR_conditional} naturally define a particle filter accompanying the recursive approximations in Alg.~\ref{alg:stt}. Suppose weighted samples $\{\widehat{\bTh}_{\vphantom{t}}^{(i)}, \widehat{\bX}_{0:\tminusone}^{(i)}, W_{\tminusone}^{(i)} \}_{i = 1}^N $ follow the joint posterior density $p(\bth, \bx_{0:\tminusone} |\by{}_{1:\tminusone})$ at time $t{-}1$. Conditioned on each pair of $\{\widehat{\bTh}_{\vphantom{t}}^{(i)}, \widehat{\bX}_{\tminusone}^{(i)}\}_{i=1}^N$, we can invert $\mathcal{F}^u_{t,t}$ to obtain a new sample
\begin{equation}\label{eq:forward_sampling}
	\widehat{\bX}_t^{(i)} = (\mathcal{F}_{t,t}^{u})^{-1} ( \bXi_{\vphantom{t}}^{(i)} | \widehat{\bTh}_{\vphantom{t}}^{(i)},\widehat{\bX}_{\tminusone}^{(i)} ), \quad \bXi_{\vphantom{t}}^{(i)} \sim \mathrm{uniform}(\bxi; [0,1]^m).
\end{equation}
The sample $\widehat{\bX}_t^{(i)}$ follows the conditional density $\hat p(\bx_t |\bth, \bx_{\tminusone}, \by{}_{1: t})$. 
Expanding the sample state path $\widehat{\bX}_{0: \tminusone}^{(i)}$ by $\widehat{\bX}_t^{(i)}$, the updated weighted samples $\{\widehat{\bTh}_{\vphantom{t}}^{(i)}, \widehat{\bX}_{0: t}^{(i)}, W_{\tminusone}^{(i)}\}_{i =1}^{N}$ jointly follow the normalized density 
\begin{equation}\label{eq:forward_sampling_density}
	\hat p(\bx_t |\bth, \bx_{\tminusone}, \by{}_{1: t}) \, p(\bth, \bx_{0:\tminusone} |\by{}_{1: \tminusone}).
\end{equation}

After applying the sampling procedure in \eqref{eq:forward_sampling}, we can update the weights of the updated samples $\{\widehat{\bTh}_{\vphantom{t}}^{(i)}, \widehat{\bX}_{0: t}^{(i)}, W_{\tminusone}^{(i)}\}_{i =1}^{N}$ according to the ratio of the new joint posterior density $p(\bth, \bx_{0: t} | \by{}_{1: t}) $ in \eqref{eq:joint_post} to the sampling density in \eqref{eq:forward_sampling_density}, which take the form 
\[
	r_{\rm f}(\bth, \bx_{0:t}) = \frac{ p(\bth, \bx_{0: t} |\by{}_{1: t})}{\hat p(\bx_t |\bth, \bx_{\tminusone}, \by{}_{1: t}) p(\bth, \bx_{0:\tminusone} |\by{}_{1: \tminusone})} .	
\]
Following the recursive form of $p(\bth, \bx_{0: t} |\by{}_{1: t})$ in \eqref{eq:joint_post}, we have
\begin{align}
	r_{\rm f}(\bth, \bx_{0:t}) \propto \omega_{\rm f}(\bth, \bx_{\tminusone:t}) := \frac{ f(\bx_t |\bx_{\tminusone}, \bth) g(\by{}_t |\bx_t, \bth)}{ \hat p(\bx_t |\bth, \bx_{\tminusone}, \by{}_{1: t}) } .\label{eq:update_weight}
\end{align}
Thus, by updating the weights using $W_{t}^{(i)} = W_{\tminusone}^{(i)} \, \omega_{\rm f}(\widehat{\bTh}_{\vphantom{t}}^{(i)}, \widehat{\bX}_{\tminusone:t}^{(i)})$ followed by the renormalization  
\[
W_{t}^{(i)} = \frac{W_t^{(i)}}{\sum_{i=1}^{N} W_t^{(i)}} \quad \text{for} \quad i = 1, \ldots, N, 
\] 
we have the reweighed samples $\{\widehat{\bTh}_{\vphantom{t}}^{(i)}, \widehat{\bX}_{0: t}^{(i)}, W_{t}^{(i)}\}_{i =1}^{N}$ that follow the joint posterior density $p(\bth, \bx_{0: t} |\by{}_{1: t})$. Alg.~\ref{alg:stt_weights} provides details of the particle filter accompanying the sequential learning procedure in Alg.~\ref{alg:stt}.

\begin{myalg}{One iteration of the particle filter accompanying Alg.~\ref{alg:stt}.}\label{alg:stt_weights}
\begin{enumerate}[leftmargin=18pt]

\item[(a)] At time $t{-}1$, suppose we have weighted samples $\{\widehat{\bTh}_{\vphantom{t}}^{(i)}, \widehat{\bX}_{0:\tminusone}^{(i)}, W_{\tminusone}^{(i)} \}_{i = 1}^N $ following the joint density $p(\bth, \bx_{0:\tminusone} |\by{}_{1:\tminusone})$.
\item[(b)] At time $t$, given the approximate density $\hat \pi(\bx_t, \bth, \bx_{\tminusone} | \by{}_{1:t})$ computed in Alg.~\ref{alg:stt}, apply Proposition \ref{prop:marginal} to integrate $\hat \pi$ from the left variable $x_{t,1}$ to the right variable $x_{t,m}$ to define a conditional KR rearrangement $\mathcal{F}_{t,t}^u$ in the form of \eqref{eq:upper_KR_conditional}.
\item[(c)] For each of $\{\widehat{\bTh}_{\vphantom{t}}^{(i)}, \widehat{\bX}_{0:\tminusone}^{(i)}\}_{i = 1}^N $, invert  $\mathcal{F}_{t,t}^u$ as in \eqref{eq:forward_sampling} to generate a new sample $\widehat{\bX}_t^{(i)}$, and then update the weights according to $ W_t^{(i)} = W_{\tminusone}^{(i)} \, \omega_{\rm f}(\widehat{\bTh}_{\vphantom{t}}^{(i)}, \widehat{\bX}_{\tminusone:t}^{(i)})$, where the function $\omega_{\rm f}(\cdot, \cdot)$ is defined in \eqref{eq:update_weight}. 
\item[(d)] Renormalize the weights $W_t^{(i)} \gets W_t^{(i)}/\sum_{i=1}^{N} W_t^{(i)}$.
\end{enumerate}
\end{myalg}

In Alg.~\ref{alg:stt_weights}, constructing the conditional densities in step (b), which defines the conditional KR rearrangement $ \mathcal F_{t,t}^u $, requires $ \mathcal O(m\ell r^3) $ flops. This step  does not involve any sampling, and thus its cost is independent of the sample size $ N $. Then, in step (c), generating $ N $ samples from the resulting conditional KR rearrangement requires an additional $ \mathcal O(Nm\ell (r^2 + \log\ell+r + c)) $ flops, where $c$ is the maximum number of iterations used in the root finding algorithms, as discussed in Sec.~\ref{sec:stt}.

In summary, Alg.~\ref{alg:stt_weights} can be viewed as a standard particle filter with TT approximations to be the proposal density.
Commonly used resampling techniques in SMC methods can also be used in Alg.~\ref{alg:stt_weights} as a rejuvenation  to re-balance the weights.

\subsection{Path estimation and smoothing}\label{sec:smoothing}

Similar to other particle filter algorithms, weights computed by Alg.~\ref{alg:stt_weights} may degenerate over time. This has an intuitive explanation in the joint state and parameter  estimation context. The locations of the parameter samples are fixed at the initial time, so that the weights (and hence the effective sample size) necessarily degenerate because the posterior parameter density may concentrate with more data observed over time. 
We overcome this issue by drawing random samples from the last approximation $\hat p(\bx_T, \bth, \bx_{T\text{-}1} | \by{}_{1: T})$---which is conditioned on all available data by construction---and removing the approximation bias by designing  a path estimation algorithm accompanying Alg.~\ref{alg:stt} to replace the particle filtering. The path estimation algorithm also naturally leads to particle smoothing.

Our path estimation algorithm uses the same sequence of non-negativity-preserving approximations constructed in Alg.~\ref{alg:stt}.
It starts at the final time $T$ and recursively samples backward in time using the lower conditional map $\mathcal{F}^l_{t,\tminusone}$ defined in \eqref{eq:lower_KR_conditional}.
At time $T$, initial samples $\{\widehat{\bX}_{T\text{-}1}^{(i)}, \widehat{\bX}_T^{(i)}, \widehat{\bTh}_{\vphantom{t}}^{(i)}\}_{i = 1}^N$ are generated from $  \hat p(\bx_T, \bth, \bx_{T\text{-}1} | \by{}_{1: T}) $ using a full KR rearrangement $\mathcal{F}^l_T$ defined in \eqref{eq:lower_KR_full}.
Then, at each of steps $t = T{-}1, \ldots, 1$, for each of $\{\widehat{\bTh}_{\vphantom{t}}^{(i)}, \widehat{\bX}_{t:T}^{(i)} \}_{i = 1}^N $, we invert the lower conditional map $\mathcal{F}^l_{t,\tminusone}$ to obtain 
\begin{equation}\label{eq:backward_sampling}
\widehat{\bX}_{\tminusone}^{(i)} =  (\mathcal{F}_{t,\tminusone}^{l})^{-1} ( \bXi_{\tminusone}^{(i)} | \widehat{\bTh}_{\vphantom{t}}^{(i)},\widehat{\bX}_{t}^{(i)} ), \quad \bXi_{\tminusone}^{(i)} \sim \mathrm{uniform}(\bxi; [0,1]^m).
\end{equation}
Following from Proposition \ref{prop:backward}, after completing backward recursion, each pair of the parameter sample and state path sample in $\{\widehat{\bTh}_{\vphantom{t}}^{(i)}, \widehat{\bX}_{0:T}^{(i)}\}_{i = 1}^N$ follows the joint density
\begin{equation}\label{eq:smooting_density}
	 \hat p(\bth, \bx_{0: T} | \by{}_{1: T}) = \hat p(\bx_T, \bth, \bx_{T\text{-}1} | \by{}_{1: T}) \prod_{t=1}^{T-1} \hat p(\bx_{\tminusone} |\bth, \bx_{t} ,  \by{}_{1: t}).
\end{equation}
Conditioned on $ \bth $ and $ \bx_t $, the state $ \bx_{\tminusone} $ is independent of the data $\by_{t:T}$ observed at future times, and hence we have $ p(\bx_{\tminusone} |\bth, \bx_{t} ,  \by{}_{1: t}) = p(\bx_{\tminusone} |\bth, \bx_{t} ,  \by{}_{1: T}) $. Using this identity, the original joint posterior density $p(\bth, \bx_{0: T} | \by{}_{1: T})$ defined in \eqref{eq:joint_post_pdf} has the same factorized form as the density in \eqref{eq:smooting_density}. Therefore, the above backward sampling procedure can have a reasonable efficiency if the approximations built in Alg.~\ref{alg:stt} are sufficiently accurate.

\begin{figure}[h!]
	\centering
	\includegraphics[width=0.99\textwidth]{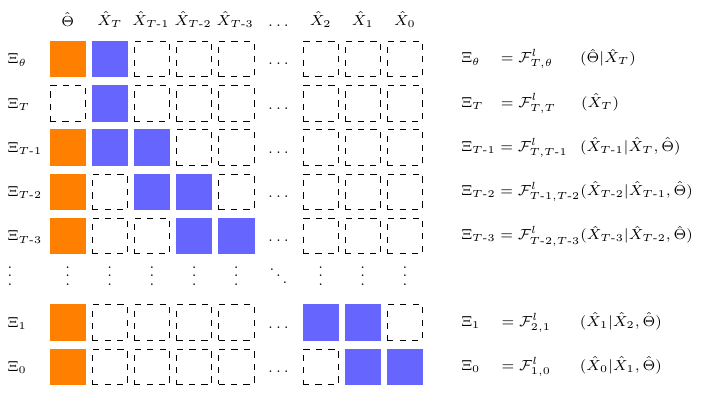}\vspace{-0pt}
	\caption{Conditional dependency structures of the joint map $\cF$ used in particle smoothing. The blue entries indicate the conditional dependency on the states, and the orange entries indicate the conditional dependency on the parameters. The $2\times 2$ sub-matrix in the top left corner represents the structure of the first two blocks of the map $\mathcal{F}^l_T$.} \label{fig:dependency_pattern}
\end{figure}

In fact, the backward recursion can be considered as the inverse of a joint KR rearrangement $\cF:\Uptheta \times \mathcal{X}^{T+1}  \rightarrow [0,1]^{d + (m+1) \times T }$ that transforms approximate posterior random variables $(\widehat{\bTh}, \widehat{\bX}_{0:T} | \bY_{1:T} = \by{}_{1:T})$ that follow the density \eqref{eq:smooting_density}, to the reference uniform random variables. As illustrated in Figure \ref{fig:dependency_pattern}, the joint KR rearrangement $\cF$ has a block sparse structure. The non-zero blocks in the first column reflect the conditional dependency of each of the state variables on the parameter, while the non-zero blocks on the first sub-diagonal reflect the Markov dependency of state variables. From this perspective, our TT-based construction of the joint KR rearrangement $\cF$ can be viewed as a constructive proof and a numerical implementation of the decomposition theorem \cite[Theorem 12]{spantini2018inference} for the joint parameter and state estimation in general state-space models. 

For a pair of sample parameter and sample state path $(\widehat{\bTh}_{\vphantom{t}}^{(i)}=\bth, \widehat{\bX}_{0:T}^{(i)}=\bx_{0:T})$ generated by the backward recursion, it yields a weighted representation of the posterior sample with the importance weight 
\[
	r_{\rm b}(\bth, \bx_{0:T}) = \frac{p(\bth, \bx_{0: T}| \by{}_{1: T} )}{\hat p(\bth, \bx_{0:T} | \by{}_{1: T}) }.
\]
Following the recursive form of $p(\bth, \bx_{0: t} |\by{}_{1: t})$, the ratio $r_{\rm b}(\bth, \bx_{0:T})$ has an  unnormalized computable form
\begin{align}
	r_{\rm b}(\bth, \bx_{0:T}) & \propto  \omega_{\rm b}(\bth, \bx_{0:T}) \nonumber\\
	& :=  {p(\bth) p(\bx_0 |\bth)}
	\Big({\prod_{t=1}^{T-1} \frac{f(\bx_t |\bx_{\tminusone}, \bth) g(\by{}_t |\bx_t, \bth)}{\hat p(\bx_{\tminusone} |\bth, \bx_{t} ,  \by{}_{1: t})} } \Big) \frac{f(\bx_T |\bx_{T\text{-}1}, \bth) g(\by{}_T |\bx_T, \bth)}{\hat p(\bx_T, \bth, \bx_{T\text{-}1} | \by{}_{1: T}) }. \label{eq:smoothing_weights}
\end{align}
We summarize the particle smoothing procedure in Alg.~\ref{alg:stt_smooth}. The set of weighted samples $ \{\widehat{\bTh}_{\vphantom{t}}^{(i)}, \widehat{\bX}_{0:T}^{(i)}, W_{\vphantom{t}}^{(i)}\}_{i=1}^{N}$ is an unbiased representation of the joint posterior random variables $(\bTh, \bX_{0: T} |\by{}_{1: T})$. It mitigates the particle degeneracy because the samples are drawn conditioned all observed data $ \by{}_{1:T} $. 

\begin{myalg}{Path estimation algorithm accompanying Alg.~\ref{alg:stt}.}\label{alg:stt_smooth}
\begin{enumerate}[leftmargin=18pt]
\item[(a)] Generate samples $ \{\widehat{\bX}_{T}^{(i)}, \widehat{\bTh}_{\vphantom{t}}^{(i)}, \widehat{\bX}_{T\text{-}1}^{(i)} \}_{i = 1}^N  $ from $\hat p(\bx_T, \bth, \bx_{T\text{-}1} |\by{}_{1:T}) $ using $\mathcal{F}^l_T$.
\item[(b)] For $t = T-1, \ldots, 1$, do the following
\begin{itemize}[leftmargin=18pt]
\item For each of $\{\widehat{\bTh}_{\vphantom{t}}^{(i)}, \widehat{\bX}_{t:T}^{(i)}\}_{i = 1}^N $, invert the lower conditional map $\mathcal{F}_{t,\tminusone}^l$ as in \eqref{eq:backward_sampling} to generate a new sample $\widehat{\bX}_{\tminusone}^{(i)}$. 
\end{itemize}
\item[(c)] For each of $\{\widehat{\bTh}_{\vphantom{t}}^{(i)}, \widehat{\bX}_{0:T}^{(i)}\}_{i = 1}^N $, compute the weight $W^{(i)}$ using \eqref{eq:smoothing_weights}.  \item[(d)] Normalize the weights as 
\(
W^{(i)} \gets  W^{(i)} / \sum_{i=1}^{N} W^{(i)}.
\)
\end{enumerate}
\end{myalg}

In our setting, applying one step of particle filtering and applying one step of particle smoothing have the same computational complexity. 
For particle filtering, integrals used for constructing the upper conditional map $\mathcal{F}_{t,t}^u$ in \eqref{eq:upper_KR_conditional} take the opposite direction to the integrals used for building $\hat \pi(\bx_t, \bth | \by{}_{1:t})$ in Alg.~\ref{alg:stt}. 
This incurs an additional $\mathcal{O}(m\ell r^3)$ flops. 
In comparison, constructing the lower conditional map $\mathcal{F}_{t,\tminusone}^l$ in \eqref{eq:lower_KR_conditional} for particle smoothing shares the same integrals as in Alg.~\ref{alg:stt}, and thus does not incur additional computational cost. 

Marginalizing the square TT approximation \eqref{eq:sqrt_2} from either end has an $\mathcal{O}(\ell r^3)$ computational complexity per coordinate. This is much lower than marginalizing variables in the middle, which has an $\mathcal{O}(\ell r^6)$ computational complexity per coordinate. Thus, the variable ordering $(\bx_t, \bth, \bx_{\tminusone})$ is computationally advantageous for both smoothing and filtering. One can also consider the variable ordering $(\bth, \bx_t, \bx_{\tminusone})$ if only the smoothing algorithm is used for debiasing, as Alg.~\ref{alg:stt} and the lower conditional map in Alg.~\ref{alg:stt_smooth} only need to integrate the state $\bx_{\tminusone}$.

\section{Error analysis}\label{sec:error}

At the core of Alg.~\ref{alg:basic} and \ref{alg:stt} is the recursive approximation of the joint densities of the posterior random variables $(\bX_t,\bTh, \bX_{\tminusone} |\bY_{1:t} = \by{}_{1:t})$. We begin this section with an analysis of the accumulation of approximation errors in general state-space models. Based on this result, we provide an upper bound on the TT-approximation errors of Alg.~\ref{alg:basic} and show the stability of Alg.~\ref{alg:stt}. Although most of the analysis presented centers around the TT decomposition, our results on the error accumulation rate are readily generalizable to other function approximation schemes applied to state-space models. For instance, one can consider the deep polynomial method of \cite{cui2023self} and the sum-of-square approximation of \cite{zanger2024sequential}. The same error accumulation analysis will apply. 

\subsection{Error decomposition and propagation} \label{sec:error_decomp}

Starting with the initial unnormalized joint posterior density 
\[\pi(\bx_1, \bth, \bx_0 | \by{}_{1}) = p(\bth)p(\bx_0|\bth)f(\bx_1|\bx_0, \bth)\lk{1},\] 
we define the unnormalized joint posterior density recursively as
\begin{align*}
\pi(\bx_t, \bth, \bx_{\tminusone} | \by{}_{1:t}) & = \int \pi(\bx_{\tminusone}, \bth, \bx_{t\text{-}2} | \by{}_{1:\tminusone}) \dd \bx_{t\text{-}2} \; \tr{t}\lk{t} \\
& = \quad\, \pi(\bx_{\tminusone}, \bth | \by{}_{1:\tminusone}) \tr{t}\lk{t},
\end{align*}
which follows a similar derivation to that of \eqref{eq:rec_post}. 
We use shorthands $p_t(\bx_t, \bth, \bx_{t\text{-}1} ) := p(\bx_t, \bth, \bx_{\tminusone} | \by{}_{1:t})$ and $\hat p_t(\bx_t, \bth, \bx_{\tminusone})$ $:= \hat p(\bx_t, \bth, \bx_{\tminusone} | \by{}_{1:t})$ to respectively denote the joint posterior density at time $t$ and its TT-based approximation. Similar notation is adopted for representing their unnormalized counterparts. We omit the input variables when no confusion arises. 
The density $\pi_t$ has the normalizing constant $ z_t := \int \pi(\bx_t, \bth, \bx_{\tminusone} | \by{}_{1:t}) \dd \bx_t \, \dd \bth \, \dd \bx_{\tminusone} $, which is indeed the evidence $p(\by{}_{1:t})$ introduced in \eqref{eq:joint_post_pdf}.

Recall that in each iteration of Alg.~\ref{alg:basic} and \ref{alg:stt}, the unnormalized joint posterior density $\pi_t$ cannot be directly evaluated since it is the density of marginal random variables. Instead, we introduce the non-separable, unnormalized approximation $q_t$ in step (a) of both algorithms, and then approximate $q_t$ using a TT-based approximation $\hat\pi_t$ in step (b). As shown in \eqref{eq:tt_0} and \eqref{eq:sqrt_0}, the function $q_t$ takes the form
\[	
	q_t(\bx_t, \bth, \bx_{\tminusone}) = \hat \pi( \bx_{\tminusone}, \bth | \by{}_{1:\tminusone}) f(\bx_t |\bx_{\tminusone}, \bth) g(\by{}_t |\bx_t, \bth),
\]
where $\hat \pi( \bx_{\tminusone}, \bth | \by{}_{1:\tminusone})$
is the marginalization of the previous TT-based approximation. The construction of the function $q_t$ enables pointwise function evaluations that are necessary for building the new TT decomposition used by $\hat\pi_t$.

This reveals that each iteration of Alg.~\ref{alg:basic} and \ref{alg:stt} can be considered as a two-step approximation procedure: the unnormalized joint density $\pi_t$ is approximated by the function $q_t$, from which we construct the TT-based approximation $\hat \pi_t$. Considering the triangle inequality for some distance metric $D(\cdot, \cdot)$, the error $D(\hat \pi_t, \pi_t) $ has the decomposition 
\begin{equation} 
D( \hat\pi_t ,  \pi_t ) \leq D( \hat \pi_t ,  q_t )  + D( q_t,  \pi_t ), \label{eq:norm_tri}
\end{equation}
where $D(\hat \pi_t,  q_t)$ is the \emph{approximation error} in the current iteration and $D( q_t, \pi_t)$ is the \emph{propagation error} from the previous iteration. In particular, we consider the $L^1$ distance of unnormalized densities and the $L^2$ distance of the square roots of unnormalized densities as the distance metrics for bounding these errors. The $L^1$ distance is useful for approximations that cannot preserve non-negativity (e.g., Alg.~\ref{alg:basic}) and is analogous to the total variation distance of probability measures. The $L^2$ distance on square-root densities is directly related to the Hellinger distance (see Lemma \ref{lemma:sirt_error}) and we use it to analyze Alg.~\ref{alg:stt}. We aim to understand how the approximation errors accumulate over time in the recursive algorithms. Based on the following assumption, we first analyze the propagation error in Propositions~\ref{prop:L1_propagation} and~\ref{prop:L2_propagation}. 

\begin{assumption}\label{assumption:propagation}
For the state transition process and the observation process, we assume either of the following bounds holds
\begin{align}
	C^{(g)}_t & = \sup_{\bx_t \in \mathcal{X}, \bth \in \Uptheta} g(\by{}_t |\bx_t, \bth) < \infty , \label{eq:ctg} 
	\\
	C^{(f)}_t & = \sup_{\bx_{\tminusone} \in \mathcal{X}, \bth \in \Uptheta} \, \int f(\bx_t |\bx_{\tminusone}, \bth) g(\by{}_t |\bx_t, \bth) \dd \bx_t < \infty.\label{eq:ctf}
\end{align}
\end{assumption}

Generally speaking, any Lipschitz continuous $f$ and $g$ satisfy the above assumptions, which are usually the case in many applications. For example, an observation model with the commonly used Gaussian noise has a bounded $C^{(g)}_t$ for all $t$. 

\begin{proposition}\label{prop:L1_propagation}
	Suppose either bound of Assumption \ref{assumption:propagation} holds. The propagation error $\| q_t - \pi_t \|_{L^1}$ at time $t$ is bounded by the previous total error
	\begin{equation}
		\|q_t -  \pi_t\|_{L^1}  \leq   C_t^{(h)} \| \hat \pi_{\tminusone} - \pi_{\tminusone} \|_{L^1}, \quad h \in \{f, g\}.
	\end{equation}
	\end{proposition}
	
	\begin{proof}
	By the construction of $q_t$, we express the propagation error as
	\begin{align}
		\| q_t - \pi_t\|_{L^1} & = \int \Big| {\hat \pi( \bx_{\tminusone}, \bth | \by{}_{1:\tminusone})} - {\pi( \bx_{\tminusone}, \bth | \by{}_{1:\tminusone})}\Big| r(\bx_{\tminusone} , \bth)  \dd \bx_{\tminusone} \dd \bth , \label{eq:L1_propagation}
	\end{align}
	where $ r(\bx_{\tminusone} , \bth) = \int f(\bx_t |\bx_{\tminusone}, \bth) g(\by{}_t |\bx_t, \bth) \dd \bx_{t} $. Either condition \eqref{eq:ctg} or \eqref{eq:ctf} leads to
	\(
		\sup_{\bx_{\tminusone} \in \mathcal{X}, \bth \in \Uptheta} r(\bx_{\tminusone} , \bth) \leq C_t^{(h)}
	\)
	for $h \in \{f, g\}$, and thus we have
	\begin{align*}
		\| q_t - \pi_t\|_{L^1} 
			& \leq  C_t^{(h)} \big\|\hat \pi( \bx_{\tminusone}, \bth | \by{}_{1:\tminusone}) - \pi( \bx_{\tminusone}, \bth | \by{}_{1:\tminusone}) \big\|_{L^1}.
	\end{align*}
	Applying Lemma \ref{lemma:L1} in Appendix \ref{appendix:A}, the $ L^1 $ distance of the marginal densities is bounded by that of the joint densities as
	\begin{align*}
		\big\|\hat \pi( \bx_{\tminusone}, \bth | \by{}_{1:\tminusone}) - \pi( \bx_{\tminusone}, \bth | \by{}_{1:\tminusone}) \big\|_{L^1} \leq  \| \hat \pi_{\tminusone} - \pi_{\tminusone} \|_{L^1},
	\end{align*}
	which leads to the result.
\end{proof}

\begin{proposition}\label{prop:L2_propagation}
	Suppose either bound of Assumption \ref{assumption:propagation} holds. The propagation error $\lerr{\surd q_t - \surd \pi_t}$ at time $t$ is bounded by the previous total error
	\begin{equation}
		\lerr{\surd q_t - \surd \pi_t}  \leq  \surd C_t^{(h)} \lerr{\surd \hat \pi_{\tminusone} - \surd \pi_{\tminusone}}, \quad h \in \{f, g\}.
	\end{equation}
	\end{proposition}
	
	\begin{proof}
	By the construction of $q_t$, we express the squared propagation error as
	\begin{align}
		\| \surd q_t - \surd \pi_t\|_{L^2}^2 & \!=\!\!\int \!\! \Big(\! \sqrt{\hat \pi( \bx_{\tminusone}, \bth | \by{}_{1:\tminusone})} -\! \sqrt{\pi( \bx_{\tminusone}, \bth | \by{}_{1:\tminusone})}\Big)^2 r(\bx_{\tminusone} , \bth)  \dd \bx_{\tminusone} \dd \bth , \label{eq:L2_propagation}
	\end{align}
	where $ r(\bx_{\tminusone} , \bth) = \int f(\bx_t |\bx_{\tminusone}, \bth) g(\by{}_t |\bx_t, \bth) \dd \bx_{t} $. Either condition \eqref{eq:ctg} or \eqref{eq:ctf} leads to
	\(
		\sup_{\bx_{\tminusone} \in \mathcal{X}, \bth \in \Uptheta} r(\bx_{\tminusone} , \bth) \leq C_t^{(h)}
	\)
	for $h \in \{f, g\}$, and thus we have
	\begin{align*}
		\| \surd q_t - \surd \pi_t\|_{L^2} 
		 & \leq  \bigg( C_t^{(h)} \int \Big(\sqrt{\hat \pi( \bx_{\tminusone}, \bth | \by{}_{1:\tminusone})} - \sqrt{\pi( \bx_{\tminusone}, \bth | \by{}_{1:\tminusone})}\Big)^2  \dd \bx_{\tminusone} \dd \bth \bigg)^{1/2} \\
		 & = \surd C_t^{(h)} \big\|\sqrt{\hat \pi( \bx_{\tminusone}, \bth | \by{}_{1:\tminusone})} - \sqrt{\pi( \bx_{\tminusone}, \bth | \by{}_{1:\tminusone})}\big\|_{L^2}.
	\end{align*}
	Applying Lemma \ref{lemma:L2_sqrt} in Appendix \ref{appendix:A}, the $ L^2 $ distance for the marginal densities is bounded by that for the joint densities as
	\begin{align*}
		\big\|\sqrt{\hat \pi( \bx_{\tminusone}, \bth | \by{}_{1:\tminusone})} - \sqrt{\pi( \bx_{\tminusone}, \bth | \by{}_{1:\tminusone})}\big\|_{L^2} \leq  \lerr{\surd{\hat \pi_{\tminusone}} - \surd{\pi_{\tminusone}}},
	\end{align*}
	which leads to the result.
\end{proof}

Based on the error propagation bounds, we can then analyze the total approximation errors of our sequential estimation algorithms. In the following, we use assumed approximation errors to show the total errors of Alg.~\ref{alg:basic} and \ref{alg:stt}. 

\begin{theorem} \label{thm1}
	Suppose either bound of Assumption \ref{assumption:propagation} holds and the approximation $\hat \pi_t$ in Alg.~\ref{alg:basic} satisfies $\|\hat \pi_t -  q_t\|_{L^1} \leq \epsilon_t$. For $t > 0$, the total error satisfies 
	\[
	\| \hat \pi_t -  \pi_t\|_{L^1} \leq  C_t^{(h)} \| \pi_{\tminusone} - \hat \pi_{\tminusone} \|_{L^1} + \epsilon_t \leq \sum_{k=1}^{t} C^{t-k} \epsilon_k,	\quad C = \sup_t C_t^{(h)},
	\]
	for $h\in\{f,g\}$.
\end{theorem}
	
\begin{proof}
	The result follows from the triangle inequality \eqref{eq:norm_tri} and induction.
\end{proof}

\begin{theorem} \label{thm2}
	Suppose either bound of Assumption~\ref{assumption:propagation} hold and the approximation $\hat \pi_t$ in Alg.~\ref{alg:stt} satisfies $\|\surd \hat \pi_t - \surd q_t\|_{L^2} \leq \epsilon_t$. For $t > 0$, the Hellinger distance between $ p_t $ and its approximation $ \hat p_t $ is bounded by
	\[
		\dhd{\hat p_t}{p_t} \leq \frac{\surd 2}{\sqrt{p(\by{}_{1:t})}} \sum_{k=1}^{t} C^{(t-k)/2} \epsilon_k, 
	\]
	where $C = \sup_t C_t^{(h)}$ with $h\in\{f,g\}$ and $p(\by{}_{1:t})$ is the evidence.
\end{theorem}
	
\begin{proof}
	We recall that there is no approximation used in Alg.~\ref{alg:stt} at the initial step $t = 0$, i.e., $\epsilon_0 = 0$, as we only have the prior for $\bX_0$ and $\bTh$. Applying the triangle inequality \eqref{eq:norm_tri} and induction, we have 
	\[
		\| \surd \hat \pi_t - \surd \pi_t\|_{L^2} \leq  \surd C_t^{(h)} \|\surd \pi_{\tminusone} - \surd \hat \pi_{\tminusone} \|_{L^2} + \epsilon_t \leq \sum_{k=1}^{t} C^{(t-k)/2} \epsilon_k.	
	\]
	Applying Lemma \ref{lemma:sirt_error}, the Hellinger distance between $ \hat p_t $ and $ p_t $ is bounded by the $ L^2 $ distance between $ \surd \hat \pi_t $ and $ \surd \pi_t $, i.e.,
\begin{equation*}
 	\dhd{\hat p_t}{p_t} \leq \frac{\surd 2}{\surd z_t} \|\surd \hat \pi_t - \surd \pi_t\|_{L^2}, 
\end{equation*}
	where $ z_t = p(\by{}_{1:t}) $ by definition. This concludes the result.
\end{proof}

By Lemmas \ref{lemma:marginal} and \ref{lemma:L1} in Appendix \ref{appendix:A}, either the Hellinger distance or the $L^1$ distance between the exact marginal and the approximate marginal is bounded from above by the corresponding distance between the joint densities. Thus, approximations of the filtering density $p(\bx_{t} | \by{}_{1:t})$ and the posterior parameter density $p(\bth | \by{}_{1:t})$ obtained by Alg.~\ref{alg:basic} and \ref{alg:stt}, which are marginalizations of the joint approximation $\hat p(\bx_t, \bth, \bx_{\tminusone} | \by{}_{1:t})$, follow the same error bounds derived here. 

\subsection{Error analysis for TT approximations}\label{sec:error_tt}

For Alg.~\ref{alg:basic}, we first show in Proposition \ref{prop:tt_error} that the approximation error $\| \hat \pi_t -  q_t\|_{L^1}$ can be bounded based on certain Sobolev-type smoothness assumptions. Then, we can apply the result of Theorem \ref{thm1} to bound the total error.

\begin{proposition}\label{prop:tt_error}
For any time $t>0$, we express the product of the state transition density and the likelihood function as 
\[
h_t(\bx_t, \bth, \bx_{\tminusone}) = f(\bx_t |\bx_{\tminusone}, \bth) g(\by{}_t |\bx_t, \bth).
\] 
We assume that state space $\mathcal{X}$ and the parameter space $\Uptheta$ are compact in the sense that they admit finite Lebesgue measures and there exist some $K \in \mathbb N $ and $s \geq 1$ such that $h_t \in W^{K+1, 2s} (\mathcal{X}\times \Uptheta \times \mathcal{X}) $. At time $t = 1$, we further assume that the (unnormalized) prior density $\pi(\bx_0, \bth) := \pi(\bx_0 | \bth) \pi(\bth)$ belongs to $ W^{K+1, 2r} (\mathcal{X}\times \Uptheta)$ for some $r \geq 1$ such that $\frac{1}{r} + \frac{1}{s} = 1$. Then, for some error controlling factor $\varepsilon_t \in (0,1)$ at time $t > 0$, there exists a TT decomposition $\hat \pi_t (\bx_t, \bth, \bx_{\tminusone}) $ with ranks 
\begin{equation}
	r_1 = \lceil \varepsilon_t^{-1/K} \rceil \quad \text{and} \quad r_k = \lceil \varepsilon_t^{-k/K}\rceil \quad \text{for}  \quad k = 2,3,\ldots, 2m+d-1, \label{eq:tt_error_rank}
\end{equation}
such that the approximation error in step (b) of Alg.~\ref{alg:basic} is bounded by
\[
\| \hat \pi_t - q_t\|_{L^1} \leq  C_\mathrm{TT} \sqrt{\Omega(\mathcal{X}\times\Uptheta\times \mathcal{X}) (2m+d-1)} \, \varepsilon_t,
\]
where $\Omega(\cdot)$ denotes the volume of the space and $C_\mathrm{TT}$ is a constant independent of $\hat \pi_t$, $q_t$, and the dimension. 
\end{proposition}

\begin{proof}
Our goal is to establish the Sobolev-type smoothness of the function $q_t$, so that results of \cite{griebel2021analysis} can be used to derive the TT ranks for some target error  measured in the $L^2$ norm. The bound on the $L^1$ norm follows from that on the $L^2$ norm using the Cauchy--Schwartz inequality.

For brevity, we define the previous marginal density $\hat \pi( \bx_{\tminusone}, \bth | \by{}_{1:\tminusone})$ as
\[
\hat \varphi_t(\bx_t, \bth, \bx_{\tminusone}) =
\left\{\begin{array}{ll}
	\hat \pi( \bx_{\tminusone}, \bth | \by{}_{1:\tminusone}), & \quad t > 1\\
	\pi( \bx_0, \bth), & \quad t = 1
\end{array}\right. ,
\]
which takes a constant value over $\bx_t$. For $t > 1$, since the marginal TT approximation $\hat \pi( \bx_{\tminusone}, \bth | \by{}_{1:\tminusone})$ is a linear combination of multivariate polynomial basis functions, the function $\hat \varphi_t$ is an analytical function. Thus, we have $\hat \varphi_t \in  W^{K+1, 2r} (\mathcal{X}\times \Uptheta \times \mathcal{X})$ for any $K \in \mathbb{N}$, $r \geq 1/2$, and $t>1$.

For $t > 0$, the $L^2$ norm of $q_t$ can be expressed as
\(
\| q_t\|_{L^2}^2 = \|\hat \varphi_t^2\, h_t^2 \|_{L^1}^{}.
\)
Applying the H\"{o}lder inequality, we have
\[
	\| q_t\|_{L^2} \leq \|\hat \varphi_t^2 \|_{L^{r}}^{1/2} \, \|h_t^2 \|_{L^{s}}^{1/2} = \|\hat \varphi_t \|_{L^{2r}} \|h_t \|_{L^{2s}} ,
\]
for $r,s \in [0, \infty]$ with $\frac{1}{r} + \frac{1}{s} = 1$. This way, we have $\| q_t\|_{L^2} < \infty$ for all $t > 0$ by our assumptions and the analyticity of $\hat \varphi_t$ for $t > 1$. For derivatives of $q_t$, with the multi-index notation and the general Leibniz rule (see \cite{constantine1996multivariate}), we have
\[ 
\partial^{\boldsymbol\alpha} q_t = \sum_{{\boldsymbol\beta} \leq {\boldsymbol\alpha}}  { \boldsymbol\alpha \choose \boldsymbol\beta} \big( \partial^{\boldsymbol\alpha-\boldsymbol\beta}\hat \varphi_t\big)  \big( \partial^{\boldsymbol\beta} h_t \big) ,
\]
where $\boldsymbol \alpha, \boldsymbol\beta$ are multi-indices. Following a similar derivation as in the $ L_2 $ norm case, we can show that $\|\partial^{\boldsymbol\alpha} q_t\|_{L^2} < \infty$ for all $|\boldsymbol \alpha| \leq K+1$, where $|\boldsymbol \alpha| = \sum \alpha_i$. Therefore, we have $q_t \in  W^{K+1, 2} (\mathcal{X}\times \Uptheta \times \mathcal{X}) $ for all $t>0$.

Given a target error trolling factor $\varepsilon_t \in (0,1)$, Theorem 4 of \cite{griebel2021analysis} states that there exists a TT decomposition $\hat \pi_t$ with ranks $r_1 = \lceil \varepsilon_t^{-1/K} \rceil$ and $r_k = \lceil \varepsilon_t^{-2k/K} \rceil$ for $k = 2,3,\ldots, 2m{+}d{-}1$ such that $\| \hat \pi_t - q_t\|_{L^2} \leq C_\mathrm{TT} \, \sqrt{2m+d-1} \, \varepsilon_t$. Then, by the Cauchy--Schwartz inequality, we have 
\[
	\| \hat \pi_t - q_t\|_{L^1}  \leq C_\mathrm{TT} \sqrt{\Omega(\mathcal{X}\times\Uptheta\times \mathcal{X})} \, \|\hat \pi_t - q_t\|_{L^2} ,
\]
and thus the result follows. 
\end{proof}

Proposition \ref{prop:tt_error} only considers bounded spaces. This can be easily extended to unbounded spaces by considering a weight measure and a slightly modified approximation scheme. See \cite{cui2021deep,cui2023self} and references therein for details. The result of \cite{griebel2021analysis} establishes sufficient conditions of the {\it a priori} error bound of TT decompositions using smoothness. In practice, there are many examples of low-rank functions that do not satisfy the smoothness assumption, e.g., products of discontinuous univariate functions. To the authors' knowledge, the analysis of TT ranks for general functions is still an active research area and the results of \cite{griebel2021analysis} may not provide a precise practical guideline for building TT decompositions. Nonetheless, Proposition \ref{prop:tt_error} provides some confidence in applying TT decompositions in the sequential state and parameter learning problems. 

It is worth mentioning that the error controlling factor $\varepsilon_t$ in Proposition \ref{prop:tt_error} determines the TT ranks and hence the computational complexity for building a TT decomposition. Proposition \ref{prop:tt_error} suggests that with a prescribed  $\varepsilon_t$, the error of a TT decomposition may increase with the dimensionality of the problem. For a class of high-dimensional problems where the Sobolev norm is equipped with suitably decaying weights \citep{dick2013high,griebel2021analysis,sloan1998quasi}, one can obtain a dimension-free error bound. Although such dimension-free bounds are not discussed here, extending Proposition \ref{prop:tt_error} to such weighted Sobolev norms is trivial. 

\begin{corollary} \label{coro:tt_thm}
	Suppose either bound of Assumption \ref{assumption:propagation} holds and the smoothness of the densities $f$ and $g$ fulfil the requirement of  Proposition \ref{prop:tt_error}. In each step of Alg.~\ref{alg:basic}, suppose further the ranks of a TT approximation is chosen according to Proposition \ref{prop:tt_error} so that it satisfies 
	\[
		\|\hat \pi_t -  q_t\|_{L^1} \leq C_\mathrm{TT} \sqrt{\Omega(\mathcal{X}\times\Uptheta\times \mathcal{X}) (2m+d-1)} \, \varepsilon_t
	\]
	for some error controlling factor $\varepsilon_t \in (0, 1)$. Then, for $t > 0$, the total error satisfies 
	\[
	\| \hat \pi_t -  \pi_t\|_{L^1} \leq  C_\mathrm{TT} \sqrt{\Omega(\mathcal{X}\times\Uptheta\times \mathcal{X}) (2m+d-1)} \sum_{k=1}^{t} C^{t-k}  \varepsilon_k,	
	\]
	where $C = \sup_t C_t^{(h)}$ with $h\in\{f,g\}$.
\end{corollary}
	
\begin{proof}
	The result is a direct consequence of Theorem \ref{thm1}.
\end{proof}

In Alg.~\ref{alg:stt}, the previous marginal density $\hat \pi( \bx_{\tminusone}, \bth | \by{}_{1:\tminusone})$ can be expressed as a sum of squares of multivariate polynomial functions (each in the TT form) to ensure the non-negativity together with a defensive term (cf. Proposition \ref{prop:marginal}). Applying the multivariate Faa di Bruno formula \citep{constantine1996multivariate}, one can show that the square root of the previous marginal density $\sqrt{\hat \pi( \bx_{\tminusone}, \bth | \by{}_{1:\tminusone})}$ yields a bounded Sobolev norm for a finite order of differentiability $K$. However, the bound on the Sobolev norm may increase with $K$. Although Theorem 4 of \cite{griebel2021analysis} may still apply in this case for a finite $K$, it requires further analysis to understand the impact of increasing Sobolev norm on the {\it a priori} error analysis of the TT decomposition. In our numerical examples (Section~\ref{sec:kalman}), we demonstrate that Alg.~\ref{alg:basic} and \ref{alg:stt} achieve comparable approximation accuracy. In the rest, we focus on analyzing the impact of the additional defensive term used in Alg.~\ref{alg:stt}. We bound the approximation error $\lerr{\surd \hat \pi_t - \surd q_t}$ in Proposition \ref{prop:L2_approximation}---we omit the proof as it is a direct consequence of Lemma~\ref{lemma:sirt_error}. Then in Corollary~\ref{coro:stt_thm}, we apply Theorem~\ref{thm2} to bound the total Hellinger error of Alg.~\ref{alg:stt} to complete the analysis.

\begin{proposition}\label{prop:L2_approximation}
		In each iteration of Alg.~\ref{alg:stt}, suppose a TT decomposition $\phi_t$ is constructed according to Proposition \ref{prop:tt_error} such that its $L^2$ error satisfies 
		\[
		\| \phi_t - \surd q_t \|_{L^2} < C_\mathrm{TT} \sqrt{\Omega(\mathcal{X}\times\Uptheta\times \mathcal{X}) (2m+d-1)} \, \varepsilon_t
		\]
		for some error controlling factor $\varepsilon_t\in (0, 1)$. Suppose further the constant $\tau_t$ in the defensive term satisfies
		\(
		\tau_t \leq \lerr{\phi_t - \surd q_t}^2.
		\)
		Then, the approximation error satisfies $\lerr{\surd \hat \pi_t - \surd q_t} \leq C_\mathrm{TT} \sqrt{2\, \Omega(\mathcal{X}\times\Uptheta\times \mathcal{X}) (2m+d-1)} \, \varepsilon_t$.
\end{proposition}

\begin{corollary} \label{coro:stt_thm}
		Suppose Proposition~\ref{prop:L2_approximation} and either bound of Assumption~\ref{assumption:propagation} hold. For $t > 0$, the total Hellinger distance between $ p_t $ and its approximation $ \hat p_t $ in Alg.~\ref{alg:stt} satisfies
		\[
			\dhd{\hat p_t}{p_t} \leq \frac{2 C_\mathrm{TT} \sqrt{\Omega(\mathcal{X}\times\Uptheta\times \mathcal{X}) (2m+d-1)}}{\sqrt{p(\by{}_{1:t})}} \sum_{k=1}^{t} C^{(t-k)/2} \varepsilon_k,
		\]
		where $C = \sup_t C_t^{(h)}$ with $h\in\{f,g\}$ and $p(\by{}_{1:t})$ is the evidence.
\end{corollary}

\section{Preconditioning methods}\label{sec:precondition}

The key step of Alg.~\ref{alg:stt} is to approximate the square root of the non-separable, unnormalized density $q_t$ by a TT decomposition $\phi_{q,t}$. In many applications, the temporally increasing data size and complex nonlinear interactions among parameters and states may concentrate posterior densities  to some submanifold, and hence lead to potentially high ranks in $\phi_{q,t}$. 
This can make the TT-based algorithms computationally demanding. Rather than directly approximating $q_t$, here we present a preconditioning framework to improve TT's approximation efficiency and discuss how to apply the resulting preconditioned approximations in our sequential estimation algorithms (Alg.~\ref{alg:stt}--\ref{alg:stt_smooth}). 

\subsection{General framework}

Our preconditioning procedure is guided by a (possibly unnormalized) bridging density $\rho_t(\bx_t, \bth, \bx_{\tminusone})$ that is easier to approximate than $q_t(\bx_t, \bth, \bx_{\tminusone})$. We introduce general reference random variables $(\bU_t, \bU_\theta, \bU_{\tminusone})$ with tensor-product normalized density $\eta(\bu_t, \bu_\theta, \bu_{\tminusone})$ $= \eta(\bu_t)\eta(\bu_\theta)\eta(\bu_{\tminusone})$,  where $(\bU_t, \bU_{\tminusone})$ take values in $\R^m$ and $\bU_\theta$ takes values $\R^d$. The preconditioning procedure has the following conceptual steps.

\begin{enumerate}[leftmargin=18pt]
    \item \textbf{Change of coordinates.} We first construct a \emph{preconditioning KR rearrangement} $\mathcal{T}_t$ such that 
    \begin{align}
        (\mathcal{T}_t)_\sharp  \, \rho_t(\bu_t, \bu_\theta, \bu_{\tminusone}) & = \rho_t\big(\mathcal{T}_t^{-1}(\bu_t, \bu_\theta, \bu_{\tminusone})\big) \big|\nabla \mathcal{T}_t^{-1}(\bu_t, \bu_\theta, \bu_{\tminusone})\big| \nonumber \\
        & \propto \eta(\bu_t, \bu_\theta, \bu_{\tminusone}).\label{eq:pushforward_1}
    \end{align}
    The map $\mathcal{T}_t$ defines a change of coordinates from $(\bx_t, \bth, \bx_{\tminusone})$ to $(\bu_t, \bu_\theta, \bu_{\tminusone})$. 
    \item \textbf{Preconditioning.} 
    Applying the identity in \eqref{eq:pushforward_1}, the pushforward of $q_t$ under $\mathcal{T}_t$ takes form 
    \begin{align}
        (\cT_t)_{\sharp} \,q_t(\bu_t, \bu_\theta, \bu_{\tminusone}) & = q_t \big(\mathcal{T}_t^{-1}(\bu_t, \bu_\theta, \bu_{\tminusone}) \big) \big|\nabla \mathcal{T}_t^{-1}(\bu_t, \bu_\theta, \bu_{\tminusone}) \big| \nonumber \\ 
        & \propto q_{\sharp,t}(\bu_t, \bu_\theta, \bu_{\tminusone}) ,\label{eq:pushforward_2}
    \end{align}
    where
    \begin{equation}
        q_{\sharp,t}(\bu_t, \bu_\theta, \bu_{\tminusone} ) := \frac{q_t\big(\cT_t^{-1}(\bu_t, \bu_\theta, \bu_{\tminusone}) \big) }{\rho_t\big(\cT_t^{-1}(\bu_t, \bu_\theta, \bu_{\tminusone})\big)} \eta(\bu_t, \bu_\theta, \bu_{\tminusone})
        \label{eq:pushforward_3}
    \end{equation}
    is a non-negative function that can be evaluated pointwise. The pushforward density $q_{\sharp,t}$ can be viewed as the reference density $\eta$ perturbed by the ratio $q_t/\rho_t$ in the transformed coordinates $(\bu_t, \bu_\theta, \bu_{\tminusone})$. With a suitable bridging density, the ratio $q_t/\rho_t$ is significantly less concentrated, and hence may be easier to approximate. See Fig.~\ref{fig:precond} for an illustration.

    \item \textbf{TT-approximation.} By approximating $\surd q_{\sharp,t}$ using a TT $\phi_{\sharp,t}$, i.e., 
    \begin{equation}
        \sqrt{ q_{\sharp,t}(\bu_t, \bu_\theta, \bu_{\tminusone} )} \approx \phi_{\sharp,t}(\bu_t, \bu_\theta, \bu_{\tminusone} ), 
    \end{equation}
    we can follow the squared-TT approximation outlined in Section \ref{sec:stt} to define an unnormalized approximate density 
    \begin{equation}\label{eq:second_stt}
        \hat \nu_{\sharp,t} (\bu_t, \bu_\theta, \bu_{\tminusone} ) := \phi_{\sharp,t}(\bu_t, \bu_\theta, \bu_{\tminusone} )^2 + \tau_{\sharp,t} \eta(\bu_t, \bu_\theta, \bu_{\tminusone} )
    \end{equation}
    where $\tau_{\sharp,t} \leq \lerr{\phi_{\sharp,t} - \surd q_{\sharp,t}}^2$. Applying Proposition \ref{prop:marginal} and denoting the normalized density by $\hat \mu_{\sharp,t} \propto \hat \nu_{\sharp,t}$, we obtain a KR rearrangement $\cS_t$ such that
    \[
        (\cS_t)_\sharp \, \hat \mu_{\sharp,t} (\bxi_t, \bxi_\theta, \bxi_{\tminusone}) = \mathrm{uniform}(\bxi_t, \bxi_\theta, \bxi_{\tminusone}).
    \]
    \item \textbf{Composition.} The above steps define a composite transformation $\cS_t \circ \cT_t$ that approximately pushes forward the normalized version of the density $q_t(\bx_t, \bth, \bx_{\tminusone})$ to a uniform density. Equivalently, we have 
    \begin{align}
        q_t(\bx_t, \bth, \bx_{\tminusone}) \approxprop & \; (\cS_t \circ \cT_t)^\sharp \mathrm{uniform}(\bx_t, \bth, \bx_{\tminusone}). 
    \end{align}
    Thus, the pullback of the uniform density under $\cS_t \circ \cT_t$ defines a normalized approximate posterior density, which takes the form
    \begin{align}
        \hat p(\bx_t, \bth, \bx_{\tminusone} | \by{}_{1:t}) & =  (\cS_t \circ \cT_t)^\sharp \mathrm{uniform}(\bx_t, \bth, \bx_{\tminusone}) \nonumber \\
        & = |\nabla (\cS_t \circ \cT_t) (\bx_t, \bth, \bx_{\tminusone})| \nonumber \\
        & \refover{eq:pushforward_1}{=} \, \frac{\hat \mu_{\sharp,t} \big(\cT_t(\bx_t, \bth, \bx_{\tminusone} )\big)} {\eta\big(\cT_t(\bx_t, \bth, \bx_{\tminusone} )\big)} \rho_t(\bx_t, \bth, \bx_{\tminusone}) \nonumber \\
        & = \frac1{\hat z_t} \hat \pi(\bx_t, \bth, \bx_{\tminusone} | \by{}_{1:t}) ,
    \end{align}
    where
    \begin{equation}\label{eq:precon_approx}
        \hat \pi(\bx_t, \bth, \bx_{\tminusone} | \by{}_{1:t}) = \frac{\hat \nu_{\sharp,t}\big(\cT_t(\bx_t, \bth, \bx_{\tminusone} )\big)}{\eta\big(\cT_t(\bx_t, \bth, \bx_{\tminusone} ) \big)}  \rho_t(\bx_t, \bth, \bx_{\tminusone})
    \end{equation}
    is the unnormalized approximate posterior density and 
    \begin{align*}
        \hat z_t  = \!\int\!  \hat \pi(\bx_t, \bth, \bx_{\tminusone} | \by{}_{1:t})\dd \bx_t \dd \bth \dd\bx_{\tminusone} = \!\int\! \hat \nu_{\sharp,t}(\bu_t, \bu_\theta, \bu_{\tminusone} )\dd \bu_t \dd \bu_\theta \dd\bu_{\tminusone},
    \end{align*}
    is the normalizing constant. The last equality follows from the change of coordinates defined by $\mathcal{T}_t$.
    \end{enumerate}
There are many ways to defining the bridging density $\rho_t$ and building the preconditioning map $\mathcal{T}_t$. In the following, we provide some examples of bridging densities relying on the particular structure of $q_t$, which is defined in \eqref{eq:sqrt_0}.

\begin{figure}[h]
    \centering
    \hfill
    \includegraphics{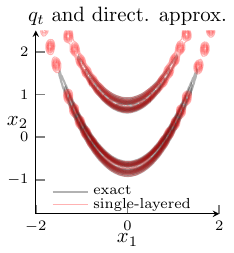}\hfill
    \includegraphics{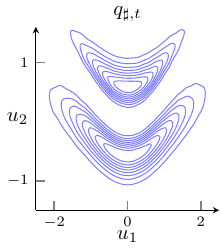}\hfill
    \includegraphics{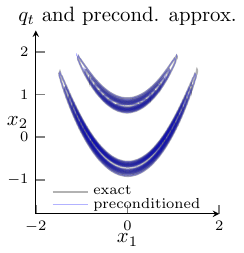}\hfill\vspace{-9pt}
    \caption{Left: approximation of $ q_t $ (gray contours) using a single layer of TT-based approximation (red contours), in which Fourier basis of order 30 and a TT rank 24 is used. Middle: preconditioned density using the tempering technique (cf. Section \ref{sec:nonlinear_precond}), in which the nonlinear preconditioning transform is defined by a TT with Fourier basis of order 30 and a rank 12. Right: approximation of $ q_t $ (gray contours) using the preconditioned approximation (blue contours), in which Fourier basis of order 30 and a TT rank 12 is used in the preconditioned approximation. }\label{fig:precond}
\end{figure}

\subsection{Gaussian bridging and linear preconditioning}\label{sec:linear_precond}

We consider a Gaussian approximation to the density $q_t$ as the bridging density, i.e., $\rho_t(\cdot):=\mathcal{N}(\cdot;\boldsymbol{\mu}_t, \boldsymbol{\Sigma}_t)$, where $\boldsymbol{\mu}_t \in \R^{2m + d}$ is the mean vector and $\boldsymbol{\Sigma}_t \in \R^{(2m + d)\times(2m+d)}$ is the covariance matrix. In each iteration, we can draw random variables from the previous marginal approximation $\hat \pi(\bth, \bx_{\tminusone} | \by{}_{1:\tminusone})$ and use one step of particle filter---e.g., a bootstrap filter---to estimate $(\boldsymbol{\mu}_t, \boldsymbol{\Sigma}_t)$ using particles. With $\rho_t(\cdot):=\mathcal{N}(\cdot;\boldsymbol{\mu}_t, \boldsymbol{\Sigma}_t)$, we compute the Cholesky factorization of $\boldsymbol{\Sigma}_t$ to obtain $\mL_t \mL_t{}^\top = \boldsymbol{\Sigma}_t$, where $\mL_t$ is lower-triangular. This defines a lower-triangular linear KR rearrangement $\mathcal{T}^l_t(\bx_t, \bth, \bx_{\tminusone}) =\mL_t^{-1} \big([\bx_t, \bth, \bx_{\tminusone}]^\top - \boldsymbol{\mu}_t \big)$ that transforms $(\bX_t, \bTh, \bX_{\tminusone}) \sim \rho_t$ into reference variables $(\bU_t, \bU_\theta, \bU_{\tminusone})$ following a tensor-product zero mean standard Gaussian density $\eta$. The upper-triangular linear KR rearrangement can also be derived by permuting the variables. 

\subsection{Tempering and nonlinear preconditioning}\label{sec:nonlinear_precond}

We employ the tempering idea, e.g., 
\cite{beskos2016convergence,gelman1998simulating,herbst2019tempered} and \cite{kantas2014sequential}, to build an unnormalized bridging density 
\begin{equation} \label{eq:bridging}
    \rho_t(\bx_t, \bth, \bx_{\tminusone}) = \hat \pi(\bth, \bx_{\tminusone} |\by_{1:t-1})^{\beta_\pi} \tr{t}^{\beta_f} \lk{t}^{\beta_g} ,
\end{equation}
with some constants $\beta_\pi, \beta_f, \beta_g \in [0,1]$. When $\beta_\pi = \beta_f= \beta_g = 1$, it recovers the target density $q_t$. By choosing appropriate constants $(\beta_\pi, \beta_f, \beta_g)$---for example, using the adaptation strategy discussed in \cite{beskos2016convergence} and \cite{kantas2014sequential}---to construct $\rho_t$ that is less concentrate than $q_t$ but retains some features of $q_t$. 
We can apply the squared approximation procedure to decompose $\surd \rho_t$ into a TT $\phi_{\rho,t}$, approximate $\rho_t$ by 
\begin{equation}\label{eq:pre_stt}
\hat \rho_t(\bx_t, \bth, \bx_{\tminusone}) := \phi_{\rho,t}(\bx_t, \bth, \bx_{\tminusone})^2 + \tau_{\rho,t} \lambda(\bx_t, \bth, \bx_{\tminusone}),
\end{equation}
with some $\tau_{\rho,t} > 0$, and derive the corresponding KR rearrangement $\mathcal{R}_t$ such that
\[
	(\mathcal{R}_t)_\sharp \hat \rho_t (\bxi_t, \bxi_\theta, \bxi_{\tminusone}) \propto \mathrm{uniform}(\bxi_t, \bxi_\theta, \bxi_{\tminusone}). 
\]
Given a diagonal map $\mathcal{D}$ such that 
\(
    \mathcal{D}_\sharp\, \eta(\bxi_t, \bxi_\theta, \bxi_{\tminusone}) = \mathrm{uniform}(\bxi_t, \bxi_\theta, \bxi_{\tminusone}),
\)
composing $\mathcal{D}^{-1}$ and $\mathcal{R}_t$, we obtain $\mathcal{T}_t = \mathcal{D}^{-1} \circ \mathcal{R}_t$ such that 
\begin{equation}
    (\mathcal{T}_t)_\sharp \hat \rho_t (\bu_t, \bu_\theta, \bu_{\tminusone}) \propto \eta(\bu_t, \bu_\theta, \bu_{\tminusone}). 
\end{equation}

In this construction, the preconditioned target density $q_{\sharp,t}$---which is originally defined in \eqref{eq:pushforward_3}---becomes
\begin{equation}
q_{\sharp,t}(\bu_t, \bu_\theta, \bu_{\tminusone} ) := \frac{q_t\big(\cT_t^{-1}(\bu_t, \bu_\theta, \bu_{\tminusone}) \big) }{\hat \rho_t\big(\cT_t^{-1}(\bu_t, \bu_\theta, \bu_{\tminusone})\big)} \eta(\bu_t, \bu_\theta, \bu_{\tminusone}),
\label{eq:pushforward_4}     
\end{equation}
where the exact bridging density $\rho_t$ is replaced with $\hat \rho_t$, which is the pushforward of the reference density under the preconditioning map $\mathcal{T}_t$. As a consequence, the unnormalized approximate density in \eqref{eq:precon_approx} becomes 
\begin{equation}\label{eq:precon_approx_2}
    \hat \pi(\bx_t, \bth, \bx_{\tminusone} | \by{}_{1:t}) = \frac{\hat \nu_{\sharp,t}\big(\cT_t(\bx_t, \bth, \bx_{\tminusone} )\big)}{\eta\big(\cT_t(\bx_t, \bth, \bx_{\tminusone} )\big)}  \hat \rho_t(\bx_t, \bth, \bx_{\tminusone}).
\end{equation}

\subsection{Marginalization and conditional maps} \label{sec:marginal_precond}

After building the preconditioned approximation, we need to integrate the approximate density $\hat \pi(\bx_t, \bth, \bx_{\tminusone} | \by{}_{1:t})$ over $\bx_{\tminusone}$, so that the marginal density 
\[
	\hat \pi(\bx_t, \bth | \by{}_{1:t}) = \int \hat \pi(\bx_t, \bth, \bx_{\tminusone} | \by{}_{1:t}) \dd \bx_{\tminusone}
\]
can be used in the next iteration of Alg.~\ref{alg:stt}. 
In addition, we also need to derive the lower and upper conditional KR rearrangements to be used in path estimation and particle filtering, respectively. We use the lower-triangular KR rearrangements $\mathcal{S}^l_t$ and $\mathcal{T}^l_t$, which respectively take the form
\[	
    \mathcal{S}^l_t(\bu_t, \bu_\theta, \bu_{\tminusone}) =  
    \begin{bmatrix*}[l]  \mathcal{S}^l_{t,t}\hspace{6pt}(\bu_{t}) \vspace{4pt}\\ \mathcal{S}^l_{t,\theta}\hspace{5pt}(\bu_\theta & \hspace{-11pt} | \bu_{t})  \vspace{4pt}\\ \mathcal{S}^l_{t,\tminusone}(\bu_{\tminusone} & \hspace{-11pt} | \bu_{t},\bu_\theta)  \end{bmatrix*} \quad \text{and} \quad
    \mathcal{T}^l_t(\bx_t, \bth, \bx_{\tminusone}) =  
    \begin{bmatrix*}[l]  \mathcal{T}^l_{t,t}\hspace{6pt}(\bx_{t}) \vspace{4pt}\\ \mathcal{T}^l_{t,\theta}\hspace{5pt}(\bth & \hspace{-11pt} | \bx_{t})  \vspace{4pt}\\ \mathcal{T}^l_{t,\tminusone}(\bx_{\tminusone} & \hspace{-11pt} | \bx_{t},\bth)  \end{bmatrix*} ,
\]
as well as nonlinear preconditioning defined above, to outline how to carry the marginalizations and conditional sampling.

Because the composition  $\mathcal{S}^l_t \circ \mathcal{T}^l_t$ is  also lower triangular, we can apply a similar derivation to that of Proposition \ref{prop:backward} to compute the marginal density $\hat \pi(\bx_t, \bth | \by{}_{1:t})$. The necessary steps for computing $\hat \pi(\bx_t, \bth | \by{}_{1:t})$ is integrated with the preconditioned density approximation procedure, and thus we summarize them together in Alg.~\ref{alg:precond}, which replace steps (b) and (c) of the sequential estimation algorithm (Alg.~\ref{alg:stt}).

\begin{myalg}{Preconditioned replacements for steps (b) and (c) of the sequential estimation algorithm (Alg.~\ref{alg:stt}).}\label{alg:precond}
    \begin{enumerate}[wide=0pt,leftmargin=18pt,labelsep=9pt]
        \item[(b.1)\hspace{-3pt}] Approximate the bridging density $\rho_t(\bx_t, \bth, \bx_{\tminusone})$ by the TT-based approximate density $\hat \rho_t(\bx_t, \bth, \bx_{\tminusone})$ defined in \eqref{eq:pre_stt}.

        \item[(b.2)\hspace{-3pt}] Integrate $\hat \rho_t(\bx_t, \bth, \bx_{\tminusone})$ from the right variable $x_{\tminusone,m}$ to the left variable $x_{t,1}$ using Proposition \ref{prop:marginal} to define lower-triangular preconditioning map $\mathcal{T}^l_t$, the lower conditional map $\mathcal{T}^l_{t,\tminusone}(\bx_{\tminusone} | \bx_{t},\bth)$, and the marginal density $\hat \rho_t(\bx_t, \bth)$.

        \item[(b.3)\hspace{-3pt}] Approximate the pushforward density $q_{\sharp,t}(\bu_t, \bu_\theta, \bu_{\tminusone})$ by the TT-based approximation $\hat \nu_{\sharp,t}(\bu_t, \bu_\theta, \bu_{\tminusone})$ defined in \eqref{eq:second_stt}.

        \item[(c.1)\hspace{-3pt}] Integrate the last block of  $\hat \nu_{\sharp,t}(\bu_t, \bu_\theta, \bu_{\tminusone})$ from the right variable $u_{\tminusone,m}$ to the left variable $u_{\tminusone,1}$ using Proposition \ref{prop:marginal} to define the marginal density $\hat \nu_{\sharp,t}(\bu_t, \bu_\theta)$ and the lower conditional map $\mathcal{S}^l_{t,\tminusone}(\bu_{\tminusone} | \bu_{t},\bu_\theta)$. 
        \item[(c.2)\hspace{-3pt}] Using the lower-triangular transformation $(\bu_t, \bu_\theta) = (\mathcal{T}^l_{t,t}(\bx_{t}), \mathcal{T}^l_{t,\theta}(\bth | \bx_{t}) )$, we have the marginal density 
        \[
        \hat \pi(\bx_t, \bth | \by{}_{1:t}) = \frac{\hat \nu_{\sharp,t}\big(\mathcal{T}^l_{t,t}(\bx_{t}), \mathcal{T}^l_{t,\theta}(\bth | \bx_{t}) \big)}{\eta\big(\mathcal{T}^l_{t,t}(\bx_{t}),\mathcal{T}^l_{t,\theta}(\bth | \bx_{t}) \big)}  \hat \rho_t(\bx_t, \bth),
        \]
        which will be used in the next iteration of the sequential estimation.
    \end{enumerate}
\end{myalg}

As a byproduct, steps (b.2) and (c.1) of Alg.~\ref{alg:precond} also define the last block of the composite map $\mathcal{S}^l_t \circ \mathcal{T}^l_t$, which can be expressed as
\begin{equation}
    \mathcal{F}^l_{t,\tminusone}(\bx_{\tminusone} | \bx_t, \bth) :=\mathcal{S}^l_{t,\tminusone}\big(\mathcal{T}^l_{t,\tminusone}(\bx_{\tminusone} | \bx_{t},\bth) \big| \mathcal{T}^l_{t,t}(\bx_{t}),\mathcal{T}^l_{t,\theta}(\bth|\bx_t)\big).
\end{equation}
This precisely gives the lower conditional map for the path estimation (Alg.~\ref{alg:stt_smooth}). 


\section{Numerical results}\label{sec:numerics}

We provide several numerical examples to demonstrate the efficiency of our TT-based methods. 
These include a linear Kalman filter controlled by unknown parameters, in which the posterior density of the parameters has an analytical form; 
a stochastic volatility model (cf. Example~\ref{eg:sv}) commonly used in the literature for benchmarking sequential algorithms; a high-dimensional compartmental susceptible-infectious-removed model following the Austrian state adjacency map; and a dynamical system modelling the interaction of a predator-prey system. 
We also compare our method to SMC$^2$ when it is applicable. 
For all the numerical examples reported here, we compute the $25\%, 50\%, 75\%$ quantiles of the effective sample sizes using $40$ repeated experiments from smoothing samples obtained in Alg.~\ref{alg:stt_smooth}.

\subsection{Linear Kalman filter with unknown parameters}\label{sec:kalman}

\textbf{Setup.} Our first example considers a linear Kalman filter with unknown parameters, in which the state process and the observation process take the form:
\begin{equation} \label{eq:kalman filter}
	\left\{ 
		\begin{array}{rl}
			\bX_t - \mu \; = \! & b \left(\bX_{\tminusone}-\mu\right) + a \varepsilon^{(x)}_{t}\\
			\bY_t \; = \! & \mC \bX_t + d \varepsilon^{(y)}_{t}
		\end{array}
	\right. ,
\end{equation}
where $\varepsilon^{(x)}_t$ and $\varepsilon^{(y)}_t$ follow independent Gaussian distributions $\mathcal{N}(0,\mI_m)$ and $\mathcal{N}(0, \mI_n)$, and $\mC \in \R^{n \times m}$ is the observation matrix. Here $\mI_{m}$ represents an $m\times m$ identity matrix. The prior density of the initial state $ \bX_0 $ is given by $p(\bx_0|\mu) := \mathcal{N}(\bx_0;\mu \mathbf{1}, \mI_{m})$, where $\mathbf{1}$ is the $m$-dimensional vector filled with ones. The parameters controlling the model are collected as $ \bth = (\mu,a,b, d)$.

Conditioned on the parameters, the joint state filtering density $ p(\bx_{0:t} | \bth, \by_{1:t}) $ follows a multivariable Gaussian distribution, which can be derived from the classical Kalman filter. Integrating over the states, we obtain an explicit form of the posterior parameter density  $ p( \bth | \by_{1:t}) = \int  p(\bx_{0:t} , \bth | \by_{1:t}) \dd \bx_{0:t} $. 
Although the posterior parameter density cannot be directly sampled, we can evaluate the density pointwisely for any given parameter $\bth$, and thus we can use the model in \eqref{eq:kalman filter} to benchmark the performance of various estimation algorithms---for example, this enables us to estimate the Hellinger distance between the exact posterior parameter density $p(\bth| \by_{1:t})$ and its approximation $\hat p( \bth| \by_{1:t})$ obtained by Alg.~\ref{alg:stt}. 
For the sake of completeness, we include the derivation of the posterior parameter density  $ p( \bth | \by_{1:t}) $ in Appendix \ref{appendix:kalman}.

In our numerical experiment, we take $ m $ and $ n $ to be three and let the total number of time steps $ T $ be 50. Additionally, we set $\mu=0$ and set the observation matrix $ \mC $ to be a prescribed random value. We impose the condition $a^{2}+b^{2}=1$ such that the state process $ \bX_t $ is stationary. 
Thus, the parameters to be estimated are reduced to $ \bth = (a, d) $. We impose a uniform prior density $p(\bth) := \mathrm{uniform}(\bth; [0.4,1]^2)$ on the parameters, and generate the synthetic data $\bY_t$ using $a=0.8$ and $d=0.5$ in the numerical experiments. When building TT approximations, we transform the parameters to an unbounded domain using the inverse distribution function of the standard Gaussian distribution to facilitate the linear preconditioning techniques in Sec.~\ref{sec:linear_precond}.

\textbf{Comparison of Alg.~\ref{alg:basic} and \ref{alg:stt}.} We start with the comparison between Alg.~\ref{alg:basic} and \ref{alg:stt}. For both algorithms, we use 5 alternating least square (ALS) iterations to construct TT decompositions. Note that using two ALS iterations is sufficient in most of the situations presented here. We use a rather large number of ALS iterations to eliminate possible error sources in our benchmarks. The TT decomposition used here employs a piecewise Lagrange basis function defined by four subintervals and polynomials with order eight. We consider maximum TT ranks $r \in \{10, 20, 30\}$ in this comparison.

Following the discussion of Section~\ref{sec:error}, we use the relative $L^1$ error of posterior parameter densities, $\|\hat \pi(\bth|\by_{1:t}) - \pi(\bth|\by_{1:t})\|_{L^1}/\|\pi(\bth|\by_{1:t})\|_{L^1}$, to benchmark the accuracy of these TT-based algorithms. Note that the denominator $\|\pi(\bth|\by_{1:t})\|_{L^1}$ is also the normalizing constant of the posterior density. Fig.~\ref{fig:alg12} shows the change of the relative $L^1$ error over time. 
For both algorithms, we observe that the error accumulations behave similarly, and relative errors reduce with increasing TT ranks. For all TT ranks, Alg.~\ref{alg:stt} is more accurate than Alg.~\ref{alg:basic}. In the other numerical experiments, we only demonstrate the performance of Alg.~\ref{alg:stt} and its accompanying sampling algorithms.

\begin{figure}[h]
	\centering
	\includegraphics{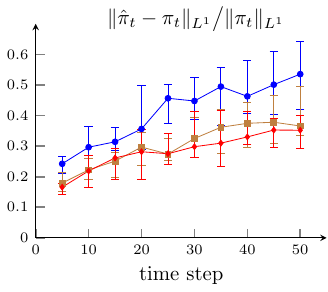}
	\includegraphics{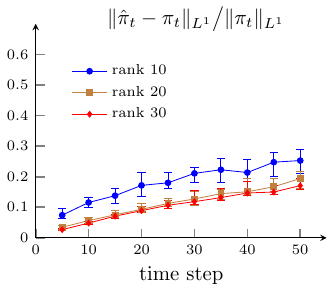}\vspace{-9pt}
	\caption{Linear Kalman filter with unknown parameters. The relative $ L^1 $ error of approximations of posterior parameter densities built by Alg.~\ref{alg:basic} and \ref{alg:stt}.}
	\label{fig:alg12}
\end{figure}

\textbf{Demonstration of Alg.~\ref{alg:stt} and debiasing.} 
Since the non-negativity-preserving Alg.~\ref{alg:stt} is the workhorse of this paper, here we thoroughly demonstrate its accuracy and the sampling performance of the accompanying path estimation algorithm (Alg.~\ref{alg:stt_smooth}) with various algorithmic settings. In all experiments, we use five ALS iterations to construct TT decompositions and piecewise basis functions defined on four subintervals with Lagrange polynomials. We denote the number of degrees of freedom of the basis functions by $\ell$.

\begin{figure}[h]
	\centering\hfill
	\includegraphics{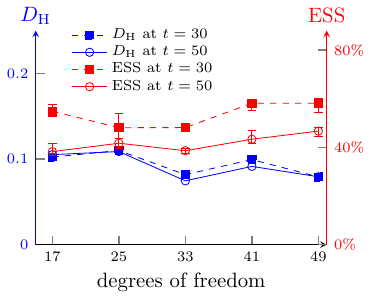}\hfill
	\includegraphics{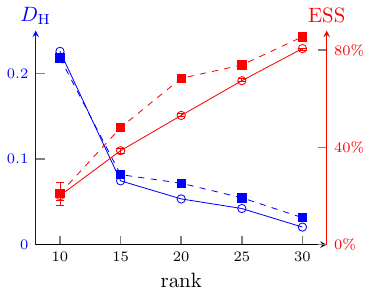}\hfill\vspace{-9pt}
	\caption{Linear Kalman filter with unknown parameters. The Hellinger distances (blue lines) between the theoretical posterior parameter density and its TT approximation, and the ESS (red lines) for the joint posterior of the states and the parameters at time $ t \in \{30, 50\}$. Left: changing the number of degrees of freedom of the basis functions $\ell$ with a fixed maximum TT rank $r = 15$. Right: changing the maximum TT rank $r$ with a fixed number of degrees of freedom of the basis functions $\ell = 33$.}
	\label{fig:kalman_rank}
\end{figure}

In the first set of numerical experiments, we vary the number of degrees of freedom of the piecewise basis functions $ \ell $ and the maximum TT rank $ r $ to investigate their impacts on the accuracy of TT-based methods. We first fix the maximum TT rank $r$ to be $15$ and vary $\ell = \{17, 25, 33, 41, 49\}$ by fixing the number of subintervals to be four and increasing the order of the Lagrange polynomials. For each $\ell$, we run Alg.~\ref{alg:stt} for $ t $ up to 50 and apply the path estimation (Alg.~\ref{alg:stt_smooth}) using $N=1000$ sample paths at $t \in \{30, 50\}$.  
For each of the approximate posterior parameter densities, we report the Hellinger distance and the effective sample size (ESS)---for the joint density $p(\bth, \bx_{0:t} | \by{}_{1:t})$---at $t \in \{30, 50\}$ in the left plot of Fig.~\ref{fig:kalman_rank}. 
We observe that increasing the number of degrees of freedom of the basis functions only marginally improves the accuracy. 
Then we fix the number of degrees of freedom $\ell$ to be $33$ with order-eight Lagrange polynomials on four subintervals, and vary the maximum TT rank $r \in \{10, 15, 20, 25, 30\}$. 
The resulting Hellinger distances and ESSs are shown in the right plot of Fig.~\ref{fig:kalman_rank}. In this case, we observe a significant error reduction with increasing TT ranks.

\begin{figure}[h]
	\centering\hfill
	\includegraphics[scale = 1.1]{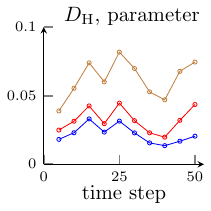}\hfill
	\includegraphics[scale = 1.1]{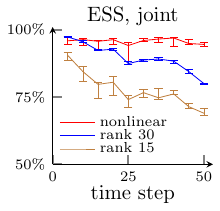}\hfill
	\includegraphics[scale = 1.1]{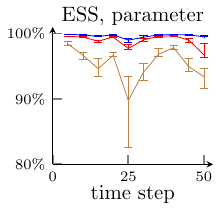}\hfill\vspace{-9pt}
	\caption{Linear Kalman filter with unknown parameters. Left: the Hellinger distance between the theoretical posterior parameter density $ p(\bth|\by_{1:t})$ and its TT approximations versus time. Middle: the ESS of the path estimation (Alg.~\ref{alg:stt_smooth}) for sampling the joint posterior $p(\bth, \bx_{0:t}|\by_{1:t})$ versus time. Right: the ESS for sampling the posterior parameter density $p(\bth|\by_{1:t})$ versus time.}
	\label{fig:kalman}
\end{figure}

In the second set of numerical experiments, we demonstrate the impact of preconditioning techniques. 
Here we consider two approximation ansatzes with $ (\ell, r) = (33, 15) $ and $(\ell, r) =  (33, 30) $ using linear preconditioning (cf. Section~\ref{sec:linear_precond}) and another ansatz $ (\ell, r) = (33, 15) $ using nonlinear preconditioning (cf. Section~\ref{sec:nonlinear_precond}). 
For nonlinear preconditioning, we use a standard multivariate Gaussian for the reference density $ \eta $ and $ \beta_\pi = \beta_f = \beta_g = 0.4 $ to construct the bridging density. For each $t \in \{5, 10, \ldots, 50\}$, we report the Hellinger distance, the ESS of the path estimation (Alg.~\ref{alg:stt_smooth}) for the joint density, and the ESS of the approximate posterior parameter density. The results are shown in Fig.~\ref{fig:kalman}. 
In the linear preconditioning case, we observe that the approximation ansatz $ (\ell, r) =  (33, 30) $ outperforms the approximation ansatz $ (\ell, r) =  (33, 15) $ at all time steps, which is expected. 
We also observe that the performance of the nonlinear preconditioning technique using the approximation ansatz $(\ell, r) = (33, 15)$ is comparable with that of linear preconditioning with a higher-rank approximation ansatz $(\ell, r) = (33, 30)$. For both $(\ell, r) = (33, 15)$ using nonlinear preconditioning and $(\ell, r) = (33, 30)$ using linear preconditioning, the Hellinger distance remains below $0.05$ for all time steps, while the ESSs for the joint densities and the marginal parameter densities are respectively around 80\% and 98\% after 50 time steps. These results provide strong evidence of the reliability of the TT-based sequential estimation methods. 

\begin{figure}[t]
	\centering
	\includegraphics[valign=t,scale=0.99]{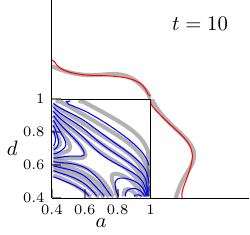}
	\includegraphics[valign=t,scale=0.99]{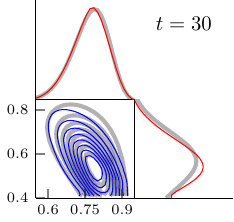}
	\includegraphics[valign=t,scale=0.99]{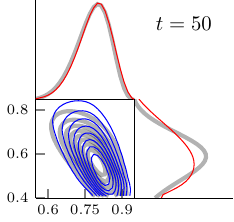}
	\\
	\includegraphics[valign=t,scale=0.99]{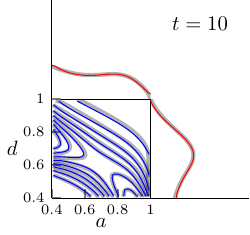}
	\includegraphics[valign=t,scale=0.99]{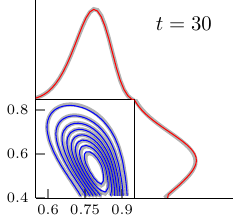}
	\includegraphics[valign=t,scale=0.99]{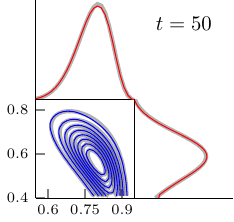}
	\\
	\includegraphics[valign=t,scale=0.99]{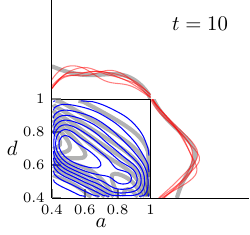}
	\includegraphics[valign=t,scale=0.99]{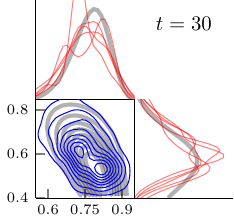}
	\includegraphics[valign=t,scale=0.99]{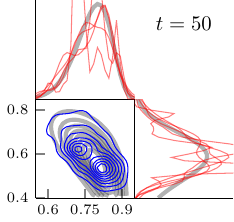}
	\\
	\includegraphics[valign=t,scale=0.99]{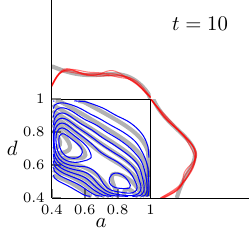}
	\includegraphics[valign=t,scale=0.99]{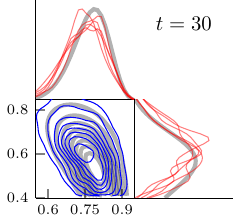}
	\includegraphics[valign=t,scale=0.99]{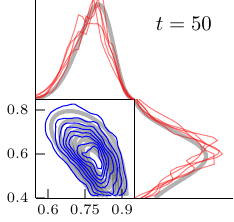}
	\caption{Linear Kalman filter with unknown parameters. The contours of posterior densities of $\bth=(a,d)$ at different time. From the top row to the bottom row, blue lines  are results obtained by Alg.~\ref{alg:stt} with $ (\ell,r) = (33,15) $,  Alg.~\ref{alg:stt} with $ (\ell,r) = (33,30) $, SMC$^2$ with 500 particles, and SMC$^2$ with 5000 particles, respectively. The marginal densities are shown by red curves on each side. The analytical solutions are plotted as thick gray curves at the background. For SMC$^2$, the marginal densities are results obtained from 5 batches, whereas the contours are results obtained by putting all batches together. }
	\label{fig:kalman_50}
\end{figure}

\textbf{Comparison with SMC$^2$.} We compare Alg.~\ref{alg:stt} with $(\ell,r)=(33,30)$ and $(\ell,r)=(33,15)$ to the SMC$^2$ method. We consider two setups for the SMC$^2$ method. One uses a particle size $ N_\theta = 500 $ for the parameters and another one uses an increased particle size $ N_\theta = 5000 $. In all numerical experiments, the particle size for the states is initially set to $ N_x = 100 $ and adaptively increased so that the acceptance rate of the parameter particles remains a relatively high level \citep{chopin2013smc2}. In our setup, the computational cost of running SMC$^2$ with $(N_\theta,N_x) = (500,100)$ is about the same as running Alg.~\ref{alg:stt} with $(\ell,r)=(33,30)$. Both take roughly 10 minutes using a 3.20 GHz Intel i7-8700 CPU. 
In Fig.~\ref{fig:kalman_50}, we present contours of the posterior parameter densities and their one-dimensional marginals at $ t \in \{10, 30, 50\}$ estimated using various algorithms. 
For SMC$^2$, we follow the setup of \cite{chopin2013smc2} to run five independent batches of particles, and then apply the kernel density estimation (KDE) method \citep{silverman1986density,peter1985kernel} to estimate the joint density using the union of the five batches and the one-dimensional parameter marginal density of each batch. 
Both TT-based setups perform better than the SMC$^2$-based setups.
The approximate posterior parameter densities from TT with $(\ell,r)=(33,30)$ accurately capture the true densities at all time steps, while those from SMC$^2$ do not have a consistent result across different batches, especially with $(N_\theta,N_x) = (500,100)$.

\subsection{Stochastic volatility}\label{sec:volatility}

Our second numerical example uses the stochastic volatility model defined in Example \ref{eg:sv}. 
We test Alg.~\ref{alg:stt} and Alg.~\ref{alg:stt_smooth} on synthetic data sets and a real-world data set taken from the S\&P 500 index.

\textbf{Synthetic data.} We fix $ \sigma = 1 $ and aim to estimate parameters $\bth=(\gamma,\beta)$. 
We consider the sequential estimation problems for $T=1000$ steps. Here $\gamma=0.6$ and $\beta=0.4$ are used to generate synthetic data. 
We set the prior density of the initial state to be $p(\bx|\bth) := \mathcal N(\bx;0, 1/(1 -\gamma^2))$ and the prior density of the parameters to be $p(\bth) := \mathrm{uniform}(\bth; [0.1,0.9]^2)$. 
Since the analytical marginal densities are no longer available in this example, we can only evaluate the $ 1002 $-dimensional joint posterior density of the parameters and all the states $(\bth, \bx_{0:1000})$, and then estimate the corresponding joint ESS as a measure of the approximation errors.

\begin{figure}[!ht]
\centering\hfill
\includegraphics{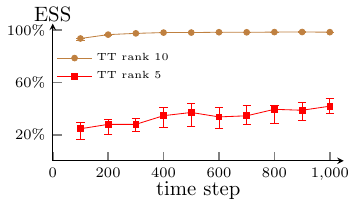}\hfill
\includegraphics{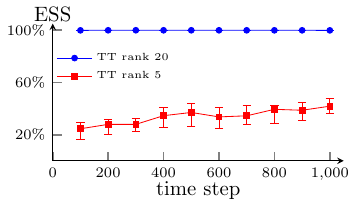}\hfill\vspace{-9pt}
\caption{Stochastic volatility model with synthetic data. The ESS for the joint posterior density versus time, computed using different maximum  TT ranks $r \in \{5, 10, 20\}$.}
\label{fig:svess}
\end{figure}

\begin{figure}[!ht]
\centering
\includegraphics[scale=0.93]{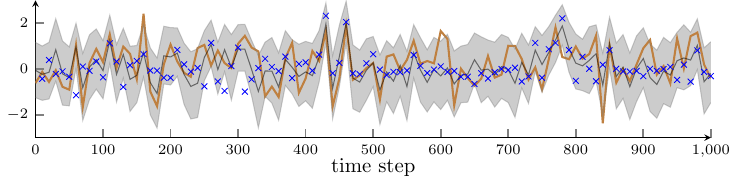}\hspace{-24pt}\vspace{-9pt}
\caption{Stochastic volatility model with synthetic data. The thick brown line represents the path of the true state and crosses are the observations. Estimated using the approximate posterior state path density $ \hat p(\bx_{0:1000} |\by_{1:1000}) $ defined in \eqref{eq:smooting_density} with a maximum TT rank $ r = 10 $, the black line and the shaded region represent the median path and the credible interval bounded between 5\% and 95\% percentiles, respectively.\label{fig:svstate}}
\end{figure}

\begin{figure}[!ht]
\centering
\includegraphics[scale = .9]{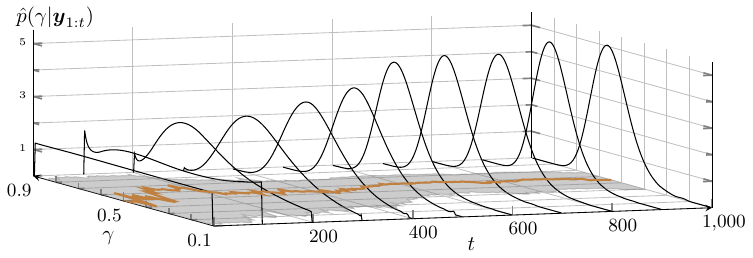}\hspace{-24pt}\vspace{-9pt}
\caption{Stochastic volatility model with synthetic data. The brown line and the shaded region in the horizontal $(t, \gamma)$-plane show the evolution of the median and the credible interval of the posterior estimates of $ \gamma $ bounded between 5\% and 95\% percentiles. The density profiles in the vertical axis show the approximate posterior parameter density $ \hat p(\gamma |\by_{1:t})$ at different time. \label{fig:sv3d}} 
\end{figure}

\begin{figure}[!ht]
\centering
\includegraphics[valign=t,scale=0.9]{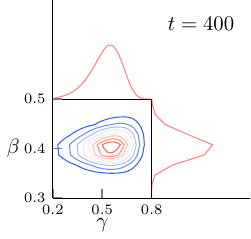}
\includegraphics[valign=t,scale=0.9]{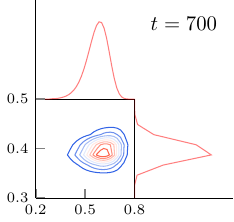}
\includegraphics[valign=t,scale=0.9]{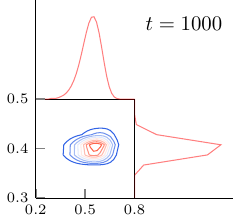}\vspace{-9pt}
\\
\includegraphics[valign=t,scale=0.9]{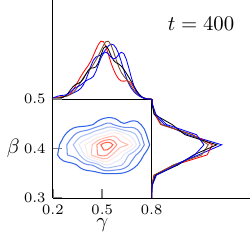}
\includegraphics[valign=t,scale=0.9]{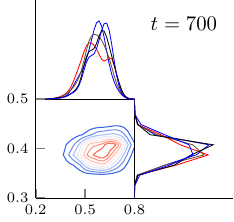}
\includegraphics[valign=t,scale=0.9]{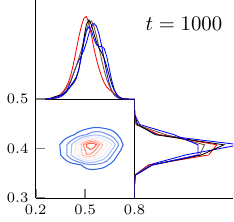}\vspace{-9pt}
\caption{Stochastic volatility model with synthetic data. The contours of posterior densities of $\bth=(\gamma,\beta)$ at different time. The top and bottom rows show results obtained by Alg.~\ref{alg:stt} and SMC$^2$ respectively. The marginal densities are shown on each side. For SMC$^2$, the marginal densities shown are results obtained from $5$ batches with 1000 particles in each batch, whereas the contours are results obtained by putting all batches together.} \label{fig:svsmc}\label{fig:svftt}
\vspace{-15pt}
\end{figure}

Similarly to the linear Kalman filter case, we use the piecewise Lagrange polynomial with the number of degrees of freedom $ \ell $ to be $  33 $ as the basis function, linear preconditioning, and five ALS iterations for building TT decompositions in Alg.~\ref{alg:stt}.
The accompanying path estimation algorithm (Alg.~\ref{alg:stt_smooth}) is used to generate weighted samples from the resulting TT-based approximations. We set the maximum TT rank to be $r\in\{5, 10, 20\}$ and show the change of ESS over time in Fig.~\ref{fig:svess}. 
We observe that with $r=5$, the TT-based approximation yields a reliable result with ESS above $ 20\% $ even after 1000 steps, while with $r =10$ it provides nearly $ 100\% $ ESS at almost all time steps. 
Fig.~\ref{fig:svstate} presents the trajectories of the states estimated from the approximate posterior state path density  $ \hat p( \bx_{0:1000} |\by_{1:1000}) $ using rank $ r = 10 $. 
We observe the true states are well followed by the trajectories estimated from the approximations. 
In Fig.~\ref{fig:sv3d}, we also show the approximate filtering densities of the parameter $ \gamma $ from the TT-based approximations, i.e., $ \hat p(\gamma |\by_{1:t}) $ over time. 
We observe that the densities get concentrated around the true value $ \gamma = 0.6 $ (for generating the data) with an increasing amount of data observed over time.

Next, we compare the TT-based method with $ r = 10 $ to SMC$^2$ with 5 independent batches, each batch with $ N_\theta = 1000 $ parameter particles and initially $ N_x = 100 $ state particles.
Fig.~\ref{fig:svftt} shows the contours and the marginals of the posterior parameter densities at $t \in \{400, 700, 1000\}$ obtained by the TT-based method and SMC$^2$. Both methods yield densities that concentrate around the true values of the parameters used to generate the data, i.e., $ (\gamma,\beta)=(0.6, 0.4) $.  However, the computation time of the TT-based method is about one hour using a 3.20 GHz Intel i7-8700 CPU, while that for one batch of SMC$^2$ method is over eight hours.

\textbf{S\&P 500 index.} We then apply our algorithms to the daily returns of the S\&P 500 index.
The historical data of the stock prices are obtained from Yahoo Finance\footnote{https://au.finance.yahoo.com/}.
The data set contains 1009 observations from 31st of December 2019 to 29th of December 2023, excluding weekends and holidays.
We compute the continuously compounded daily returns $ \{\by_t\}_{t=1}^{T} $ from the stock prices, where $ T = 1008 $.

We aim to estimate the states $\{\bX_t\}_{t=0}^T$  and all the three parameters $ \bTh = (\gamma, \sigma, \beta) $ from the observed daily returns $ \{\by_t\}_{t=1}^{T} $ of the S\&P 500 index.
Following the common practice in econometrics \citep{zhang2008box, yu2006class}, we set the priors for $ \bTh $ and $ \bX_0 $ as:
\begin{itemize}
	\item $ (\gamma + 1) /2 \sim \text{Beta}(\omega_1, \omega_2) $, with $ \omega_1 =20, \omega_2 = 1.5 $.
	\item $ \sigma^2 \sim \text{IG}(\zeta/2, S_\sigma /2) $, with $ \zeta = 2, S_\sigma = 0.01 $, where IG stands for the inverse Gamma distribution.
	\item $ \log(\beta) | \sigma \sim \mathcal N (\beta_0, \sigma^2 / q_0) $, with $  \beta_0 = 0 , q_0 = 0.8$.
	\item $ \bX_0  | \gamma,  \sigma \sim \mathcal N\big(0, \sigma^2 / (1- \gamma^2)\big) $.
\end{itemize}

When {building} TT approximations, we transform the parameters and the states in the stochastic volatility model to
\begin{align*}
	\bTh' &= \big(\Phi^{-1}(\gamma), \log(\sigma), \log(\beta)/\sigma\big) ,\\
	\bX_t' &= \bX_t / \sigma,
\end{align*}
where $ \Phi $ denotes the distribution function of {the} standard Gaussian random variable.
The transformed parameters and states take values in an unbounded domain, so that the linear preconditioning techniques developed in Sec.~\ref{sec:linear_precond} can be applied.
After obtaining the approximate densities of $ (\bTh', \bX_t') $ using our TT-based methods, we transform them back to the original variables $ (\bTh, \bX_t) $ to present the estimation results.

\begin{figure}[!ht]
	\centering
	\includegraphics{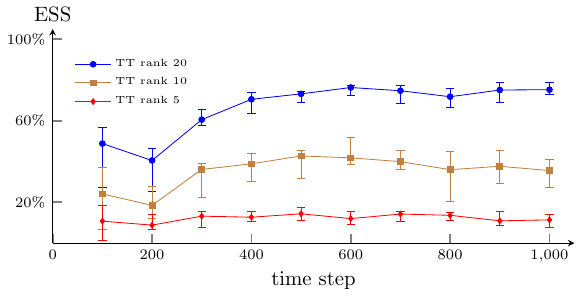}\vspace{-9pt}
	\caption{Stochastic volatility model with the S\&P500 returns. The ESS for the joint posterior density at different time, computed using maximum  TT ranks $r \in \{5, 10, 20\}$.}
	\label{fig:sp500ess}
\end{figure}

\begin{figure}[!ht]
	\centering
	\hspace{-5pt}\includegraphics[scale=0.93]{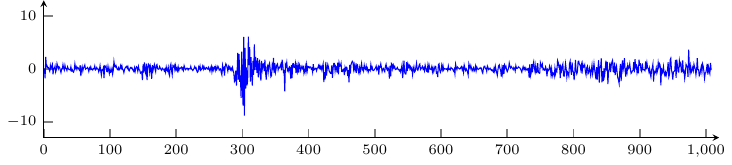}
	\includegraphics[scale=0.93]{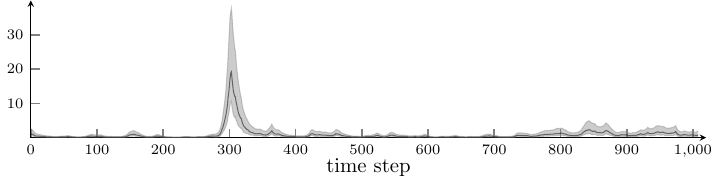}
	\caption{Stochastic volatility model with the S\&P500 returns. The observed returns (top) and the estimated trajectories of squared volatilities (bottom). In the bottom plot, the black line and the shaded region represent the median path and the credible interval bounded between 5\% and 95\% percentiles, respectively. These estimation results are obtained  using the approximate posterior state path density $ \hat p(\bx_{0:1008} |\by_{1:1008}) $ defined in \eqref{eq:smooting_density} with a maximum TT rank $ r = 20 $.\label{fig:sp500state}}
\end{figure}
\begin{figure}[!ht]
	\centering
	\includegraphics[valign=t,scale=0.9]{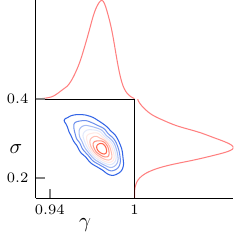}
	\includegraphics[valign=t,scale=0.9]{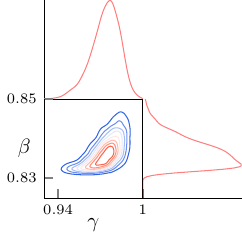}
	\includegraphics[valign=t,scale=0.9]{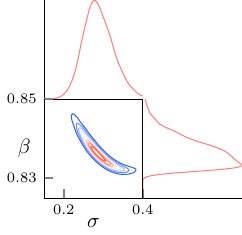}\vspace{-9pt}
	\caption{Stochastic volatility model with the S\&P500 returns. The contours of posterior densities of $\bth=(\gamma, \sigma, \beta)$ at $ T = 1008 $, estimated by Alg.~\ref{alg:stt}. The marginal densities are shown on each side.}\label{fig:sp500para}
	\vspace{-15pt}
\end{figure}
We set the maximum TT rank to be $r\in\{5, 10, 20\}$ and maintain the other TT configurations used in the synthetic case.
Fig.~\ref{fig:sp500ess} shows the change of ESS over time.
The ESS clearly increases with the TT rank, reaching around $ 70\% $ after 1008 steps for $ r = 20 $.
Fig.~\ref{fig:sp500state} shows the observed returns of the S\&P 500 index during the selected period and the trajectories of the squared volatilities estimated from the approximate posterior state path density  $ \hat p( \bx_{0:1008} |\by_{1:1008}) $ using rank $ r = 20 $. 
We observe that the highly volatile returns around $ t = 300 $ are well captured by the peak of the estimated volatilities, and the increasing variability of the returns after $ t = 800 $ is also captured by the increase in the estimated volatilities in the corresponding timeframe. 
In Fig.~\ref{fig:sp500para}, we also present the approximate posterior densities of the parameters $ \theta = \{\gamma, \sigma, \beta\} $ from the TT-based approximations $ \hat p(\theta |\by_{1:1008}) $. 
We observe a highly nonlinear pattern between $ \beta $ and the other two parameters.
These results demonstrate that our TT-based sequential estimation methods are effective for real-world data in this case.

\subsection{Susceptible-infectious-removed model}
We then apply our TT-based methods to a high-dimensional compartmental susceptible-infectious-removed (SIR) model to demonstrate its dimension scalability.
The SIR model describes the numbers of susceptible, infectious and removed individuals across $ J \in \mathbb N $ spatially dependent demographic compartments, denoted by $ S_j(t) $, $ I_j(t) $ and $ R_j(t) $, respectively.
The interaction between them are modelled by the following coupled ODEs:
\begin{equation}
	\left\{
		\begin{array}{rcl}
			\displaystyle \frac{\dd S_j}{\dd t} &=& - \kappa_j S_j I_j + \frac12 \sum_{i\in \mathcal I_j} (S_i - S_j) \vspace{8pt} \\
			\displaystyle \frac{\dd I_j}{\dd t} &=& \kappa_j S_j I_j - \nu_j I_j + \frac12 \sum_{i\in \mathcal I_j} (I_i - I_j) \vspace{8pt} \\
			\displaystyle \frac{\dd R_j}{\dd t} &=& \nu_j I_j + \frac12 \sum_{i\in \mathcal I_j} (R_i - R_j) 
		\end{array}
		\right., \label{eq:SIRode}
\end{equation}
where $ \kappa_j $ and $ \nu_j $ are the parameters representing the infection and recovery rates, and $ \mathcal I_j $ is the index set containing all neighbors of the $ j $-th compartment.
We fix the parameters $ \kappa_j = 0.1 $ and $ \nu_j = 18 $ for $ j = 1, 2, \ldots, J $, and focus on the inference for $ S_j $, $ I_j $ and $ R_j $. 

In the SIR model, it is sufficient to infer any two of the three random processes in each compartment, and the third process is known given the other two.
Here we infer the processes $ S_j $ and $ I_j $, and set the state to be
\[
	\bx = (S_1, I_1, \ldots, S_J, I_J) \in \mathbb R^{2J}.
\]
Then, the dynamic for the state follows
\[
	\displaystyle \frac{d\bx}{dt} = G(\bx), 
\]
where $G(\bx)$ is defined by the right-hand side of the first two equations in \eqref{eq:SIRode}.
To formulate the state-space model, we discretize the dynamical system at discrete time points $t = k \Delta t$, where $\Delta t = 0.02$, and perturb the states using Gaussian noises $\varepsilon^{(x)}_k \sim \mathcal N(\mathbf{0},  \mI_{2J}) $. 
Thus, conditioned on the state $\bX_{k\text{-}1}$ at time $(k-1)\Delta t$, the state $\bX_k$ at time $k \Delta t$ is given by
\[
\bX_k = \bX_{k\text{-}1} + \int_{(k\text{-}1) \Delta t}^{k \Delta t} G \big(\bx(t) \big) \dd t + \varepsilon^{(x)}_k.
\]
We use the explicit fourth-order Runge-Kutta formula with constant time step sizes $ \Delta t = 0.005 $ to numerically solve the above time integration problem. 
For the observation process, we assume only $ I_j(t) $ is observable at each time step, and the observations are perturbed by Gaussian noises $ \varepsilon^{(y)}_k \sim \mathcal N(\mathbf{0}, 100 \ \mI_{J})$. 
This defines the observation process
\[
	\bY_{k, j} = \bX_{k, 2j} + \varepsilon^{(y)}_{k, j}, \quad k = 1,2,\ldots, T, \quad j = 1,2,\ldots, J.
\]

\begin{figure}[!ht]
	\centering
	\includegraphics{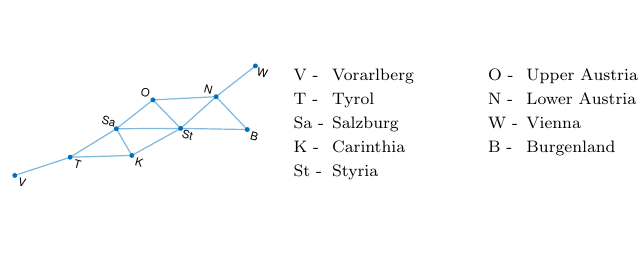}\vspace{-9pt}
	\caption{Compartment connectivity graph of the Austrian states \citep{cui2024deep}.}
	\label{fig:austria}
\end{figure}

We consider the same realistic setting as in \cite{cui2024deep}, with $ J = 9 $ compartments following the Austrian state adjacency map shown in Fig \ref{fig:austria}. 
We generate the synthetic data for $ T = 20 $ time steps, with fixed inhomogeneous initial states $ S_j(0) = 485 + j $ and $ I_j(0) = 15 - j $, for $ j = 1, 2, \ldots, 9 $.
Subsequently, we impose a Gaussian prior on the initial states $p(\bx_0) := \mathcal N(\bx_0;\boldsymbol{\mu}_0, \mI_{18})$ where $\boldsymbol{\mu}_0 = \big(S_1(0), I_1(0), \ldots, S_9(0), I_9(0) \big)$.
We aim to estimate the $ 18 $-dimensional states from the partial observations, which poses a challenging high-dimensional inference problem.

\begin{figure}[!ht]
	\centering
	\includegraphics{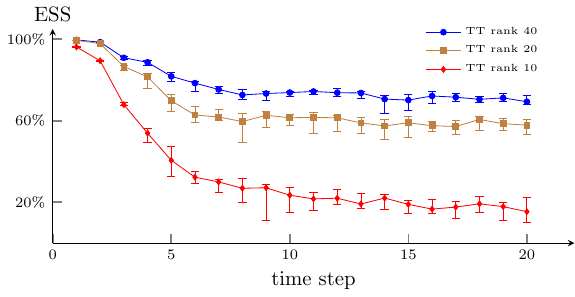}\vspace{-9pt}
	\caption{SIR model. The ESS for the joint posterior density versus time, computed using maximum TT ranks $r \in \{10, 20, 40\}$.}
	\label{fig:siress}
\end{figure}

\begin{figure}[!ht]
	\centering
	\includegraphics[scale=0.93]{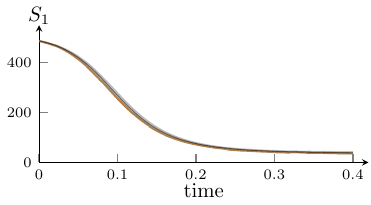}
	\includegraphics[scale=0.93]{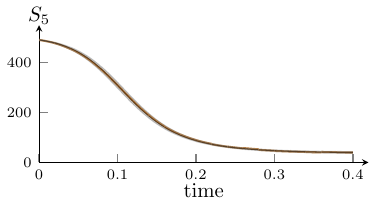}
	\includegraphics[scale=0.93]{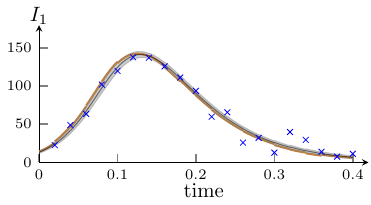}
	\includegraphics[scale=0.93]{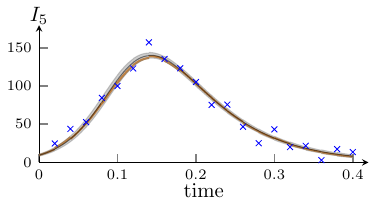}
	\caption{SIR model. The trajectories of the number of susceptible and infectious individuals in Vorarlberg (left) and Styria (right). The thick brown curves are the true state trajectories. The black curves represent the median posterior state path obtained by the approximate density $\hat p(\bx_{0:20} |\by_{1:20})$, and the shaded region is the credible interval bounded between 5\% and 95\% percentiles of the approximate smoothing densities.\label{fig:sirstate}}
\end{figure}

Similarly to the previous experiments, we employ Alg.~\ref{alg:stt} with linear preconditioning, the piecewise Lagrange polynomial with the number of degrees of freedom $ \ell  $ to be 33, and five ALS iterations.
Nevertheless, we increase the maximum TT rank to $r \in \{10, 20, 40\}$ to accommodate the larger state dimensions in this experiment.
Fig.~\ref{fig:siress} shows the change of ESS over time.
We observe that the large state dimension in this example requires TT ranks above $ r = 20 $ to achieve accurate results.
In Fig.~\ref{fig:sirstate}, we present the trajectories of the first and middle dimensions of the state (Vorarlberg and Styria), estimated from the approximate posterior state path density  $ \hat p( \bx_{0:20} |\by_{1:20}) $ using rank $ r = 40 $. 
We observe that the true dynamics are closely followed by the estimated trajectories, despite $ S_j $ being unobservable and the observations for $ I_j $ being far from the true dynamics.
These results demonstrate a strong dimension scalability of our TT-based sequential estimation methods in this case.

\subsection{Predator-prey model}

Finally, we consider the predator-prey model, which is a time-invariant dynamical system in the form of
\begin{equation} \label{eq:ppode}
	\left\{ \begin{array} {rl}
		\displaystyle \frac{dP}{dt} \; = \! & \displaystyle rP\Big(1-\frac{P}{K}\Big)-s\Big(\frac{P\,Q}{a+P}\Big)\vspace{4pt}\\
		\displaystyle \frac{dQ}{dt} \; = \! & \displaystyle u\Big(\frac{PQ}{a+P} - vQ\Big)
	\end{array} \right. ,
\end{equation}
where $P$ and $Q$ denote the prey population and the predator population, respectively. There are six parameters, $ \bth = (r, K, a, s, u, v) $, that control the behavior of the system. 
To formulate the state-space model, we represent $P$ and $Q$ as $\bx = (P,Q)$. This way, the dynamical system in \eqref{eq:ppode} can be equivalently expressed as
\[
	\displaystyle \frac{d\bx}{dt} = G(\bx; \bth), 
\]
where $G(\bx; \bth)$ is defined by the right-hand side of \eqref{eq:ppode}. 
Similarly to the previous SIR model, we discretize the dynamical system at discrete time points $t = k \Delta t$, where $\Delta t = 2$, and perturb the states using Gaussian noises $\varepsilon^{(x)}_k \sim \mathcal N(\mathbf{0}, 4 \,\mI_2) $. This way, the state transition process is given by
\[
\bX_k = \bX_{k\text{-}1} + \int_{(k\text{-}1) \Delta t}^{k \Delta t} G(\bx(t); \bth) \dd t + \varepsilon^{(x)}_k.
\]
We still use the explicit fourth-order Runge-Kutta formula with constant time step sizes $ \Delta t = 0.1 $ to numerically solve the above time integration problem. 
We observe $\bx$ at the same sequence of discrete time points, and the observations are perturbed by Gaussian noises $ \varepsilon^{(y)}_k \sim \mathcal N(\mathbf{0}, 4 \,\mI_2)$. 
This defines the observation process
\[
	\bY_{k} = \bX_{k} + \varepsilon^{(y)}_k.
\]
We generate synthetic observations using $ \bth = (0.6, 114, 25, 0.3, 0.5, 0.5) $ and  $\bx_0 = [50, 5]^\top$ up to the terminal time $T=20$. 
We impose a uniform prior on the parameters, $p(\bth):=\mathrm{uniform}(\bx; [\mathbf{a}, \mathbf{b}])$, where $[\mathbf{a}, \mathbf{b}] \equiv [a_1, b_1] \times \cdots \times [a_6, b_6]$ is a hypercube defined by $ \mathbf{a} = (0.1, 110, 20, 0.1, 0, 0) $ and $ \mathbf{b} = (1.1, 130, 30, 1.1, 1, 1) $. 
Furthermore, we impose a Gaussian prior on the initial states $p(\bx_0) := \mathcal N(\bx_0;\boldsymbol{\mu}_0, \mI_2)$ where $\boldsymbol{\mu}_0 = (50,5)$.

\begin{figure}[h]
	\centering
	\includegraphics[scale = 1]{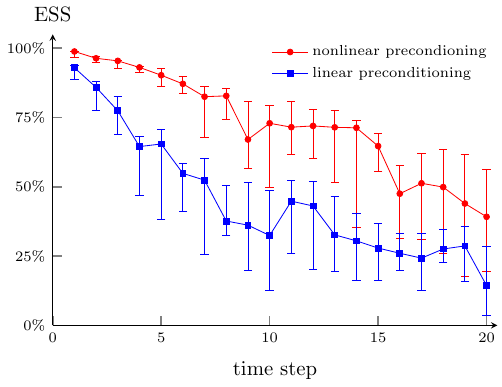}\vspace{-9pt}
	\caption{Predator-prey model. The ESS for the joint posterior density versus time, computed using different preconditioning techniques. }
	\label{fig:ppess}
\end{figure}

\begin{figure}[h]
	\centering\hfill
	\includegraphics{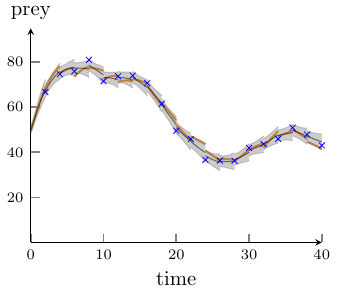}\hfill
	\includegraphics{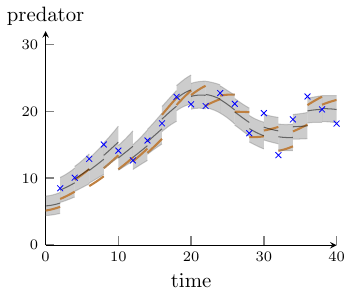}\hfill\vspace{-9pt}
	\caption{Predator-prey model. The thick brown curves are the true trajectories, discontinuity indicates the state noises at discrete time steps, and blue crosses are the observations.  The black curves represent the median posterior state path obtained by the approximate density $\hat p(\bx_{0:20} |\by_{1:20})$, and the shaded region is the credible interval bounded between 5\% and 95\% percentiles of the approximate smoothing densities. }
	\label{fig:ppstate}
\end{figure}

We first employ Alg.~\ref{alg:stt} with linear preconditioning, the piecewise Lagrange polynomial with the number of degrees of freedom $ \ell $ to be 33, a maximum TT rank $ r =20 $, and five ALS iterations. 
The accompanying path estimation algorithm (Alg.~\ref{alg:stt_smooth}) is used to generate weighted samples from the resulting TT-based approximations. In Fig.~\ref{fig:ppess}, the ESSs for the joint posterior density of the parameters and all the states are represented by square markers in blue. 
In this case, the ESS decreases rapidly. 
Instead of brutally increasing the rank to enhance the approximation power of TT, we keep the same TT rank and apply the nonlinear preconditioning technique. 
Here we use the bridging density \eqref{eq:bridging} with $ \beta_\pi = \beta_f = \beta_g = 0.4 $ and a standard multivariate Gaussian reference density.
With nonlinear preconditioning, the ESSs are shown by red circles in Fig.~\ref{fig:ppess}. 
We observe that nonlinear preconditioning significantly increases the ESS, where it remains at about $ 40\% $ after 20 steps. 
In Fig.~\ref{fig:ppstate}, we also present the trajectories of the states estimated from the approximate posterior state path density $ \hat p(\bx_{0:20} |\by_{1:20}) $ with nonlinear preconditioning and the true state trajectory. 
The curves between the two adjacent states $ \bx_{k\textit{-}1} $ and $ \bx_k $ are the solution of the dynamical system \eqref{eq:ppode} with initial condition $ \bx = \bx_{k\textit{-}1} $. 
We observe that the true trajectories are well followed by the approximations.

\section{Conclusion and further extensions}\label{sec:conclusion}

We present new TT-based methods to sequentially learn the posterior distributions of states and parameters of state-space models. Our main innovation is the recursive approximation of posterior densities using TT rather than relying on (weighted) particle representations. Using the square-root approximation technique (cf. Section \ref{sec:estimation}), we derive conditional KR rearrangements that map reference random variables to the approximated posterior random variables. As a result, this defines the particle filtering, path estimation and particle smoothing algorithms accompanying the recursive TT-based approximations. 

This work opens the door to many future research directions. For example, one can integrate our TT-based methods into SMC methods \citep{chopin2013smc2,crisan2018nested}, in which TT-based methods can potentially provide more efficient particle filters and MCMC proposal kernels for SMC. The KR rearrangements offered by our methods also open the door to integrating other structured particle sets---for instance, quasi Monte Carlo points (see \cite{dick2013high} and references therein)---into sequential learning algorithms to improve the rate of convergence of estimators. This integration can significantly complement the Hilbert space-filling curve technique used by the sequential quasi Monte Carlo method \citep{gerber2015sequential}.

Further research is needed to understand the rank structure of TT decompositions in state-space models---particularly on the impact of the Markov property---which may provide {\it a posteriori} error bounds to certify the algorithms presented here. The function approximation perspective of the work presented here is not limited to the TT decomposition. Other function approximation methods based on separable bases---for example, Gaussian processes and radial basis functions \citep{wendland2004scattered,williams2006gaussian} with product-form kernels, multivariate wavelet and Fourier bases \citep{daubechies1992ten,de1993construction,devore1988interpolation,mallat1989multiresolution}, and sparse grids and spectral polynomials \citep{bungartz2004sparse,shen2011spectral,xiu2002wiener}---can also be used to approximate posterior densities and compute their marginal densities. See \cite{cui2023self} for the preliminary investigations in stationary inference problems. This offers a pathway to designing problem-specific function approximation methods for sequential learning. For instance, state transition processes governed by certain partial differential equations may yield a hyperbolic-cross structure \citep{dung2018hyperbolic} suitable for sparse spectral polynomials.

Dimension scalability is a limitation of the TT decomposition and general transport-map methods. Although the complexity of building the TT decomposition can be dimension independent for functions equipped with rapid decaying weights \citep{griebel2021analysis}---which is the case for many high- or infinite-dimensional problems \citep{IP:Stuart_2010}---it is still computationally prohibitive to operate with the apparent state and parameter dimensions for very high-dimensional systems such as weather forecasts.  A possibility of reducing the number of variables in the approximation problem is to use the conditional Gaussian structure of certain state-space models \citep{chen2022conditional,chen2018rigorous,chen2018conditional} to marginalize high-dimensional conditional Gaussian variables analytically. Then, one only needs to apply TT-based approximations to the rest of the non-Gaussian variables. For a broader range of state-space models, there may exist approximate conditional Gaussian structures and similar structures that allow for analytical marginalization. In ongoing work, we are extending likelihood-informed subspace methods \citep{DimRedu:Cui_etal_2014,cui2022unified, zahm2022certified} to state-space models to identify these intrinsic structures.

\section*{Acknowledgments}
We would like to thank Sergey Dolgov, Rob Scheichl and Olivier Zahm for many fruitful discussions about the tensor trains and sparse polynomials. This work was supported in part by the Australian Research Council under the grant DP210103092.

\appendix

\section{$f$-divergence of marginal random variables}\label{appendix:A}

Commonly used statistical divergences such as the Kullback--Leibler divergence, the total variation distance and the squared Hellinger distance are instances of the  $f$-divergence. Briefly, given a convex function $f(\cdot)$, the $f$-divergence is defined as
\[
\df{p}{q} = \int f\Big(\frac{p(\bx)}{q(\bx)}\Big) q(\bx) \dd \bx.  
\] 
The following lemma bounds the $f$-divergence of marginal random variables by the corresponding $f$-divergence of joint random variables. 

\begin{lemma} \label{lemma:marginal}
	Let $ p_1(\bx_1, \bx_2) $ and $ p_2(\bx_1, \bx_2) $ be two joint densities, and $ \bar p_1(\bx_1) $ and $ \bar p_2(\bx_1) $ be the marginal densities of $p_1$ and $p_2$, respectively.
	The $ f $-divergence of  $ \bar p_1$ from $ \bar p_2$ is bounded from above by that of $ p_1$ from $p_2$, i.e., $\df{\bar p_1}{\bar p_2} \leq \df{p_1}{p_2}$.
\end{lemma}
	
\begin{proof}
	Since the function $f$ is convex, applying Jensen's inequality, we have
	\begin{align*}
		\df{p_1}{p_2} & = \mathbb E_{p_2(\bx_1, \bx_2)} \bigg[ f\Big(\frac{p_1(\bX_1,\bX_2)}{p_2(\bX_1, \bX_2)}\Big) \bigg]\\
		& = \int \Big( \int f \Big(\frac{p_1(\bx_1,\bx_2)}{p_2(\bx_1, \bx_2)}\Big) \frac{p_2(\bx_1, \bx_2) }{\bar p_2(\bx_1) } \dd \bx_2 \Big) \bar p_2(\bx_1) \dd \bx_1\\
		& \geq \int f \Big( \int \frac{p_1(\bx_1,\bx_2)}{p_2(\bx_1, \bx_2)} \frac{p_2(\bx_1, \bx_2) }{\bar p_2(\bx_1) } \dd \bx_2 \Big) \bar p_2(\bx_1) \dd \bx_1\\
		& = \int f \Big( \int \frac{p_1(\bx_1,\bx_2)}{\bar p_2(\bx_1) } \dd \bx_2 \Big) \bar p_2(\bx_1) \dd \bx_1\\
		& = \int f \Big( \frac{\bar p_1(\bx_1)}{\bar p_2(\bx_1) } \Big) \bar p_2(\bx_1) \dd \bx_1 = \df{\bar p_1}{\bar p_2},
	\end{align*}
    which concludes the proof.
\end{proof}

When the densities are only known up to some constants, the $ L^2 $ distance between the square root of the unnormalized marginal densities is also bounded by that of the unnormalized joint densities. This is shown in the following lemma.
\begin{lemma} \label{lemma:L2_sqrt}
	Let $ \pi_1(\bx_1, \bx_2) $ and $ \pi_2(\bx_1, \bx_2) $ be two unnormalized densities, and $ \bar \pi_1(\bx_1) = \int \pi_1(\bx_1, \bx_2) \dd \bx_2 $ and $ \bar \pi_2(\bx_1) = \int \pi_2(\bx_1, \bx_2) \dd \bx_2 $ be their corresponding unnormalized marginal densities.
	The $ L^2 $ distance between $ \surd \bar \pi_1 $ and $ \surd \bar \pi_2 $ is bounded from above by the $L^2$ distance between $\surd \pi_1 $ and $\surd \pi_2$, i.e.,
	\(
		\lerr{ \surd \bar \pi_1 - \surd \bar \pi_2} \leq \lerr{\surd \pi_1 - \surd \pi_2}.
	\)
\end{lemma}

\begin{proof}
	The $L^2$ distance between $\surd \pi_1 $ and $\surd \pi_2$ can be expressed as
	\begin{align*}
		\lerr{ \surd  \pi_1 - \surd \pi_2}^2  & = \int \int \big( \surd \pi_1(\bx_1, \bx_2) - \surd \pi_2(\bx_1, \bx_2) \big)^2 \dd \bx_1 \dd \bx_2 \\
		& = \int \int \bigg( \sqrt{\frac{ \pi_1(\bx_1, \bx_2)}{\pi_2(\bx_1, \bx_2)}} - 1 \bigg)^2  \pi_2(\bx_1, \bx_2) \dd \bx_1 \dd \bx_2 \\
		& = \int \int \bigg( \sqrt{\frac{ \pi_1(\bx_1, \bx_2)}{\pi_2(\bx_1, \bx_2)}} - 1 \bigg)^2  \frac{\pi_2(\bx_1, \bx_2)}{\bar\pi_2 (\bx_1)} \dd \bx_2 \bar\pi_2(\bx_1) \dd \bx_1 .
	\end{align*}
	Since $ f(r) = (\surd r - 1)^2, r \geq 0 $ is a convex function, applying Jensen's inequality to the inner integral over $\bx_2$, the rest of the proof follows the same steps of Lemma \ref{lemma:marginal}.
\end{proof}

Lemma \ref{lemma:L2_sqrt} is useful to establish the error bound between two unnormalized marginal densities given that between the joint densities. However, Lemma \ref{lemma:L2_sqrt} implicitly assumes that both unnormalized joint functions (densities) are non-negative by construction, which is generally not true when $ \pi_2 $ is a function approximation of $ \pi_1 $. The following lemma does not restrict the approximate functions to be non-negative, and gives a similar result under the $L^1$ distance, which is analogous to the total variation distance for normalized densities.

\begin{lemma} \label{lemma:L1}
	Let $ \pi_1(\bx_1, \bx_2) $ be an unnormalized density and $ \bar \pi_1(\bx_1) = \int \pi_1(\bx_1, \bx_2) \dd \bx_2 $ be the corresponding marginal density. Suppose $ \pi_1(\bx_1, \bx_2) $ has an approximation $ \pi_2(\bx_1, \bx_2) $, which does not preserve the non-negativity, and let  $ \bar \pi_2(\bx_1) = \int \pi_2(\bx_1, \bx_2) \dd \bx_2 $ be the corresponding marginal function.
	The $ L^1 $ distance between $\bar \pi_1 $ and $\bar \pi_2 $ is bounded from above by that between $\pi_1 $ and $\pi_2$, i.e.,
	\(
		\| \bar \pi_1 - \bar \pi_2 \|_{L_1} \leq \| \pi_1 - \pi_2 \|_{L_1}.
	\)
\end{lemma}

\begin{proof}
	The $L^1$ distance between $\pi_1 $ and $\pi_2$ can be expressed as
	\begin{align*}
		\| \pi_1 - \pi_2\|_{L_1}  & = \int \int \big|  \pi_1(\bx_1, \bx_2) - \pi_2(\bx_1, \bx_2) \big| \dd \bx_1 \dd \bx_2 \\
		& = \int \int \bigg| \frac{ \pi_2(\bx_1, \bx_2)}{\pi_1(\bx_1, \bx_2)} - 1 \bigg|  \pi_1(\bx_1, \bx_2) \dd \bx_1 \dd \bx_2 \\
		& = \int \int \bigg| \frac{ \pi_2(\bx_1, \bx_2)}{\pi_1(\bx_1, \bx_2)} - 1 \bigg|  \frac{\pi_1(\bx_1, \bx_2)}{\bar\pi_1 (\bx_1)} \dd \bx_2 \bar\pi_1(\bx_1) \dd \bx_1 .
	\end{align*}
	Since $ f(r) = |r - 1| $ is a convex function, applying Jensen's inequality to the inner integral over $\bx_2$, the rest of the proof follows the same steps of Lemma \ref{lemma:marginal}.
\end{proof}

\newcommand{\reals}{\mathbb R}
\newcommand{\vone}{\mathbf 1}
\newcommand{\mA}{\mathsf A}

\section{Posterior densities of the linear Kalman filter with unknown parameters} \label{appendix:kalman} We recall that the linear Kalman filter specifies the state and observation processes as
\begin{equation} 
	\left\{ 
	\begin{array}{rl}
		\bX_t - \mu  \; = \! & b \left(\bX_{\tminusone}-\mu\right) + a \, \varepsilon^{(x)}_{t}\\
		\bY_t  \; = \!  & \mC \bX_t + d \, \varepsilon^{(y)}_{t} 
	\end{array} \label{eq:app_kalman}
	\right. ,
\end{equation}
where $\varepsilon^{(x)}_t$ and $\varepsilon^{(y)}_t$ follow respectively independent Gaussian distributions $\mathcal{N}(0,\mI_m)$ and $\mathcal{N}(0, \mI_n)$, and $\mC \in \R^{n \times m}$ is the observation matrix. 
Here $\mI_{m}$ represents an $m\times m$ identity matrix.
The prior density of the initial state $ \bX_0 $ is given by $p(\bx_0|\mu) := \mathcal{N}(\bx_0;\mu \mathbf{1}_m, \mI_{m})$, where $\mathbf{1}_m$ is an $m$-dimensional vector filled with ones.  
Without loss of generality, we set $\mu$ to be zero to simplify notation used in the derivation. 
This way, the parameters controlling the model are collected as $ \bth = (a,b, d)$.
It is well known that the filtering density $(\bX_t | \bTh = \bth, \bY_{1:t} = \by{}_{1:t})$ and the smoothing density $(\bX_k | \bTh =\bth, \bY_{1:t} = \by{}_{1:t})$ for any $k < t$ are Gaussian conditioned on $\bTh = \bth$. See \cite{sarkka2013bayesian} and reference therein for details. Using this property, here we derive the analytical form of the posterior parameter density $ p(\bth|\by_{1:t})$.

As a starting point, we collect all the quantities from time $ 0 $ to $ t $ in vectorized forms,
\[
	\begin{array}{lll}
		\bar \bX & \hspace{-3pt} =[\bX_0;\ldots;\bX_t] & \hspace{-3pt} \in\reals^{m(t+1)} , \vspace{4pt} \\
		\bar \bY & \hspace{-3pt} =[\bY_{1};\ldots;\bY_{t}] & \hspace{-3pt} \in\mathbb{R}^{nt},  \vspace{4pt} \\
		\bar \varepsilon^{(x)} & \hspace{-3pt} = [\varepsilon^{(x)}_{0} ;\ldots ;\varepsilon^{(x)}_{t}] & \hspace{-3pt}\in\mathbb{R}^{m(t+1)}, \vspace{4pt} \\
		\bar \varepsilon^{(y)} & \hspace{-3pt} = [ \varepsilon^{(y)}_{1} ; \ldots ; \varepsilon^{(x)}_{t}] & \hspace{-3pt}\in\mathbb{R}^{nt}.
	\end{array}
\]
Corresponding to the collected states $\bX$ and the collected observations $\bY$, we define the joint state transition matrix $ \mA_{\bth} \in\reals^{m\left(t+1\right)\times m\left(t+1\right)} $ and the joint observation matrix $ \mH \in\reals^{nt\times m\left(t+1\right)} $ as
\[
		\mA_{\bth} =  \begin{bmatrix*}[r]
			a\,\mI_{m}\\
			-b\,\mI_{m} & \mI_{m}\\
			& -b\,\mI_{m} & \mI_{m}\\
			&  & \ddots & \ddots\\
			&  &  & -b\,\mI_{m} & \mI_{m}
		\end{bmatrix*} \quad \text{and} \quad
		\mH =  \begin{bmatrix*}[c]
			\mathsf 0 & \mC\\
			& \mathsf 0 & \mC\\
			&  & \mathsf 0 & \mC\\
			&  &  & \ddots & \ddots\\
			&  &  &  & \mathsf 0 & \mC
		\end{bmatrix*} ,
\]
respectively, where $ \mathsf 0 $ denotes a matrix of size $ n \times m $ filled with zeros.
Then, the state and observation processes defined in \eqref{eq:app_kalman} can be equivalently expressed as
\begin{equation}
	\bar \bX = a \, \mA_{\bth}^{-1} \bar\varepsilon^{(x)}, \quad \bar \bY = \mH \bar \bX + d \, \bar\varepsilon^{(y)}. \label{eq:XandY}
\end{equation}
The joint posterior density of the state path $\bX_{0:t}\equiv \bar\bX$ and the parameters $\bTh$ takes the form
\begin{equation}\label{eq:app_joint}
p(\bar\bx,\bth|\bar\by) = \frac{1}{p(\bar\by)}  p(\bar\by|\bar\bx,\bth)p(\bar\bx|\bth)p(\bth) .
\end{equation}
Our goal is to derive the posterior parameter density 
\begin{equation}\label{eq:app_marginal_post}
	p(\bth|\bar\by) = \int p(\bar\bx,\bth|\bar\by) \dd \bar\bx = \frac{1}{p(\bar\by)} p(\bth)p(\bar\by | \bth), 
\end{equation}
where $p(\bar\by | \bth)$ is the marginal likelihood defined as
\begin{equation*}\label{eq:app_marginal_like}
	p(\bar\by | \bth) = \int p(\bar\by|\bar\bx,\bth)p(\bar\bx|\bth)  \dd \bar\bx.
\end{equation*}

Combining the two equations in \eqref{eq:XandY}, we obtain
\[
	\bar \bY = a\mH  \, \mA_{\bth}^{-1} \bar\varepsilon^{(x)} + d \, \bar\varepsilon^{(y)}.
\]
Note that $ \bar\varepsilon^{(x)} $ and $ \bar\varepsilon^{(y)} $ are independent Gaussian random variables.
Thus, conditioned on the parameters $ \bTh = \bth $, the observations $ \bar \bY $ follow a multivariate Gaussian distribution:
\[
	(\bar \bY | \bTh = \bth) \sim \mathcal N(\bar \by ; 0, \bar \Sigma_{\bth}) \quad \text{where} \quad \bar \Sigma_{\bth} = a^2 \mH   \mA_{\bth}^{-1}  \mA_{\bth}^{-\top} \mH^\top + d^2 \, \mI_{nt}.
\]
To be more specific, we have 
\[
	p(\bar\by | \bth) = (2\pi)^{-nt/2} \, \det(\bar \Sigma{}_{\bth})^{-1/2} \exp\Big( - \frac{1}{2} \bar\by^{\top} \bar \Sigma{}_{\bth}^{-1} \bar\by \Big).
\]
The explicit form of the posterior parameter density follows from \eqref{eq:app_marginal_post}.


\begin{thebibliography}{76}
\providecommand{\natexlab}[1]{#1}
\providecommand{\url}[1]{\texttt{#1}}
\expandafter\ifx\csname urlstyle\endcsname\relax
  \providecommand{\doi}[1]{doi: #1}\else
  \providecommand{\doi}{doi: \begingroup \urlstyle{rm}\Url}\fi

\bibitem[Anderson and Moore(2012)]{anderson2012optimal}
Brian D~O Anderson and John~B Moore.
\newblock \emph{Optimal filtering}.
\newblock Courier Corporation, 2012.

\bibitem[Andrieu et~al.(2010)Andrieu, Doucet, and
  Holenstein]{andrieu2010particle}
Christophe Andrieu, Arnaud Doucet, and Roman Holenstein.
\newblock Particle {M}arkov chain {M}onte {C}arlo methods.
\newblock \emph{Journal of the Royal Statistical Society: Series B (Statistical
  Methodology)}, 72\penalty0 (3):\penalty0 269--342, 2010.

\bibitem[Beskos et~al.(2016)Beskos, Jasra, Kantas, and
  Thiery]{beskos2016convergence}
Alexandros Beskos, Ajay Jasra, Nikolas Kantas, and Alexandre Thiery.
\newblock On the convergence of adaptive sequential {M}onte {C}arlo methods.
\newblock \emph{Annals of Applied Probability}, 26\penalty0 (2):\penalty0
  1111--1146, 2016.

\bibitem[Bigoni et~al.(2016)Bigoni, Engsig-Karup, and
  Marzouk]{bigoni2016spectral}
Daniele Bigoni, Allan~P Engsig-Karup, and Youssef~M Marzouk.
\newblock Spectral tensor-train decomposition.
\newblock \emph{SIAM Journal on Scientific Computing}, 38\penalty0
  (4):\penalty0 A2405--A2439, 2016.

\bibitem[Bresler(1986)]{bresler1986two}
Yoram Bresler.
\newblock Two-filter formulae for discrete-time non-linear {B}ayesian
  smoothing.
\newblock \emph{International Journal of Control}, 43\penalty0 (2):\penalty0
  629--641, 1986.

\bibitem[Briers et~al.(2010)Briers, Doucet, and Maskell]{briers2010smoothing}
Mark Briers, Arnaud Doucet, and Simon Maskell.
\newblock Smoothing algorithms for state--space models.
\newblock \emph{Annals of the Institute of Statistical Mathematics},
  62:\penalty0 61--89, 2010.

\bibitem[Bungartz and Griebel(2004)]{bungartz2004sparse}
Hans-Joachim Bungartz and Michael Griebel.
\newblock Sparse grids.
\newblock \emph{Acta Numerica}, 13:\penalty0 147--269, 2004.

\bibitem[Capp{\'e} et~al.(2006)Capp{\'e}, Moulines, and
  Ryd{\'e}n]{cappe2009inference}
Olivier Capp{\'e}, Eric Moulines, and Tobias Ryd{\'e}n.
\newblock \emph{Inference in Hidden {M}arkov Models}.
\newblock Springer Science \& Business Media, 2006.

\bibitem[Carpenter et~al.(1999)Carpenter, Clifford, and
  Fearnhead]{carpenter1999improved}
James Carpenter, Peter Clifford, and Paul Fearnhead.
\newblock Improved particle filter for nonlinear problems.
\newblock \emph{IEE Proceedings-Radar, Sonar and Navigation}, 146\penalty0
  (1):\penalty0 2--7, 1999.

\bibitem[Chen and Majda(2018)]{chen2018conditional}
Nan Chen and Andrew~J Majda.
\newblock Conditional {G}aussian systems for multiscale nonlinear stochastic
  systems: Prediction, state estimation and uncertainty quantification.
\newblock \emph{Entropy}, 20\penalty0 (7):\penalty0 509, 2018.

\bibitem[Chen et~al.(2018)Chen, Majda, and Tong]{chen2018rigorous}
Nan Chen, Andrew~J Majda, and Xin~T Tong.
\newblock Rigorous analysis for efficient statistically accurate algorithms for
  solving {F}okker--{P}lanck equations in large dimensions.
\newblock \emph{SIAM/ASA Journal on Uncertainty Quantification}, 6\penalty0
  (3):\penalty0 1198--1223, 2018.

\bibitem[Chen et~al.(2022)Chen, Li, and Liu]{chen2022conditional}
Nan Chen, Yingda Li, and Honghu Liu.
\newblock Conditional {G}aussian nonlinear system: A fast preconditioner and a
  cheap surrogate model for complex nonlinear systems.
\newblock \emph{Chaos: An Interdisciplinary Journal of Nonlinear Science},
  32\penalty0 (5):\penalty0 053122, 2022.

\bibitem[Chopin et~al.(2013)Chopin, Jacob, and
  Papaspiliopoulos]{chopin2013smc2}
Nicolas Chopin, Pierre~E Jacob, and Omiros Papaspiliopoulos.
\newblock {SMC}2: an efficient algorithm for sequential analysis of state space
  models.
\newblock \emph{Journal of the Royal Statistical Society: Series B (Statistical
  Methodology)}, 75\penalty0 (3):\penalty0 397--426, 2013.

\bibitem[Chorin and Tu(2009)]{chorin2009implicit}
Alexandre~J Chorin and Xuemin Tu.
\newblock Implicit sampling for particle filters.
\newblock \emph{Proceedings of the National Academy of Sciences}, 106\penalty0
  (41):\penalty0 17249--17254, 2009.

\bibitem[Constantine and Savits(1996)]{constantine1996multivariate}
Gregory~M Constantine and Thomas~H Savits.
\newblock A multivariate {F}aa di {B}runo formula with applications.
\newblock \emph{Transactions of the American Mathematical Society},
  348\penalty0 (2):\penalty0 503--520, 1996.

\bibitem[Crisan and Miguez(2018)]{crisan2018nested}
Dan Crisan and Joaquin Miguez.
\newblock Nested particle filters for online parameter estimation in
  discrete-time state-space {M}arkov models.
\newblock \emph{Bernoulli}, 24\penalty0 (4A):\penalty0 3039--3086, 2018.

\bibitem[Cui and Dolgov(2022)]{cui2021deep}
Tiangang Cui and Sergey Dolgov.
\newblock Deep composition of tensor-trains using squared inverse {R}osenblatt
  transports.
\newblock \emph{Foundations of Computational Mathematics}, 22\penalty0
  (6):\penalty0 1863--1922, 2022.

\bibitem[Cui and Tong(2022)]{cui2022unified}
Tiangang Cui and Xin~T Tong.
\newblock A unified performance analysis of likelihood-informed subspace
  methods.
\newblock \emph{Bernoulli}, 28\penalty0 (4):\penalty0 2788--2815, 2022.

\bibitem[Cui et~al.(2014)Cui, Martin, Marzouk, Solonen, and
  Spantini]{DimRedu:Cui_etal_2014}
Tiangang Cui, James Martin, Youssef~M. Marzouk, Antti Solonen, and Alessio
  Spantini.
\newblock Likelihood-informed dimension reduction for nonlinear inverse
  problems.
\newblock \emph{Inverse Problems}, 30:\penalty0 114015, 2014.

\bibitem[Cui et~al.(2023)Cui, Dolgov, and Zahm]{cui2023self}
Tiangang Cui, Sergey Dolgov, and Olivier Zahm.
\newblock Self-reinforced polynomial approximation methods for concentrated
  probability densities.
\newblock \emph{arXiv preprint: 2303.02554}, 2023.

\bibitem[Cui et~al.(2024)Cui, Dolgov, and Scheichl]{cui2024deep}
Tiangang Cui, Sergey Dolgov, and Robert Scheichl.
\newblock Deep importance sampling using tensor trains with application to a priori and a posteriori rare events.
\newblock \emph{SIAM Journal on Scientific Computing}, 46\penalty0 (1):\penalty0 C1--C29, 2024.

\bibitem[Daubechies(1992)]{daubechies1992ten}
Ingrid Daubechies.
\newblock \emph{Ten lectures on wavelets}.
\newblock SIAM, 1992.

\bibitem[de~Boor et~al.(1993)de~Boor, DeVore, and Ron]{de1993construction}
Carl de~Boor, Ronald~A DeVore, and Amos Ron.
\newblock On the construction of multivariate (pre) wavelets.
\newblock \emph{Constructive approximation}, 9\penalty0 (2):\penalty0 123--166,
  1993.

\bibitem[Del~Moral et~al.(2012)Del~Moral, Doucet, and Jasra]{del2012adaptive}
Pierre Del~Moral, Arnaud Doucet, and Ajay Jasra.
\newblock On adaptive resampling strategies for sequential {M}onte {C}arlo
  methods.
\newblock \emph{Bernoulli}, 18\penalty0 (1):\penalty0 252--278, 2012.

\bibitem[DeVore and Popov(1988)]{devore1988interpolation}
Ronald~A DeVore and Vasil~A Popov.
\newblock Interpolation of besov spaces.
\newblock \emph{Transactions of the American Mathematical Society},
  305\penalty0 (1):\penalty0 397--414, 1988.

\bibitem[Dick et~al.(2013)Dick, Kuo, and Sloan]{dick2013high}
Josef Dick, Frances~Y Kuo, and Ian~H Sloan.
\newblock High-dimensional integration: the quasi-{M}onte {C}arlo way.
\newblock \emph{Acta Numerica}, 22:\penalty0 133--288, 2013.

\bibitem[Dolgov and Savostyanov(2014)]{dolgov2014alternating}
Sergey~V Dolgov and Dmitry~V Savostyanov.
\newblock Alternating minimal energy methods for linear systems in higher
  dimensions.
\newblock \emph{SIAM Journal on Scientific Computing}, 36\penalty0
  (5):\penalty0 A2248--A2271, 2014.

\bibitem[Doucet and Johansen(2009)]{doucet2009tutorial}
Arnaud Doucet and Adam~M Johansen.
\newblock A tutorial on particle filtering and smoothing: Fifteen years later.
\newblock \emph{Handbook of nonlinear filtering}, 12\penalty0
  (656-704):\penalty0 3, 2009.

\bibitem[D{\~u}ng et~al.(2018)D{\~u}ng, Temlyakov, and
  Ullrich]{dung2018hyperbolic}
Dinh D{\~u}ng, Vladimir Temlyakov, and Tino Ullrich.
\newblock \emph{Hyperbolic cross approximation}.
\newblock Springer, 2018.

\bibitem[Einicke and White(1999)]{einicke1999robust}
Garry~A Einicke and Langford~B White.
\newblock Robust extended {K}alman filtering.
\newblock \emph{IEEE Transactions on Signal Processing}, 47\penalty0
  (9):\penalty0 2596--2599, 1999.

\bibitem[Evensen(2003)]{evensen2003ensemble}
Geir Evensen.
\newblock The ensemble {K}alman filter: Theoretical formulation and practical
  implementation.
\newblock \emph{Ocean Dynamics}, 53\penalty0 (4):\penalty0 343--367, 2003.

\bibitem[Evensen et~al.(2022)Evensen, Vossepoel, and van
  Leeuwen]{evensen2022data}
Geir Evensen, Femke~C Vossepoel, and Peter~Jan van Leeuwen.
\newblock \emph{Data assimilation fundamentals: A unified formulation of the
  state and parameter estimation problem}.
\newblock Springer Nature, 2022.

\bibitem[Fearnhead et~al.(2010)Fearnhead, Wyncoll, and
  Tawn]{fearnhead2010sequential}
Paul Fearnhead, David Wyncoll, and Jonathan Tawn.
\newblock A sequential smoothing algorithm with linear computational cost.
\newblock \emph{Biometrika}, 97\penalty0 (2):\penalty0 447--464, 2010.

\bibitem[Gelman and Meng(1998)]{gelman1998simulating}
Andrew Gelman and Xiao-Li Meng.
\newblock Simulating normalizing constants: From importance sampling to bridge
  sampling to path sampling.
\newblock \emph{Statistical Sciences}, pages 163--185, 1998.

\bibitem[Gerber and Chopin(2015)]{gerber2015sequential}
Mathieu Gerber and Nicolas Chopin.
\newblock Sequential quasi {M}onte {C}arlo.
\newblock \emph{Journal of the Royal Statistical Society Series B: Statistical
  Methodology}, 77\penalty0 (3):\penalty0 509--579, 2015.

\bibitem[Gilks and Berzuini(2001)]{gilks2001following}
Walter~R Gilks and Carlo Berzuini.
\newblock Following a moving target—{M}onte {C}arlo inference for dynamic
  {B}ayesian models.
\newblock \emph{Journal of the Royal Statistical Society: Series B (Statistical
  Methodology)}, 63\penalty0 (1):\penalty0 127--146, 2001.

\bibitem[Godsill et~al.(2004)Godsill, Doucet, and West]{godsill2004monte}
Simon~J Godsill, Arnaud Doucet, and Mike West.
\newblock {M}onte {C}arlo smoothing for nonlinear time series.
\newblock \emph{Journal of the American Statistical Association}, 99\penalty0
  (465):\penalty0 156--168, 2004.

\bibitem[Gordon et~al.(1993)Gordon, Salmond, and Smith]{gordon1993novel}
Neil~J Gordon, David~J Salmond, and Adrian~FM Smith.
\newblock Novel approach to nonlinear/non-{G}aussian {B}ayesian state
  estimation.
\newblock In \emph{IEEE Proceedings F (Radar and Signal Processing)}, volume
  140, pages 107--113. IET, 1993.

\bibitem[Gorodetsky et~al.(2019)Gorodetsky, Karaman, and
  Marzouk]{gorodetsky2019continuous}
Alex Gorodetsky, Sertac Karaman, and Youssef~M Marzouk.
\newblock A continuous analogue of the tensor-train decomposition.
\newblock \emph{Computer Methods in Applied Mechanics and Engineering},
  347:\penalty0 59--84, 2019.

\bibitem[Gottwald and Reich(2021)]{gottwald2021combining}
Georg~A Gottwald and Sebastian Reich.
\newblock Combining machine learning and data assimilation to forecast
  dynamical systems from noisy partial observations.
\newblock \emph{Chaos: An Interdisciplinary Journal of Nonlinear Science},
  31\penalty0 (10):\penalty0 101103, 2021.

\bibitem[Griebel and Harbrecht(2023)]{griebel2021analysis}
Michael Griebel and Helmut Harbrecht.
\newblock Analysis of tensor approximation schemes for continuous functions.
\newblock \emph{Foundations of Computational Mathematics}, 23:\penalty0
  219--240, 2023.

\bibitem[Hackbusch(2012)]{hackbusch2012tensor}
Wolfgang Hackbusch.
\newblock \emph{Tensor spaces and numerical tensor calculus}, volume~42.
\newblock Springer Science \& Business Media, 2012.

\bibitem[Herbst and Schorfheide(2019)]{herbst2019tempered}
Edward Herbst and Frank Schorfheide.
\newblock Tempered particle filtering.
\newblock \emph{Journal of Econometrics}, 210\penalty0 (1):\penalty0 26--44,
  2019.

\bibitem[Hoang et~al.(2021)Hoang, Krumscheid, Matthies, and
  Tempone]{hoang2021machine}
Truong-Vinh Hoang, Sebastian Krumscheid, Hermann~G Matthies, and Ra{\'u}l
  Tempone.
\newblock Machine learning-based conditional mean filter: a generalization of
  the ensemble {K}alman filter for nonlinear data assimilation.
\newblock \emph{arXiv preprint: 2106.07908}, 2021.

\bibitem[{K}alman(1960)]{kalman1960new}
Rudolph~Emil {K}alman.
\newblock A new approach to linear filtering and prediction problems.
\newblock 1960.

\bibitem[Kantas et~al.(2009)Kantas, Doucet, Singh, and
  Maciejowski]{kantas2009overview}
Nicholas Kantas, Arnaud Doucet, Sumeetpal~Sindhu Singh, and Jan~Marian
  Maciejowski.
\newblock An overview of sequential {M}onte {C}arlo methods for parameter
  estimation in general state-space models.
\newblock \emph{IFAC Proceedings Volumes}, 42\penalty0 (10):\penalty0 774--785,
  2009.

\bibitem[Kantas et~al.(2014)Kantas, Beskos, and Jasra]{kantas2014sequential}
Nikolas Kantas, Alexandros Beskos, and Ajay Jasra.
\newblock Sequential {M}onte {C}arlo methods for high-dimensional inverse
  problems: A case study for the navier--stokes equations.
\newblock \emph{SIAM/ASA Journal on Uncertainty Quantification}, 2\penalty0
  (1):\penalty0 464--489, 2014.

\bibitem[Kitagawa(1996)]{kitagawa1996monte}
Genshiro Kitagawa.
\newblock {M}onte {C}arlo filter and smoother for non-{G}aussian nonlinear
  state space models.
\newblock \emph{Journal of Computational and Graphical Statistics}, 5\penalty0
  (1):\penalty0 1--25, 1996.


\bibitem[{K}nothe et~al.(1957)]{knothe1957contributions}
Herbert {K}nothe et~al.
\newblock Contributions to the theory of convex bodies.
\newblock \emph{The Michigan Mathematical Journal}, 4\penalty0 (1):\penalty0
  39--52, 1957.

\bibitem[Liu et~al.(2001)Liu, Chen, and Logvinenko]{liu2001theoretical}
Jun~S Liu, Rong Chen, and Tanya Logvinenko.
\newblock A theoretical framework for sequential importance sampling with
  resampling.
\newblock In \emph{Sequential {M}onte {C}arlo methods in practice}, pages
  225--246. Springer, 2001.

\bibitem[Mallat(1989)]{mallat1989multiresolution}
Stephane~G Mallat.
\newblock Multiresolution approximations and wavelet orthonormal bases of
  {$L^2(\mathbb{R})$}.
\newblock \emph{Transactions of the American Mathematical Society},
  315\penalty0 (1):\penalty0 69--87, 1989.

\bibitem[Maskell and Gordon(2002)]{maskell2002tutorial}
Simon Maskell and Neil Gordon.
\newblock A tutorial on particle filters for on-line nonlinear/non-{G}aussian
  {B}ayesian tracking.
\newblock \emph{IEEE Target Tracking: Algorithms and Applications (Ref. No.
  2001/174)}, pages 2--1, 2002.


\bibitem[Morzfeld et~al.(2012)Morzfeld, Tu, Atkins, and
  Chorin]{morzfeld2012random}
Matthias Morzfeld, Xuemin Tu, Ethan Atkins, and Alexandre~J Chorin.
\newblock A random map implementation of implicit filters.
\newblock \emph{Journal of Computational Physics}, 231\penalty0 (4):\penalty0
  2049--2066, 2012.


\bibitem[Oseledets and Tyrtyshnikov(2010)]{oseledets2010tt}
Ivan Oseledets and Eugene Tyrtyshnikov.
\newblock {TT}-cross approximation for multidimensional arrays.
\newblock \emph{Linear Algebra and its Applications}, 432\penalty0
  (1):\penalty0 70--88, 2010.

\bibitem[Oseledets(2011)]{oseledets2011tensor}
Ivan Oseledets.
\newblock Tensor-train decomposition.
\newblock \emph{SIAM Journal on Scientific Computing}, 33\penalty0
  (5):\penalty0 2295--2317, 2011.

\bibitem[Peter~D(1985)]{peter1985kernel}
Hill Peter~D.
\newblock Kernel estimation of a distribution function.
\newblock \emph{Communications in Statistics-Theory and Methods}, 14\penalty0
  (3):\penalty0 605--620, 1985.

\bibitem[Pitt and Shephard(1999)]{pitt1999filtering}
Michael~K Pitt and Neil Shephard.
\newblock Filtering via simulation: Auxiliary particle filters.
\newblock \emph{Journal of the American Statistical Association}, 94\penalty0
  (446):\penalty0 590--599, 1999.

\bibitem[Pitt and Shephard(2001)]{pitt2001auxiliary}
Michael~K Pitt and Neil Shephard.
\newblock Auxiliary variable based particle filters.
\newblock \emph{Sequential {M}onte {C}arlo methods in practice}, pages
  273--293, 2001.

\bibitem[Reich(2013)]{reich2013nonparametric}
Sebastian Reich.
\newblock A nonparametric ensemble transform method for {B}ayesian inference.
\newblock \emph{SIAM Journal on Scientific Computing}, 35\penalty0
  (4):\penalty0 A2013--A2024, 2013.

\bibitem[Reich and Cotter(2015)]{reich2015probabilistic}
Sebastian Reich and Colin Cotter.
\newblock \emph{Probabilistic forecasting and {B}ayesian data assimilation}.
\newblock Cambridge University Press, 2015.

\bibitem[Rohrbach et~al.(2022)Rohrbach, Dolgov, Grasedyck, and
  Scheichl]{rohrbach2022rank}
Paul~B Rohrbach, Sergey Dolgov, Lars Grasedyck, and Robert Scheichl.
\newblock Rank bounds for approximating {G}aussian densities in the
  tensor-train format.
\newblock \emph{SIAM/ASA Journal on Uncertainty Quantification}, 10\penalty0
  (3):\penalty0 1191--1224, 2022.

\bibitem[{R}osenblatt(1952)]{rosenblatt1952remarks}
Murray {R}osenblatt.
\newblock Remarks on a multivariate transformation.
\newblock \emph{The Annals of Mathematical Statistics}, 23\penalty0
  (3):\penalty0 470--472, 1952.

\bibitem[S{\"a}rkk{\"a}(2013)]{sarkka2013bayesian}
Simo S{\"a}rkk{\"a}.
\newblock \emph{{B}ayesian filtering and smoothing}.
\newblock Number~3. Cambridge University Press, 2013.

\bibitem[S{\"a}rkk{\"a} et~al.(2015)S{\"a}rkk{\"a}, Hartikainen, Mbalawata, and
  Haario]{sarkka2015posterior}
Simo S{\"a}rkk{\"a}, Jouni Hartikainen, Isambi~S Mbalawata, and Heikki
  Haario.
\newblock Posterior inference on parameters of stochastic differential
  equations via non-linear {G}aussian filtering and adaptive mcmc.
\newblock \emph{Statistics and Computing}, 25\penalty0 (2):\penalty0 427--437,
  2015.

\bibitem[Shen et~al.(2011)Shen, Tang, and Wang]{shen2011spectral}
Jie Shen, Tao Tang, and Li-Lian Wang.
\newblock \emph{Spectral methods: algorithms, analysis and applications},
  volume~41.
\newblock Springer Science \& Business Media, 2011.

\bibitem[Silverman(1986)]{silverman1986density}
Bernard~W Silverman.
\newblock \emph{Density estimation for statistics and data analysis},
  volume~26.
\newblock CRC press, 1986.


\bibitem[Sloan and Wo{\'z}niakowski(1998)]{sloan1998quasi}
Ian~H Sloan and Henryk Wo{\'z}niakowski.
\newblock \emph{When are quasi-Monte Carlo algorithms efficient for high dimensional integrals?},
\newblock \emph{Journal of Complexity}, 14\penalty0 (1):\penalty0
1--33, 1998.

\bibitem[Snyder et~al.(2008)Snyder, Bengtsson, Bickel, and
  Anderson]{snyder2008obstacles}
Chris Snyder, Thomas Bengtsson, Peter Bickel, and Jeff Anderson.
\newblock Obstacles to high-dimensional particle filtering.
\newblock \emph{Monthly Weather Review}, 136\penalty0 (12):\penalty0
  4629--4640, 2008.

\bibitem[Spantini et~al.(2018)Spantini, Bigoni, and
  Marzouk]{spantini2018inference}
Alessio Spantini, Daniele Bigoni, and Youssef Marzouk.
\newblock Inference via low-dimensional couplings.
\newblock \emph{The Journal of Machine Learning Research}, 19\penalty0
  (1):\penalty0 2639--2709, 2018.

\bibitem[Spantini et~al.(2022)Spantini, Baptista, and
  Marzouk]{spantini2022coupling}
Alessio Spantini, Ricardo Baptista, and Youssef Marzouk.
\newblock Coupling techniques for nonlinear ensemble filtering.
\newblock \emph{SIAM Review}, 64\penalty0 (4):\penalty0 921--953, 2022.

\bibitem[Stuart(2010)]{IP:Stuart_2010}
Andrew~M Stuart.
\newblock Inverse problems: a {B}ayesian perspective.
\newblock \emph{Acta Numerica}, 19:\penalty0 451--559, 2010.

\bibitem[Wendland(2004)]{wendland2004scattered}
Holger Wendland.
\newblock \emph{Scattered data approximation}, volume~17.
\newblock Cambridge university press, 2004.

\bibitem[Williams and Rasmussen(2006)]{williams2006gaussian}
Christopher~KI Williams and Carl~Edward Rasmussen.
\newblock \emph{Gaussian processes for machine learning}, volume~2.
\newblock MIT press Cambridge, MA, 2006.

\bibitem[Xiu and Karniadakis(2002)]{xiu2002wiener}
Dongbin Xiu and George~Em Karniadakis.
\newblock The wiener--askey polynomial chaos for stochastic differential
  equations.
\newblock \emph{SIAM Journal on Scientific Computing}, 24\penalty0
  (2):\penalty0 619--644, 2002.

\bibitem[Yu et~al.(2006)Yu, Yang, and
Zhang]{yu2006class}
Jun Yu and Zhenlin Yang and Xibin Zhang.
\newblock A class of nonlinear stochastic volatility models and its implications for pricing currency options.
\newblock \emph{Computational Statistics \& Data Analysis}, 51\penalty0 (4):\penalty0 2218--2231, 2006.

\bibitem[Zahm et~al.(2022)Zahm, Cui, Law, Spantini, and
  Marzouk]{zahm2022certified}
Olivier Zahm, Tiangang Cui, Kody Law, Alessio Spantini, and Youssef Marzouk.
\newblock Certified dimension reduction in nonlinear Bayesian inverse problems.
\newblock \emph{Mathematics of Computation}, 91\penalty0 (336):\penalty0
  1789--1835, 2022.

\bibitem[Zanger et~al.(2024)Zanger, Cui, Schreiber, and
  Zahm]{zanger2024sequential}
Benjamin Zanger, Tiangang Cui, Martin Schreiber, and Olivier Zahm.
\newblock Sequential transport maps using SoS density estimation and $\alpha$-divergences.
\newblock \emph{arXiv preprint}, arXiv:2402.17943, 2024.

\bibitem[Zhang and Maxwell(2008)]{zhang2008box}
Xibin Zhang and Maxwell L King
\newblock Box-Cox stochastic volatility models with heavy-tails and correlated errors.
\newblock \emph{Journal of Empirical Finance}, 15\penalty0 (3):\penalty0 549--566, 2008.


\end{thebibliography}
\end{document}